\documentclass[11pt,a4paper,leqno]{article}

\usepackage{hyperref}
\usepackage{mathtools}

\usepackage[x11names]{xcolor}
\usepackage{url}
\usepackage[utf8]{inputenc}
\usepackage[T1]{fontenc}
\usepackage{mathtools,amssymb,mathrsfs}
\usepackage{lmodern}
\usepackage{amsmath}
\usepackage{amsthm}
\numberwithin{equation}{section}
\usepackage{fancybox}
\usepackage{soul}
\def\op{\mathrm{op}}
\def\pdo{\mathrm{pdo}}

\def\v_j{\underline{ {v}^{j} }}
\def\R_j{\mathrm{R}^{j}}

\def\l_j{\widetilde{ \lambda^{ j } }}

\def\dsp{\displaystyle}

\title{\textbf{On M\'etivier's Lax-Mizohata theorem and extensions to weak defects of hyperbolicity. \\ Part one.}}

\date{\today}

\author{Karim Ndoumajoud\thanks{Institut de Math\'ematiques de Jussieu - Paris Rive Gauche UMR CNRS 7586, Sorbonne Universit\'e} \and Benjamin Texier\thanks{Universit\'e Claude Bernard Lyon 1, Institut Camille Jordan UMR CNRS 5208}}

\newtheorem{defin}{Definition}[section]
\newtheorem{lem}[defin]{Lemma}
\newtheorem{prop}[defin]{Proposition}
\newtheorem{theo}{Theorem}
\newtheorem{cor}[defin]{Corollary}
\newtheorem{ass}[defin]{Assumption}
\newtheorem{rem}[defin]{Remark}

\def\R{\mathbb R}
\def\C{\mathbb C}
\def\S{\mathbb S}
\def\N{\mathbb N} 
\def\op{{\rm op}}
\def\e{\varepsilon}
\def\s{\sigma}
\newcommand{\be}{\begin{equation}}
\newcommand{\ee}{\end{equation}}
\newcommand{\ba}{\begin{aligned}}
\newcommand{\ea}{\end{aligned}}
\renewcommand{\d}{\partial}
\def\a{\alpha}
\def\b{\beta}
\def\g{\gamma}
\def\l{\lambda}
\def\t{\tau}
\def\tr{\mbox{tr}\,}

\def\op{{\rm op}}

\def\o{\omega}

\begin{document}

\maketitle

\mathtoolsset{showonlyrefs}

\begin{abstract}
We prove that, for first-order, fully nonlinear systems of partial differential equations, under an hypothesis of ellipticity for the principal symbol, the Cauchy problem has no solution within a range of Sobolev indices depending on the regularity of the initial datum. This gives a new and greatly detailed proof of a result of G. M\'etivier [{\it Remarks on the Cauchy problem}, 2005]. We then extend this result to systems experiencing a transition from hyperbolicity and ellipticity, in the spirit of recent work by N. Lerner, Y. Morimoto, and C.-J. Xu, [{\it Instability of the Cauchy-Kovalevskaya solution for a class of non-linear systems}, 2010], and N. Lerner, T. Nguyen and B. Texier [{\it The onset of instability in first-order systems}, 2018].
\end{abstract}

\section{Introduction}
We study the fully nonlinear, order-one Cauchy problem in $\R^d:$ 
\begin{equation} \label{equation reference} \left\{ 
\begin{aligned}
\d_t u + F(t,x,u,\d_x u) & = 0, \\ 
u_{ |t = 0 } & = u_{in} \in H^{ \sigma }( \mathbb{ R }^{ d } )
\end{aligned}
\right.
\end{equation}
where $ t \geq 0$ and $d \geq 1.$ The map
$ F: \R \times \R^d \times \R^N \times \R^{Nd} \to \R^N$
is smooth in a neighborhood of some $(0,x^0,u^0,v^0),$ where $N \geq 1.$ The principal symbol is defined as
\be \label{def:A}
A(t,x,u,v,\xi) := \xi \cdot \d_4 F(t,x,u,v) =  \sum_{1 \leq j \leq d} \xi_j \d_{4_j} F(t,x,u,v),
\ee
where $\d_{4_j} F$ refers to the derivative of $F$ with respect to the $j$-th component $v_j$ of its fourth argument $(v_1,\dots,v_d) \in (\R^N)^d.$ 

Under an assumption of initial ellipticity (strong defect of hyperbolicity, see Assumption \ref{ass:ell} below) for $A(0,\o^0),$ for some $\o^0 = (x^0,u^0,v^0,\xi^0) \in \R^d \times \R^N \times \R^{Nd} \times \S^{d-1},$ we prove in Theorem \ref{th:main} that if $s > 1 + d/2,$ then no local-in-time $H^s$ solution exists to the Cauchy problem \eqref{equation reference}, even on arbitrarily small time intervals, for some data with a Sobolev degree of regularity $\s \in [s, 2 s -1 - d/2).$

Under an assumption of a transition from hyperbolicity to ellipticity (weak defect of hyperbolicity; see Assumption \ref{ass:bif} below) for $A(0,\o^0),$ for some $\o^0,$ we prove in Theorem \ref{th:2} that if $s > 3/2 + 3/4 + d/2,$ then no local-in-time $H^s$ solution exists to the Cauchy problem \eqref{equation reference}, even on arbitrarily small time intervals, for some data with a Sobolev degree of regularity $\s \in [s, 2 s - 3/2 - d/2).$

Our results recover and extend an elliptic theorem of G. M\'etivier \cite{Metivier}. The analysis of weak defects of hyperbolicity was pioneered in the article of N. Lerner, Y. Morimoto and C.-J. Xu \cite{LMX}. These authors proved that for complex scalar equations exhibiting a weak defect of hyperbolicity, no solutions could exist above a certain level of regularity. Following \cite{LMX}, strong Hadamard instabilities were proved by N. Lerner, T. Nguyen and B. Texier \cite{LNT} for systems exhibiting weak defects of hyperbolicity. Our second result (Theorem \ref{th:2}) extends Theorem 1.6 from \cite{LNT}: while the result of \cite{LNT}, under an assumption analogous to Assumption \ref{ass:bif}, proved a lack of regularity of an hypothetical flow, we prove that such a flow does not exist. The functional spaces considered in \cite{LNT} are however larger: the lack of H\"older continuity is from $W^{1,\infty}$ to $H^m,$ any $m \in \R,$ while here we restrict to sufficiently small Sobolev spaces.

\section{Assumptions and results}

\subsection{Elliptic initial-value problems}

Consider the ellipticity assumption:

\begin{ass} \label{ass:ell} For some $(x^0, u^0, v^0, \xi^0) \in \R^d \times \R^N \times \R^{Nd} \times \S^{d-1},$ the spectrum of the principal symbol $A(0,x^0,u^0,v^0,\xi^0)$ defined in \eqref{def:A} is not entirely contained in $\R.$
\end{ass}

The above ellipticity assumption implies non-existence of Sobolev solutions to the initial-value problem \eqref{equation reference}, in the following sense: 
\begin{theo}\label{th:main}
Under Assumption {\rm \ref{ass:ell},} for any $s > d/2 + 1,$ any $s \leq \s < 2 s - 1 - d/2,$ for some $u_{in} \in H^\sigma(\R^d),$ there is no $T > 0$ and no ball $B_{x^0}$ centered at $x^0$ such that the initial-value problem \eqref{equation reference} has a solution in $C^0([0,T], H^s(B_{x^0})).$  
\end{theo}

 Theorem {\rm \ref{th:main}} recovers the ellipticity theorem of M\'etivier (Theorem 4.2 \cite{Metivier}), which extends the classical linear ellipticity theorems of Lax \cite{Lax} and Mizohata \cite{Mi}. The merit of our proof might be to clarify M\'etivier's and prepare for extensions to weak defects of hyperbolicity (see Theorem \ref{th:2} below), and to Gevrey spaces in the spirit of the work of B. Morisse \cite{Mo1,Mo2,Mo3}. The proof of Theorem \ref{th:main} is given in Section \ref{sec:proof:elliptic}.

\subsection{Transition to ellipticity} 

{\it Weak defects of hyperbolicity} were first explored for scalar equations in \cite{LMX}. The results of \cite{LMX} were extended to quasi-linear systems in \cite{LNT}. Roughly speaking, an assumption of a weak defect of hyperbolicity consists in relaxing Assumption \ref{ass:ell} into an assumption of initial hyperbolicity that is instantaneously lost: the principal symbol is hyperbolic at $t = 0$ but not for any $t > 0.$ This kind of assumption was given precise formulations for systems in \cite{LNT}, and shown to induce strong Hadamard instabilities in initial-value problems for quasi-linear systems.

Introduce $P$ the characteristic polynomial of $A$ evaluated along a solution $u$ of \eqref{equation reference}: 
\be \label{def:char}
 P(t,x,\xi,\l) = \mbox{det}\, (A(t,x, u(t,x), \d_x u(t,x),\xi) - \l {\rm Id}\big).
\ee
Via the equation \eqref{equation reference}, at $t = 0$ the characteristic polynomial $P$ and its partial derivatives  can all be expressed in terms of the initial datum. Denoting $v_{in} = (u_{in}, \d_x u_{in}),$ and $A_{in}(x,\xi) = A(0,x,v_{in}(x),\xi),$ $G_{in}(x) = F(0, x, v_{in}(x)),$ we see, in particular, that $P(0,x,\xi,\l)$ is equal to $\mbox{det}\, \big(A_{in}(x,\xi) - \l {\rm Id}\big),$ and $\d_t P(0,x,\xi,\l)$ to 
\be \label{dtP} \begin{aligned} \d_t P(0,x,\xi,\l)  = \mbox{det}'  \big( A_{in} (x,\xi) & - \l {\rm Id}\big) \\ & \cdot \Big( \d_1 A(0,x, v_{in}(x), \xi) + \d_3 A(0,x, v_{in}(x),\xi) G_{in}(x) \\ & \quad + \d_4 A(0,x, v_{in}(x),\xi) \cdot \d_x G_{in}(x) \Big) . \end{aligned}\ee 

The transition from hyperbolicity to ellipticity is formulated in \cite{LNT} in terms of the jet of $P$ (the first terms in its Taylor expansion) at $t = 0.$ These conditions thus bear on the datum, and involve the nonlinear terms in the equation.

We formulate here an assumption that is analogous to Hypothesis 1.5 in \cite{LNT}, and describes a time-differentiable branching of eigenvalues of $A$ from the hyperbolic to the elliptic region. 
Given $U$ an open set in $\R^d \times \S^{d-1},$ denote ${\mathcal S}_U$ the spectrum at $t = 0$ above $U:$ 
$$ {\mathcal S}_U := \big\{ (x,\xi,\l) \in U \times \C, \quad P(0,x,\xi,\l) = 0 \big\}.$$

\begin{ass} \label{ass:bif} For the principal symbol $A$ \eqref{def:A} we assume the following conditions:
\begin{itemize}
\item[{\rm (i)}] {\rm initial hyperbolicity:} for some open set $U,$ we have $\o = (x,\xi,\l) \in {\mathcal S}_U$ only if $\l \in \R.$
\item[{\rm (ii)}] {\rm Smooth diagonalizability:} for some $(x^0,\xi^0) \in U,$ the matrix $A(0,x,\xi)$ is smoothly diagonalizable for all $(x,\xi)$ near $(x^0,\xi^0)$ in $U.$ 
\item[{\rm (iii)}] {\rm Existence of double eigenvalues:} for some $\o^0 = (x^0,\xi^0,\l^0) \in {\mathcal S}_U,$ we have $\d_\l P(0,\o^0) = 0.$ 
\item[{\rm (iv)}] {\rm $C^1$ bifurcations of the spectrum away from $\R:$} given any $\l \in \R$ such that $\o = (x^0,\xi^0,\l) \in {\mathcal S}_U$ and $\d_\l P(0,\o) = 0:$ we have $\d_t P(0,\o) = 0$ and 
 \be \label{ineg:bif} \big (\d_{t\l}^2 P(0, \o) \big)^2 < (\d_t^2 P \d_\l^2 P)(0,\o),\ee
  and $\d_t P(0,\o') = 0,$ $\d_\l P(0, \o') = 0$ for all $\o'$ near $\o$ in ${\mathcal S}_U.$   
 \end{itemize}
\end{ass} 

In condition (ii) above, by ``smooth'' we mean ``as smooth as $A_{|t =0}$'', which by assumption on $F$ means ``as smooth as the datum''.  

We describe in Section \ref{sec:the:spectral:picture} how Assumption \ref{ass:bif} paints a picture of the spectrum of $A$ evaluated along a putative solution $u$ of \eqref{equation reference} for very small time: at $t$ goes from $0$ to positive values, some real eigenvalues of $A$ stay real, while pairs of eigenvalues branch out of the real axis into the rest of the complex plane, in a $C^1$ fashion, so that the real parts of the branching eigenvalues is $O(t).$

\begin{theo} \label{th:2} 
Under Assumption {\rm \ref{ass:bif},} for any $s > 3/2 + 3/4 + d/2,$ any $\s$ such that $3 + d/2 \leq \s < 2 s - 3/2 - d/2,$ for some $u_{in} \in H^\sigma(\R^d),$ there is no $T > 0$ and no ball $B_{x^0}$ centered at $x^0$ such that the initial-value problem \eqref{equation reference} has a solution in $C^0([0,T], H^s(B_{x^0})).$ 
\end{theo}

Examples of systems and data satisfying Assumption \ref{ass:bif} are given in \cite{LNT}. Those include in particular systems of coupled Burgers equations.

In Assumption \ref{ass:bif} the instability is weaker than in Assumption \ref{ass:ell}. As a consequence, the instability is slower to develop, and the condition on $\s - s$ in Theorem \ref{th:2} is stronger than the one in Theorem \ref{th:main}. 

Condition $\s > 3 + d/2$ is required in order to make sense of two time derivatives (corresponding, via the equation, to three spatial derivatives) at $t = 0,$ which we need in Assumption \ref{ass:bif}(iv). For the proof to go through we only need $s > 2 + d/2,$ but then the conditions bearing on $\s$ impose that $s$ is actually greater than $3/2 + 3/4 + d/2.$  

The proof of Theorem \ref{th:2} is given in Section \ref{sec:weak}.

\subsection{Remarks} \label{sec:remarks}

In the companion paper \cite{NT2}, we get rid of the upper bound on $\sigma,$ that is prove non-existence of Sobolev solutions above a certain level of regularity {\it no matter how regular the datum is,} under an extra assumption of semi-simplicity and constant multiplicity for the purely imaginary eigenvalues of the principal symbol. We do so in the elliptic case and in the case of a transition to ellipticity. This extends Theorem 4.4 in \cite{Metivier}.

  A lengthy appendix (pages \pageref{proof:end} to \pageref{page:endofapp}) gives some standard and  not-so-standard results pertaining to pseudo- and para-differential operators. In the course of the analysis, we indeed need an estimate for the action on Sobolev maps of the difference between a pseudo-differential operator in semiclassical quantization and its para-differential counterpart. This issue was examined earlier in \cite{Lannes} and \cite{em3}. We correct here small errors from \cite{Lannes} and \cite{em3}; see Proposition \ref{prop:lannes}. We also show how G\r{a}rding's inequality translates into rates of growth for the flows of para-differential operators of order zero, in semiclassical quantization, for symbols with H\"older spatial regularity (see Proposition \ref{lem:rates}).

Note finally that the condition $s > 1 + d/2$ is critical, in the following sense: in the favorable case of symmetric hyperbolic systems, well-posedness is known only under that condition on the regularity of the datum; moreover, for the related (but nonlocal, hence distinct) incompressible Euler equations in two dimensions, well-posedness fails to hold in $H^2$ \cite{BL}. Here the assumption $s > 1 + d/2$ means that the posited solution $u$ of \eqref{equation reference} (the existence of which we endeavor to disprove) does not have a high level of pointwise regularity. Since $F$ depends on $(u, \d_x u),$ the symbols are only H\"older continuous in space. In particular, this forces us to rely on G\r{a}rding's inequality in order to derive bounds for the flows of relevant order-zero para-differential operators, as opposed to the high-regularity approach of \cite{Duh,em4}.

\section{Ellipticity: proof of Theorem \ref{th:main}} \label{sec:proof:elliptic}

\label{proof:begin} We work by contradiction and {\it assume} the existence of $T > 0,$ a ball $B_{x_0}$ centered at $x_0$ and $u \in C^0([0,T],H^s(B_{x_0}))$ which solves the initial-value problem \eqref{equation reference}. The contradiction will eventually come from the high but finite regularity of the datum $u_{in}.$

Our assumption puts forward a distinguished point $(x^0,u^0,v^0,\xi^0) \in \R^d \times \R^N \times \R^{Nd} \times \S^{d-1}.$ We will choose the datum $u_{in}$ (see Section \ref{sec:datum}), so that $u_{in}(x^0) \simeq u^0,$ and $\d_x u_{in}(x^0) = v^0.$ 

We identify the solution $u$ with its extension to the whole space $\R^d$ via a linear bounded extension operator $H^s(B_{x_0}) \to H^s(\R^d).$

\subsection{Hyperbolic rescaling} \label{sec:hyp}

Define the spatial rescaling map: 
\be \label{def:hj}
h_j: \quad f \in H^s \to (x \to f(2^{-j} x)) \in H^s,
\ee
for some $j$ that will eventually be chosen large, and let
\be \label{def:tildeu} \tilde u(t,x) := u( 2^{-j} t,x), \qquad u_{h} := h_j \tilde u,\ee
where $u$ is the posited solution. Then, $u_h$ belongs to $C^0([0, 2^{j}  T], H^s(\R^d)),$ and solves
the initial-value problem
\be \label{ivp:rescaled} \left\{\begin{aligned}
 \d_t u_h + 2^{-j} F(2^{-j} t, 2^{-j} x, u_h , 2^j \d_x u_h) & = 0,\\
 u_h(0,x) = u_{in}(2^{-j} x).
 \end{aligned}\right.
\ee  

\subsection{Paralinearization} \label{sec:para}

Our second step is to para-linearize the system. Since $u_h$ belongs to $H^s$ for $t \in [0, 2^j T],$ we have (see for instance Theorem 5.2.4 in \cite{Metivierbis}):
\be \label{para}
 F(2^{-j} t, 2^{-j} x, u_h, 2^j \d_x u_h) = T_{\d_3 F(u_h ,  2^{j} \d_x u_h)} u_h + T_{\d_4 F(u_h , 2^j \d_x u_h)} 2^j \d_x u_h + R_j^{para},
\ee
where $T_a b$ refers to the para-product of $b$ by $a,$ and $\d_3 F$ refers to the partial derivative of $F$ with respect to its third ($u$) argument. The remainder satisfies 
\be \label{R:para}
 R_j^{para} \in L^\infty([0, 2^j  T], H^{2 (s -1) - d/2}(\R^d)).
\ee
The proof of Theorem 5.2.4 in \cite{Metivierbis} shows moreover that
\be \label{R:est}
 \| R_j^{para}\|_{H^{2(s-1) - d/2}(\R^d)} \lesssim C_F(\| u_h, 2^j \d_x u_h \|_{L^\infty(\R^d)}) \| (u_h, 2^j \d_x u_h) \|_{H^{s-1}(\R^d)},
\ee
where $C_F: \R_+ \to \R_+$ is nondecreasing and depends only on $F.$ In \eqref{para}, we omitted the arguments $(2^{-j} t, 2^{-j} x)$ of $\d_3 F$ and $\d_4 F$ in the right-hand side.

\subsection{Semiclassical quantization} \label{sec:semicl}

Given a symbol $a,$ our preferred notation for the para-differential operator with symbol $a$ is $\op(a):$ 
\be \label{notation:op} T_a  = \op(a),\ee
where the spatial convolution with the partial inverse Fourier transform of some admissible cut-off function is implicit. For details on our notation, and classical results pertaining to para-differential operators, see Appendix \ref{sec:symb}.

By definition, the para-differential operator with symbol $a$ in $2^{-j}$-semiclassical quantization is 
$$ \op_j(a) := h_j^{-1} \op(h_j a) h_j,$$
where $h_j$ is defined in \eqref{def:hj}.
Classical results in semiclassical quantization are recalled in Appendix \ref{sec:symb}. 
We remark that, using notation \eqref{notation:op}, we can write \eqref{para} in the form
\be \label{para2}
 2^{-j} h_j F(\tilde u, \d_x \tilde u) = 2^{-j} \op\big( h_j \d_3 F(\tilde u,\d_x \tilde u) \big) h_j \tilde u + \op\big( h_j \d_4 F(\tilde u, \d_x \tilde u) \big) \d_x (h_j \tilde u) + 2^{-j} R^{para}_j.
 \ee
That is, 
\be \label{para5} \begin{aligned}
2^{-j} F(\tilde u,\d_x \tilde u) & = h_j^{-1} \op\Big( 2^{-j} h_j \d_3 F(\tilde u, \d_x \tilde u) \Big) h_j \tilde u +  h_j^{-1} \op\Big( h_j  \d_4 F(\tilde u, \d_x \tilde u) \cdot i \xi \Big) h_j \tilde u \\ & + 2^{-j} h_j^{-1} R^{para}_j. \end{aligned} 
 \ee
Thus $\tilde u$ solves the initial-value problem 
\be \label{ivp:tildeu} \left\{\begin{aligned}
 \d_t \tilde u + \op_j(\tilde A) \tilde u & = \tilde R^{para}_j,
 \\ 
 \tilde u(0) & = u_{in}, \end{aligned}\right.
 \ee 
where $\tilde R^{para}_j$ is 
\be \label{def:tildeR}
 \tilde R^{para}_j := - 2^{-j} h_j^{-1} R^{para}_j,
\ee
 and the symbol $\tilde A$ is defined by 
\be \label{def:tildeA} \begin{aligned} 
\tilde A(t, x, \tilde u, \d_x \tilde u, \xi) :=  2^{-j}  \d_3 F(2^{-j} t, x, \tilde u, \d_x \tilde u) + \d_4 F(2^{-j} t, x, \tilde u, \d_x \tilde u) \cdot i \xi,
\end{aligned}
\ee
where $\tilde u$ is evaluated at $(t,x).$ Recall that $j$ will be chosen to be large. Thus the symbol $\tilde A$ appears as a small perturbation of the principal symbol $i A,$ with $A$ defined in \eqref{def:A}.

\subsection{Localization near the distinguished point in the cotangent space} \label{sec:loc:psi}

We use a Littlewood-Paley dyadic decomposition $(\phi_j)_{j \geq 0}.$ That is, a smooth frequency truncation function $\phi_0$ is given, with $0 \leq \phi_0 \leq 1,$ and $\phi_0(\xi) \equiv 1$ for $|\xi| \leq \e_1$ and $\phi_0(\xi) \equiv 0$ for $|\xi| \geq \e_2,$ for some real numbers $\e_1, \e_2$ satisfying $0 < \e_1 < \e_2 < 1 < 2\e_1.$ We let \be \label{def:phij} \phi_j(\xi) := \phi_0(2^{-j} \xi) - \phi_0(2^{-j + 1} \xi), \qquad \mbox{for $j \geq 1.$}\ee
In particular, the cut-off $\phi_1$ is supported in the neighborhood $\{ \e_1 \leq |\xi| \leq 2 \e_2 \}$ of the unit disk; the cut-off $\phi_k$ is supported in the annulus $\{ 2^k \e_1 \leq |\xi| \leq 2^{1 + k} \e_2 \}.$  We have the decomposition   
\be \label{LP} f = \sum_{j \geq 0} \phi_j(D) f,
\ee 
where the operator $\phi_j(D)$ is the Fourier multiplier with symbol $\phi_j(\xi).$

We localize in $(x,\xi)$ near $(x^0,\xi^0),$ within the annulus $|\xi| \simeq 2^{-j},$ for some integer $j$ which will eventually be chosen large. In this view, we let $\psi$ be a smooth space-frequency cut-off supported in $(x^0,\xi^0) \in \R^d_x \times \R^d_\xi:$ for some $\delta > 0,$ we have
 \be \label{def:psi}
 \psi(x,\xi) = \left\{\begin{aligned} 1, & \quad |x - x^0| + |\xi - \xi^0| \leq \delta/2, \\
 0, & \quad |x - x^0| + |\xi - \xi^0| \geq \delta.
 \end{aligned}\right.
 \ee
  We may choose $\psi$ to be a tensor product $\psi(x,\xi) = \psi_1(x) \psi_2(\xi).$ 
We also let 
\be \label{def:v}
v := \op_j(\psi) \tilde u
\ee
where $\tilde u$ is defined in \eqref{def:tildeu} and was found in Section \ref{sec:semicl} to solve the initial-value problem \eqref{ivp:tildeu}.
The initial-value problem in $v$ is
\be \label{ivp:v}\left\{\begin{aligned}
 \d_t v + \op_j(\tilde A) v & = f , \\ 
 v(0) & = \op_j(\psi) u_{in},
\end{aligned}\right.\ee
where
\be \label{def:f}
 f = - [\op_j(\psi), \op_j(\tilde A)] \tilde u + \op_j(\psi) \tilde R^{para}_j.
\ee
We observe that the operator in \eqref{ivp:v} is further localized, up to a remainder. Indeed, we have
 $$ \op_j(\tilde A) v = \op_j(\psi^\sharp \tilde A) v - f_1,$$
 where
 \be \label{def:f1} f_1 := \op_j(\tilde A) \big(\op_j(\psi^\sharp) \op_j(\psi) - \op_j(\psi)\big) \tilde u + (\op_j(\psi^\sharp \tilde A) - \op_j(\tilde A) \op_j(\psi^\sharp)\big) v,\ee
 and $f_1$ will be seen to be appropriately sense. 

Above, we used notation $\chi \prec \chi^\sharp,$ borrowed from \cite{LNT}: 
 \be \label{def:prec}
 \chi \prec \chi^\sharp \quad \mbox{means} \quad (1 - \chi^\sharp) \chi \equiv 0,
 \ee
 with $\chi, \chi^\sharp \in C^\infty_c(\R^{2d};[0,1]).$ Whenever we use \eqref{def:prec}, it is implicit that the support of the extension $\chi^\sharp$ is {\it not much bigger} than the support of $\chi.$ The context usually makes clear what is meant by ``not much bigger''.

We let 
\be \label{def:g}
 g := f + f_1,
\ee
and the initial-value problem in $v$ now appears as
\be \label{ivp:v:2}
\left\{\begin{aligned}
 \d_t v + \op_j(M) v & = g, \\ 
 v(0) & = \op_j(\psi) u_{in}.
\end{aligned}\right.\ee
The symbol $M$ is, up to a small perturbation, a localized version of the principal symbol $A$ defined in \eqref{def:A}:
\be \label{def:M} \begin{aligned}
 M(t,x,\xi) = \psi^\sharp(x,\xi) \Big( & A(2^{-j} t,x, \tilde u, \d_x \tilde u, i  \xi) + 2^{-j} \d_3 F(2^{-j} t,x, \tilde u, \d_x \tilde u) \Big),
\end{aligned} \ee
where $\tilde u$ is evaluated at $(t,x).$ By regularity of $F,$ the posited regularity of $u,$ and Lemma \ref{lem:symbols}, we have $M(t) \in C^{[\theta],\theta - [\theta]} S^0,$ for all $t$ in the interval of definition of $\tilde u,$ where 
\be \label{def:theta}
 \theta := s - 1 - d/2.
 \ee

\subsection{Spectral decomposition} \label{sec:spectraldec}

 Assumption \ref{ass:ell} states that the spectrum of $A(0, x^0, u^0, v^0, \xi^0)$ is not entirely contained in $\R.$ 

 The datum will be chosen to satisfy $v(0,x^0) \simeq u^0,$ $\d_x v(0,x^0) \simeq v^0.$ (A precise choice is made in Section \ref{sec:datum}). By continuity of the spectrum (a consequence of Rouch\'e's theorem, see for instance Proposition 1.1 in \cite{Rouche}), if 
 \be \label{cond:spectrum} \mbox{$j$ is large enough, \quad and \quad $|v(0,x^0) - u^0| + |\d_x v(0,x^0) - v^0|$ is small enough,}
\ee
then the spectrum of $M(0,x^0,\xi^0)$ is not entirely contained in $i \R.$ 

Note that $M$ is defined in \eqref{def:M} in terms of $A$ evaluated at $i \xi,$ so that it is indeed the real parts of the eigenvalues that matter in this ill-posedness proof. Since $M$ appears in the left-hand side of \eqref{ivp:v:2}, an eigenvalue with a negative real part is susceptible to generate growth.

We denote 
$$ \g := \Re e \, \l_0 < \Re e \, \l_1 < \dots < \Re e \, \l_{p_0}$$
the spectrum of $M(0,x^0,\xi^0).$ The most unstable eigenvalue is $\l_0.$ By reality of the coefficients of $A,$ the real part of the spectrum of $M$ is symmetric with respect to 0. Thus we have $\g = \Re e \, \l_0 < 0.$ 

The elliptic eigenvalue $\l_0$ might correspond to a coalescing point $(0,x^0,\xi^0,\l_0)$ in the spectrum.
We denote $\mu_1, \dots, \mu_p$ the eigenvalues of $M$ which coalesce at $z^0 = (0,x^0,\xi^0),$ with value $\l_0$ at $z^0:$ 
\be \label{def:mu}
\mu_j(z^0) = \l_0, \qquad 1 \leq j \leq p.
\ee
We say that the eigenvalues $\mu_1,\dots,\mu_p$ are {\it elliptic}, and call {\it hyperbolic} the other eigenvalues of $M.$ The hyperbolic eigenvalues are thus those which take values $\Re e \, \l_1, \dots, \Re e \, \l_{p_0}$ at $z^0.$ We denote $\mu_{p+1},\dots,\mu_N$ the other branches of eigenvalues of $M.$ It may be that $\Re e \, \mu_j \neq 0$ for some $j \in [p+1,N],$ even though we call $\mu_j$ hyperbolic, but it will be a convenient way of referring to this part of the spectrum. 

If $j \leq p < j',$ then 
\be \label{spectral:sep} 
\Re e \, \mu_j < \Re e \, \mu_{j'},
\ee 
uniformly in a small neighborhood of $z^0,$ by continuity of the eigenvalues. We may and do assume that the support of $\psi^\sharp$ is small enough so that \eqref{spectral:sep} holds over the support of $\psi^\sharp.$

The eigenvalues $\mu_j$ may not be smooth, but satisfy the H\"older estimates (see for instance Proposition 3.1 in \cite{Rouche}, or \cite{Kato}):
\be \label{puiseux}  
|\mu_j(z) - \mu_j(z^0)| \leq c |z - z_0|^{1/p'},
\ee
for some $c > 0$ and some $p' \geq p.$

 We denote $E = E(t,x,\xi)$ the direct sum of the generalized eigenspaces associated with the eigenvalues $\mu_1, \dots, \mu_p$ which coalesce at $z^0 = (0,x^0,\xi^0)$ with value $\l_0,$ and $H = H(t,x,\xi)$ the direct sum of the generalized eigenspaces associated with the other eigenvalues, so that
\be \label{space:dec} \C^N = E(t,x,\xi) + H(t,x,\xi), \qquad \mbox{for all $(t,x,\xi)$ near $(0,x^0,\xi^0).$}\ee 

\subsection{Projections} \label{sec:proj} 
We denote $P_H$ the projector onto $H$ parallel to $E,$ and $P_E$ the projector onto $E$ parallel to $H,$ so that 
\be \label{id:proj}
{\rm Id} = P_E + P_H.
\ee
We have integral representations for $P_E$ and $P_H$ in terms of $M$ (see for instance Proposition 2.1 in \cite{Rouche}): 
\be \label{proj:int:rep}
 P_\star = \frac{1}{2 i \pi}\int_{{\mathcal C}_\star} (z - M)^{-1} \, dz, \qquad \star \in \{H, E\},
\ee
where ${\mathcal C}_H$ is a contour enclosing all the hyperbolic eigenvalues and only those, and ${\mathcal C}_E$ is a contour enclosing the elliptic eigenvalues. By spectral separations \eqref{spectral:sep}, we can choose non-overlapping contours ${\mathcal C}_E$ and ${\mathcal C}_H.$ The representations \eqref{proj:int:rep} then ensure that $P_E$ and $P_H$ are as smooth as $M,$ so that $P_E, P_H \in C^{0,\theta} S^0,$ pointwise in time. 

The symbol $M,$ and the projectors $P_E$ and $P_H$ depend on $\d_x u,$ which a priori is only continuous in time. This will be an issue in view of the change of variable \eqref{def:vEH}. Hence we regularize the projectors. Let $P_H^\e$ be $P_H$ evaluated at $k_\e \star (u, \d_x u),$ where $k_\e$ is a smoothing kernel, just like in \eqref{reg:2}. The small parameter $\e$ will be chosen in terms of $j$ in the proof of Lemma \ref{lem:Gstar}.

 The associated components of the unknown $v$ are 
\be \label{def:vEH}
v_H := \op_j(\psi^\flat P_H^\e) v, \qquad v_E = \op_j(\psi^\flat P_E^\e) v,
\ee
with 
\be \label{def:psiflat} \psi \prec \psi \prec \psi^\sharp,\ee
where we used notation \eqref{def:prec}.
We denote
\be \label{def:MHE}
M_H := P_H M, \qquad M_E := P_E M,
\ee
so that
\be \label{prop:MHE} M_H \equiv M_H P_H, \qquad M_E \equiv M_E P_E,
\ee
and now derive the equations in $v_H$ and $v_E.$ We have
$$ \begin{aligned} \d_t v_H & = \op_j(\psi \d_t P^\e_H) v + \op_j(\psi P^\e_H) \d_t v  \\ 
& = \op_j(\psi \d_t P^\e_H) v - \op_j(\psi^\flat P_H^\e) \op_j(M) v +  \op_j(\psi^\flat P_H^\e) g.
\end{aligned}$$
Next, by composition of para-differential operators in semiclassical quantization, and approximation of $P_H$ by $P_H^\e,$   
$$ \op_j(\psi^\flat P_H^\e) \op_j( M) = \op_j(\psi^\flat M_H) + \mbox{a smaller term.}$$ 
Precisely, using the composition result described in Section \ref{sec:composition}, 
$$ \op_j(\psi^\flat P_H^\e) \op_j(M) = \op_j(\psi^\flat M_H) + 2^{-j \theta} R_{\theta}(\psi^\flat P_H^\e, M) + \op_j(\psi^\flat (P_H^\e - P_H)) \op_j(M),$$
where we recall \eqref{def:theta} $\theta = s - 1 - d/2.$ 
Besides, by \eqref{prop:MHE}, 
$$ \begin{aligned} \op_j(\psi^\flat M_H) 
& = \op_j(M_H) \op_j(\psi^\flat P_H^\e) + 2^{-j \theta} R_{\theta}(M_H, \psi^\flat P_H) + \op_j(M_H) \op_j(\psi^\flat (P_H - P_H^\e)).
\end{aligned}$$
We denote $R_H$ the sum of the remainders in the last two equations, as we let 
\be \label{def:RH} \begin{aligned}
R_H & := 2^{-j \theta} \big( R_{\theta}(\psi^\flat P_H^\e, M) + R_{\theta}(M_H, \psi^\flat P_H) \big) \\ & \quad + \op_j(M_H) \op_j(\psi^\flat (P_H - P_H^\e)) + + \op_j(\psi^\flat (P_H^\e - P_H)) \op_j(M) \end{aligned}.\ee
Substituting $E$ for $H,$ we obtain a system in $(v_H, v_E):$ 
\be \label{eq:HE} \left\{
\begin{aligned} \d_t v_H + \op_j(M_H) v_H & =  g_H,
\\
\d_t v_E + \op_j(M_E) v_E & = g_E,
\end{aligned}\right.
\ee
with notation
\be \label{def:gstar}
g_\star :=  \op_j(\psi^\flat P_\star^\e) g - ( R_\star + \op_j(\psi^\flat \d_t P_\star^\e)) v, \qquad \star \in \{H, E\}.
\ee

In the next Section, we choose the datum $u_{in}$ and examine what our choice entails for the data $v_E(0)$ and $v_H(0).$

\subsection{The initial datum} \label{sec:datum} 

We carefully choose the polarization of the datum, and its frequency localization. This Section ends with Corollary \ref{cor:datum}, where a lower bound for an $L^2$ norm of $v_E(0)$ is proved. 

\subsubsection{Polarization} \label{sec:pola}

Here we choose the polarization of the initial datum $u_{in},$ meaning its direction in the state space $\R^N.$ The most obvious choice $u_{in}(x^0) = u^0$  is not always appropriate. Indeed, it could be that the elliptic component of $u^0$ (in the sense of \eqref{space:dec}) is zero. For instance, with $d = 2$ and $N = 3$ and  
$$ A(u, i \xi) = \left(\begin{array}{ccc} 0 & i u_3 \xi_1 & 0 \\ - i u_3 \xi_1 & 0 & 0\\ 0 & 0 & i u_1 \xi_2 \end{array}\right),$$
then $A$ is elliptic at $(u^0,\xi^0)$ with $u^0 = (0,0,1)$ and $\xi^0 = (1,0).$ The elliptic subspace $E,$ however, satisfies $E(u^0,\xi^0) = \R^2 \times \{ 0 \},$ so that $P_E(u^0,\xi^0) u^0 = 0.$   
Thus choosing $u_{in}(x^0) = u^0$ would be an issue in the perspective of proving a bound from below for the elliptic component $v_E(0)$ of the datum. The following Lemma shows that an appropriate choice for $u_{in}(x^0)$ is always possible:

\begin{lem} \label{lem:pola} Under Assumption {\rm \ref{ass:ell},} for some $u^1 \in \R^d,$ for $u_{in}$ such that $u_{in}(x^0) = u^1$ and $\d_x u_{in}(x^0) = v^0,$ the spectrum of $M(0, x^0, \xi^0)$ is not included in $i \R$ and we have $P_{E}(0,x^0,\xi^0) u^1 \neq 0,$ where $P_E$ is the projection onto the elliptic space and parallel to the hyperbolic space \eqref{space:dec}. 
\end{lem}

Recall that the symbol $M$ is defined in \eqref{def:M}. It depends on $(t,x)$ through $(u,\d_x u),$ where $u$ is our posited solution. In particular, the symbol $M$ depends on the datum $u_{in}.$ Accordingly, the symbol $P_E$ depends on the datum. 

\begin{proof} Let $(x^0, u^0, v^0, \xi^0)$ as in Assumption \ref{ass:ell}. Let $M_0$ be the matrix-valued symbol \eqref{def:M} at $(t,x,\xi) = (0,x^0,\xi^0),$ with the choice $u_{in}(x^0) = u^0,$ $\d_x u_{in}(x^0) = v^0.$ Let $E_0$ be the associated elliptic space, as defined in Section \eqref{spectral:sep}. If $P_{E_0} u^0 \neq 0,$ then we may choose $u^1 = u^0.$ 

Otherwise, $P_{E_0} u^0 = 0.$ Let $\vec e$ be a unitary elliptic eigenvector of $M_0,$ and posit $$u_{in}(x^0) = u^1 := u^0 + \a \vec e,$$
 for some small $\a > 0,$ and $\d_x u_{in}(x^0) = v^0.$ If $\a$ is small enough, by continuity of the spectrum, the eigenvalues of $M(0, x^0, \xi^0)$ are not all purely imaginary, so that the space $E(0, x^0, \xi^0)$ is not reduced to $\{ 0 \}.$ 

By definition of $M$ in \eqref{def:M}, the projector $P_E$ takes the form
$$ P_{E(0,x,\xi)} = \underline P_E(x,u(0,x), \d_x u(0,x), \xi).$$
 Since the elliptic eigenvalues are separated from the hyperbolic eigenvalues, the map $\underline P$ is smooth in its arguments (see for instance Proposition 2.1 in \cite{Rouche}). We may further choose $\vec e$ so that
\be \label{cond:e} \vec e + \d_u \underline P_E(0, u^0, v^0, \xi^0) \cdot u^0 \neq 0.\ee 
Indeed, if $\d_u \underline P_E(0, u^0, v^0, \xi^0) \cdot u^0 \neq 0$ does not generates the whole elliptic eigenspace $E(0,x^0,\xi^0),$ then we choose for $\vec e$ to be any other (unitary) elliptic eigenvector. If $\d_u \underline P(0, u^0, v^0, \xi^0) \cdot u^0$ happens to generate $E(0,x^0,\xi^0)$ in its entirety, then we let $\vec e = \l  \d_u \underline P_E(0, u^0, v^0, \xi^0) \cdot u^0 \neq 0,$ and we may choose $\l \neq -1$ so that $\vec e$ is unitary. 

Then,
$$ \begin{aligned} P_{E(0,x^0,\xi^0)} u^1 & = P_{E_0}(u^0 + \a \vec e) + (P_{E(0,x^0,\xi^0)} - P_{E_0}) (u^0 + \a \vec e) \\ & = \a \vec e + (P_{E(0,x^0,\xi^0)} - P_{E_0}) (u^0 + \a \vec e) \\ & = \a\big( \vec e + \d_u \underline P_E(x^0, u^0,\xi^0) \cdot u^0 \big) + O(\a^2), \end{aligned}$$
which implies $P_{E(0,x^0,\xi^0)} u^1 \neq 0$ by \eqref{cond:e}, if $\a$ is small enough. 
\end{proof}

\subsubsection{Oscillations}  \label{sec:osc}

 Let $w$ be a scalar map in the Schwartz class ${\mathcal S}(\R^d_x),$ such that $\hat w \in C^\infty_c(\R^d_\xi),$ and $w(x^0) \neq 0.$ Associated with $w,$ we let 
 \be \label{def:vin}
 w_{in}(x) := \sum_{j \geq 0} a_j e^{i x \cdot 2^j \xi^0} w,
 \ee
 where $a_j > 0$ is a $j$-dependent amplitude. 
 
 \begin{lem} \label{lem:vin} With the choice 
 \be \label{def:aj}
 a_j := 2^{-j \s} (1 + j)^{-1}, \qquad j > 0, 
 \ee
 we have 
$$ c_1 a_j \leq \| \phi_{j+1}(D) v_{in} \|_{L^2} \leq c_2 a_j,$$
for $j \geq j^0 > 0,$ for some $j^0$ and some $c_1 > 0$ and $c_2 > 0$ which depend only $w$ and the radii $\e_1$ and $\e_2$ that appear in the Littlewood-Paley decomposition $(\phi_j)_{j \geq 0}$ introduced in Section {\rm \ref{sec:loc:psi}.}
In particular, the scalar map $w_{in}$ belongs to $H^\s(\R^d)   \setminus \big( \cup_{\e > 0} H^{\s + \e}(\R^d) \big).$ Moreover, if $|a_0|$ is large enough, then $w_{in}(x^0) \neq 0.$ 
\end{lem}

\begin{proof} Given $j_0 > 0,$ we examine $\phi_{j_0+1}(D_x) v_{in}.$ We note that $\phi_{j_0+1} \equiv \phi_1(2^{- j_0} \cdot).$ The term of index $j$ in the sum is 
 $$ \phi_{j_0 + 1}( D) \big( a_j e^{i x \cdot 2^j \xi^0} w \big) = a_j e^{i x \cdot 2^j \xi^0} {\mathcal F}^{-1} \Big( \phi_1(2^{-j_0 + j} \xi_0 + 2^{-j_0} \cdot) \hat w\, \Big).$$
 Let $R > 0$ such that the support of $\hat w$ is supported in $\{ |\xi| \leq R\}.$ Then, we observe that for $\phi_1(2^{-j_0 + j} \xi_0 + 2^{-j_0} \xi) \hat w(\xi)$ to be different from zero, we must have 
 \be \label{cond:R} |\xi| \leq R, \qquad 2^{j_0} \e_1 \leq |2^{j} \xi_0 + \xi| \leq 2^{j_0 + 1} \e_2.\ee
(Recall indeed that $\phi_1$ is supported in $\{ \e_1 \leq |\xi| \leq 2 \e_2 \}.$) If $j_0$ is large enough so that 
$$ 2^{-j_0} R < \e_1 - 1/2, \quad \mbox{and} \quad 2^{-(j_0 - 1)} R + \e_2 < 1,$$
then there is no $\xi$ such that \eqref{cond:R} holds unless $j = j_0.$ Now 
 $$ \| \phi_1(2^{- j} D) \big( a_j e^{i x \cdot 2^j \xi^0} w \big) \|_{L^2} = a_j w_{j},$$
 where 
 $$ w_{j} = \| \phi_1(\xi^0 + 2^{-j} \cdot) \hat w \|_{L^2},$$
 so that $w_{j} \leq |w|_{L^2}.$ Besides, since $\hat  w \neq 0$ (indeed, $w \neq 0$), for some $|\eta| \leq R,$ we have $\hat w(\eta) \neq 0.$ If $j$ is large enough, then $\e_1 \leq |\xi^0 + 2^{-j} \eta| \leq 2 \e_2,$ so 
that $\phi_1(\xi^0 + 2^{-j} \eta) \neq 0,$ hence $w_j > 0,$ with a lower bound which does not depend on $j.$ Indeed, we have $|\hat w| \geq \g_0 > 0$ in the ball $B(\eta,r)$ for some $\g_0 > 0$ and some $r > 0$ which depend only on the choice of $w.$ Then, $\phi_1(\xi^0 + 2^{-j} \xi) \equiv 1$ on this ball if $j$ is large enough. Hence the lower bound for $w_{j}$ is independent on $j.$

Finally, if $|a_0|$ is large enough, then 
$$ |w_{in}(x^0)| \geq |w(x^0)|\Big( |a_0| - \sum_{j \geq 1} |a_j|\Big) > 0,$$
since $w(x^0) \neq 0.$ 
      \end{proof}

\subsubsection{Localization} \label{sec:datum:loc}

 The datum is defined as
 \be \label{def:datum:thisone} \begin{aligned} 
  u_{in}(x) & = u^1 \frac{w_{in}(x)}{|w_{in}(x^0)|} + (x - x^0) \cdot (v^0 - u^1 \frac{\d_x w_{in}(x^0)}{|w_{in}(x^0)|}\big) \psi_1(x), \end{aligned}
 \ee
 where the vector $u^1$ is introduced in Lemma \ref{lem:pola} and the scalar function $w_{in}$ in Section \ref{sec:osc}; the spatial cut-off $\psi_1$ is introduced in Section \ref{sec:loc:psi}: we have $\psi_1 \in C^\infty_c(\R^d;[0,1]),$ with $\psi_1 \equiv 1$ in a neighborhood of $0,$ and the support of $\psi_1$ is appropriately small.
 
 With the choice \eqref{def:datum:thisone}, we have,  by Lemma \ref{lem:vin}, 
 \be \label{reg:datum}
 u_{in} \in H^\s(\R^d)   \setminus \big( \cup_{\e > 0} H^{\s + \e}(\R^d) \big),
 \ee
 and also
 \be \label{init}
 u_{in}(x^0) = u^1, \qquad \d_x u_{in}(x^0) = v^0.
 \ee
 Moreover, the elliptic component of the datum is bounded from below:   
\begin{cor} \label{cor:datum} If the parameter $\delta$ measuring the size of the support of the cut-off $\psi$ \eqref{def:psi} is small enough, and if $\e > 0$ is small enough, then the datum $v_E(0) = \op_j(\psi^\flat P_E^\e) \op_j(\psi) u_{in}$ for the elliptic component $v_E$ satisfies 
\be \label{bd:below} \| \op_j(\tilde \psi) v_E(0)\|_{L^2(B(x^0,\delta))} \geq C 2^{-j \s} (1  +j)^{-1},\ee
for $j$ large enough and some constant $C > 0$ which depends on $\delta$ but not on $j,$ for some cut-off $\tilde \psi$ such that $\tilde \psi \prec \psi^\flat,$ with notation introduced in \eqref{def:prec}. 
\end{cor}

 Above, $\delta > 0$ is the small localization radius introduced in Section \ref{sec:loc:psi}.

\begin{proof} Throughout this proof, the projector $P_E$ (defined at the beginning of Section \ref{sec:proj}) is evaluated at $t = 0.$ In particular, $P_E$ depends on $x$ through the datum $u_{in} \in H^\s,$ and $\| P_E^\e \|_{0,0,1 + [d/2]} < \infty$ since $\s > d/2.$ Here we used the symbolic norms introduced in \eqref{norm:symb}. This implies, by \eqref{action}, the bound 
\be \label{pe} \| \op_j(\psi^\flat P_E^\e) \|_{L^2 \to L^2} \lesssim 1.\ee
The elliptic datum is 
$$ v_E(0) = \op_j(\psi^\flat P_E^\e) \op_j(\psi) u_{in}.$$
The second term in $u_{in}$ does not contribute much to the $L^2$ norm of $v_E(0).$ 
Indeed, the difference $u_{in} -  w_{in} (u^1/|w_{in}(x^0)|)$ is compactly supported in $x.$ Denote it $\tilde \psi_1.$ We have
$$ \op_j(\psi) \tilde \psi_1 = \op_j(\psi_1) \psi_2(2^{-j} D) \tilde \psi_1,$$
and since the Fourier transform of $\tilde \psi_1$ belongs to the Schwartz space, the $L^2$ norm of $\psi_2(2^{-j} D) \tilde \psi_1$ is controlled by $2^{-j N'},$ for $N'$ arbitrarily large, up to an $N'$-dependent constant. Thus we may focus on $\op_j(\psi^\flat P_E^\e) \op_j(\psi) (w_{in} u^1)$ in the rest of this proof.

At this point it is convenient to use pseudo-differential operators, as opposed to para-differential operators. By Proposition \ref{prop:lannes} and regularity of $\psi,$ 
$$\| (\op_j(\psi) - \pdo_j(\psi)) \|_{L^2 \to L^2} \lesssim 2^{-j (\s' - d)},$$
for any $\s' > 0,$ and 
$$ \| (\op_j(\psi^\flat P_E^\e) - \pdo_j(\psi^\flat P_E^\e)) \pdo_j(\psi) v \|_{L^2} \lesssim 2^{-j(\s  -1 - d/2)} \| \pdo_j(\psi) v \|_{j,d/2},$$
for any $v \in H^{d/2}.$ 
Since $\pdo_j(\psi) = \psi_1 \psi_2(2^{-j} D),$ and $\psi_1 \in C^\infty_c,$ we have
$$ \| \pdo_j(\psi) w_{in} \|_{j,d/2} \lesssim \| \psi_2(2^{-j} D) w_{in} \|_{j,d/2},$$
and by regularity of $w_{in}$ and localization of the support of $\psi_2,$  
$$ \| \psi_2(2^{-j} D) w_{in} \|_{j,d/2} \lesssim 2^{-j \s}.$$ 
With the above and \eqref{pe}, we find 
\be \label{dat:1} \| (\op_j(\psi^\flat P_E^\e) \op_j(\psi) - \pdo_j(\psi^\flat P_E^\e) \pdo_j(\psi)) (w_{in} u^1) \|_{L^2} \lesssim 2^{-j(2 \s - 1 - d/2)},\ee
and we now focus on $\pdo_j(\psi^\flat P_E^\e) \pdo_j(\psi) (w_{in} u^1).$ We have
$$ \pdo_j(\psi)(w_{in} u^1)(x) = \psi_1(x) \sum_{j' \geq 0} \psi_2(2^{-j} D) (a_j e^{i x \cdot 2^{j'} \xi^0} w u^1)(x).$$
Just like in the proof of Lemma \ref{lem:vin}, we have $\psi_2(2^{- j'} D) (e^{i x \cdot 2^{j'} \xi^0} w) \neq 0$ only if $j' = j.$ Thus 
$$ \begin{aligned} \pdo_j(\psi) w_{in} & = \psi_1(x) \psi_2(2^{-j} D)(a_j e^{i x \cdot 2^{j} \xi^0} w) \\  &= a_j e^{i x \cdot 2^{j} \xi^0} a_j \psi_1(x) \psi_2(\xi^0 + 2^{-j} D) w,  \end{aligned}$$ 
and 
\be \label{for:dat} \pdo_j(\psi^\flat P_E^\e) \pdo_j(\psi) (w_{in} u^1) = a_j e^{i x \cdot 2^{j} \xi^0} \psi_1^\flat(x) \pdo((\psi_2^\flat P_E^\e)(\xi^0 + 2^{-j} \cdot)) \tilde w (x),\ee 
where we denote $\tilde w := u^1 \psi_1 \psi_2(\xi^0 + 2^{-j} D) w.$ We approximate 
\be \label{approx:datum} \pdo((\psi_2^\flat P_E^\e)(\xi^0 + 2^{-j} \cdot)) \tilde w (x) \simeq P_E^\e(x,\xi^0) \tilde w.\ee
The error in \eqref{approx:datum} is 
$$ \int_{\R^d} e^{i x \cdot \xi} \big( \psi_2^\flat(\xi^0 + 2^{-j} \xi) P_E^\e(x, \xi^0 + 2^{-j} \xi) - P_E^\e(x,\xi^0) \big) {\mathcal F}(\tilde w)(\xi) \, d\xi.$$ 
We have, using $\psi_2^\flat(\xi^0) = 1,$ 
$$ | \psi_2^\flat(\xi^0 + 2^{-j} \xi) P_E^\e(x, \xi^0 + 2^{-j} \xi) - P_E^\e(x,\xi^0)| \leq  2^{-j} \| \d_\xi(\psi_2^\flat P_E^\e(x,\cdot) \|_{L^\infty(\R^d_\xi)} |\xi|.$$
Thus the error in \eqref{approx:datum} is bounded, in $L^2(B(x^0,1))$ norm, by 
$$  2^{-j} \| \d_\xi(\psi_2^\flat P_E^\e) \|_{L^\infty(\R^{2d})} \| {\mathcal F}(\d_x \tilde w)\|_{L^1}.$$
By regularity of the datum and $P_E^\e,$ the norm $\| \d_\xi(\psi_2^\flat P_E^\e) \|_{L^\infty(\R^{2d})}$ is finite (and independent of $j$). Besides, since $\psi_1$ is smooth and compactly supported, 
$$  \| {\mathcal F}(\d_x \tilde w)\|_{L^1} \lesssim \| \psi_2(\xi^0 + 2^{-j} D) w))\|_{H^{1 + d/2^+}},$$
where we used the bound $| \hat f \|_{L^1} \lesssim \| f \|_{H^{d/2^+}},$ where $d/2^+$ is any real number strictly greater than $d/2.$ Since $\hat w$ is smooth and compactly supported, the norm $\| \psi_2(\xi^0 + 2^{-j} D) w))\|_{H^{1 + d/2^+}}$ is finite; it is also bounded independently of $j.$ Thus the error in \eqref{approx:datum} appears to be $O(2^{-j})$ in the $L^2(B(x^0,1))$ norm. The operator in \eqref{approx:datum} appears with an extra $\psi_1$ factor in \eqref{for:dat}, which restricts the analysis to the $B(x^0,1)$ ball. Thus the leading term in the elliptic component of our datum is 
\be \label{cont:dat} a_j e^{i x \cdot 2^{j} \xi^0} \psi_1^\flat(x) P_E^\e(x,\xi^0) \tilde w,\ee
and our goal is to prove a bound from below as in \eqref{bd:below} for its $L^2(B(x^0,\delta))$ norm.

By Lemma \ref{lem:pola}, we have $P_E(x^0,\xi^0) u^1 \neq 0.$ Thus $P_E^\e(x^0,\xi^0) u^1 \neq 0$ if $\e$ is small enough. Besides, $\tilde w(x^0) \neq 0.$ Indeed, 
 $$ \tilde w(x) = u^1 \psi_1(x) w(x) + O(2^{-j}),$$
 in $L^2$ norm by regularity of $w,$ and $(\psi_1 w)(x^0) \neq 0$ by choice of $w$ and assumption on $\psi_1.$ 
In conclusion, the continuous map \eqref{cont:dat} is not equal to zero at $x^0,$ thus stays away from zero in a small neighborhood of $x^0.$ 

Finally, we put in the cut-off $\op_j(\tilde \psi).$ With the arguments exposed above, we may consider the action of $\pdo_j(\tilde \psi) = \tilde \psi_1(x) \tilde \psi_2(2^{-j} D)$ on \eqref{cont:dat}. Just like above, using the fast oscillations in \eqref{cont:dat} and a regularity argument, we see that the leading term in
\be \label{cond:dat2} \pdo_j(\tilde \psi) (a_j e^{i x \cdot 2^{j} \xi^0} \psi_1^\flat(x) P_E^\e(x,\xi^0) \tilde w)\ee
 is
 $\tilde \psi_1(x) \tilde \psi_2(\xi^0)$ times the map in \eqref{cont:dat}. It suffices to impose $\tilde \psi_1(x^0) = 1,$ and $\tilde \psi_2(\xi^0) = 1.$ Then we have a map that is not equal to $0$ at $x^0.$ By continuity, it says far from zero on a neighborhood of $x^0.$ This implies \eqref{bd:below}, by definition of $a_j$ in \eqref{def:aj}, if $\delta$ is small enough, depending on $P_E^\e$ (evaluated at $t = 0$) and $w.$ 
\end{proof}

\subsection{Sorting out the remainder term} \label{sec:remainder:1} 

We go back to system \eqref{eq:HE}, which we reproduce here:
\be \label{eq:HE:again}
 \d_t v_\star + \op_j(M_\star) v_\star  =  g_\star =  \op_j(\psi^\flat P_\star^\e) g -\big(  R_\star + \op_j( \psi^\flat \d_t P_\star^\e) \big) v, \qquad \star \in \{H, E\},
\ee
where $g$ is defined in \eqref{def:g}, in terms of $f$ and $f_1$ defined in \eqref{def:f} and \eqref{def:f1}, and $R_\star$ is defined in \eqref{def:RH}. Recall that $v = \op_j(\psi) \tilde u,$ and $\tilde u(t,x) = u(2^{-j} t).$ We say that terms in $\tilde u$ in $g_\star$ are poorly localized, and we will label these terms ``out''. Terms in $v$ in $g_\star$ will be easier to handle. Those are  
\be \label{Gstar}
G_\star v := - 2^{j \theta'} \Big( \op_j(\psi^\flat P_\star^\e) \big( \op_j(\psi^\sharp \tilde A) - \op_j(\tilde A) \op_j(\psi^\sharp) \big) + ( R_\star + \op_j( \psi^\flat \d_t P_\star^\e) \Big) v .
\ee
The poorly localized terms are %
\be \label{def:Gstarout} \begin{aligned} 
 G_\star^{out}  :=  - \op_j(  \psi^\flat P_\star^\e) \Big( \tilde R^{para}_j + \op_j(\tilde A) \big(\op_j(\psi^\sharp) \op_j(\psi) - \op_j(\psi)\big) +  \big[ \op_j(\psi), \op_j(\tilde A) \big] \Big) \tilde u.  \end{aligned}\ee
 Note the $2^{j \theta'}$ prefactor that we put in front of $G_\star,$ where 
 \be \label{def:theta'}
 \theta' := \frac{\theta}{1 + \theta}.\ee
The reason for this prefactor is that we will prove in Section \ref{sec:it} that $G_\star$ is $L^2 \to L^2$ bounded with norm, uniformly in $j.$ The initial-value problem for systems \eqref{eq:HE:again} takes the form
\be \label{eq:HE:3} \left\{ 
\begin{aligned} \d_t v_\star + \op_j(M_\star) v_\star &  = 2^{-j \theta'} G_{\star} v + G^{out}_\star, \\ v_\star(0)&  = \op_j(\psi P_\star) \op_j(\psi) u_{in} \end{aligned}\right. \qquad \star \in \{H, E\}.
 \ee

\subsection{Integral representation in forward time} \label{sec:int:rep}

The symbol $M_\star,$ with $\star = E$ or $\star = H$ \eqref{def:MHE} is order zero. Thus the associated operator $\op_j(M_\star)$ is linear bounded $L^2 \to L^2,$ and it has a solution operator $S_\star(t';t),$ for $0 \leq t' \leq t,$ defined as the map which takes a datum $z^0$ at time $t'$ to $z(t),$ where $z(t)$ is the unique solution at time $t$ to the 
initial-value problem
 $$ \d_t z + \op_j(M_\star) z = 0, \qquad z(t') = z^0.$$
For the solution $v$ to \eqref{eq:HE:3}, this means that we have (implicit) representation formulas in forward time:
\be \label{rep:v}
 v_\star(t) = S_\star(0;t) v_\star(0) +  \int_0^t S_\star(t';t) \big( 2^{-j \theta'} G_\star v(t') + G_\star^{out}(t') \big) \, dt'.
\ee 

\subsection{Iterations of the integral representation in forward time} \label{sec:it}

Consider \eqref{rep:v}: it is an implicit representation of $v_H$ or $v_E$ in terms of itself (see the $v$ term in the right-hand side), and of the remainder terms. 
 We may use inductively \eqref{rep:v}, and obtain
$$ \begin{aligned} v_H(t) & 
= S_H(0;t) v_H(0) + 2^{-j \theta'} \int_0^t S_H(t';t) G_H(t')  (S_E(0;t') v_E(0) + S_H(0;t') v_H(0)) \, dt'\\ & + 
2^{-2j \theta'} \int_0^t \int_0^{t'} S_H(t';t) G_H(t') \big( S_H(t'';t') G_H(t'') + S_E(t'';t') G_E(t'')) v(t'') \, dt' \, dt \\ & + \int_0^t S_H(t';t) G_H^{out}(t') \, dt' \\ & + 2^{-j \theta'}\int_0^t \int_0^{t'} S_H(t'';t') G_H(t') \big( S_H(t'';t') G^{out}_H(t'') + S_E(t'';t') G^{out}_E(t'') \big) dt'' \, dt',
\end{aligned}$$
and
$$ \begin{aligned}
 v_E(t) & =  S_E(0;t) v_E(0) + 2^{-j \theta'} \int_0^t S_E(t';t) G_E(t') (S_E(0;t') v_E(0) + S_H(0;t') v_H(0)) \, dt' \\ & + 2^{-2j \theta'} \int_0^t\int_0^{t'} S_E(t';t) G_E(t') \big( S_H(t'';t') G_H(t'') + S_E(t'';t') G_E(t'')) v(t'') \, dt' \, dt  \\ & + \int_0^t S_E(t';t) G_E^{out}(t') \, dt' \\ & + 2^{-j \theta'} \int_0^t\int_0^{t'} S_E(t'';t') G_E(t') \big( S_H(t'';t') G^{out}_H(t'') + S_E(t'';t') G^{out}_E(t'') \big) dt'' \, dt'.
 \end{aligned} 
$$ 
 Iterating, we obtain for any $k \geq 1:$ 
\be \label{rep:v:it} 
 v_\star(t) = v_{\star k}^f(t) + 2^{-jk \theta'} \int_{D_k(t)} S_{\star k}(T_k;t) v(t_k) dT_k + G^{out}_{\star k}(t),
 \ee
 where for $t > 0$ we define the $k$-dimensional time domain
 \be \label{def:Dk}
 D_k(t) = \Big\{ (t_k;t_{k-1};\dots;t_1) \in \R^{k}, \qquad 0 < t_k < t_{k-1} < \dots < t_1 < t \Big\}.
 \ee
 We denote $T_k$ a vector in $D_k(t):$ 
 $$ T_k := (t_k;t_{k-1};t_{k-2};\dots;t_1) \in D_k(t),$$
 and $dT_k$ the $k$-dimensional Lebesgue measure:
 $$ dT_k = dt_k dt_{k-1} \cdots dt_1.$$
 The operator $S_{\star k}$ is defined by induction as
 \be \label{def:Sk:1}
 S_{\star 1}(t';t) := S_\star(t';t) G_\star(t'), \qquad 0 \leq t' \leq t,
 \ee
 and
 \be \label{def:Sk:k} \begin{aligned}
  S_{\star k}&(T_k;t) \\ & = S_{\star k-1}(t_{k-1};\dots;t) \circ \Big( \, S_E(t_k;t_{k-1}) G_E(t_k) + S_H(t_k;t_{k-1}) G_H(t_k) \, \Big), \qquad k \geq 2.\end{aligned}
  \ee
 The ``free'' solutions $v^f_{H k}(t)$ and $v^f_{E k}(t)$ are defined as
\be \label{def:free:ops} \begin{aligned}
v^f_{\star k}(t) & = S_\star(0;t) v_\star(0) \\ & \quad + \sum_{1 \leq k' \leq k} 2^{-j k' \theta'} \int_{D_{k'}(t)} S_{\star k'}(T_{k'};t) \big( S_H(0;t_{k'}) v_H(0) + S_E(0;t_{k'}) v_E(0) \big)\,  dT_{k'} \end{aligned}
 \ee
 Finally, the remainders $G^{out}_{\star k}$ are defined as
\be \label{def:Gout} \begin{aligned}
 & G^{out}_{\star k}(t)  = \int_0^t S_\star(t_1;t) G^{out}_{\star}(t_1) \, dt_1 \\ & + \sum_{2 \leq k' \leq k} 2^{- j (k'-1) \theta'} \int_{D_{k'-1}(t)} S_{\star k'-1}(T_{k'-1};t) \Big( \, S_H(t_{k'-1};t_{k'}) G^{out}_H(t_{k'}) + S_E(t_{k'-1};t_{k'}) G^{out}_E(t_{k'}) \, \Big) dT_{k'}. 
\end{aligned}
\ee 
Recall that $\theta'$ is defined in Section \ref{sec:remainder:1} as $\theta' = \theta/(1 + \theta).$ 

\subsection{Backward-in-time elliptic solution operator} \label{sec:bw}

Consider now the backward-in-time solution operator for the vector field $\op_j(M_E)$ in $L^2.$ This solution operator is $S_E(t;0),$ where $S_E(0;t)$ is introduced in Section \ref{sec:int:rep} as the forward-in-time solution operator. We apply $S_E(t;0)$ to the integral representation \eqref{rep:v:it} with $\star = E,$ and obtain
 \be \label{rep:bw:0} \begin{aligned}
S_E(t;0) v^f_{E k}(t) & = S_E(t;0) v_E(t) - S_E(t;0) G^{out}_{E k}(t) \\ & - 2^{-jk \theta'} S_E(t;0) \int_{D_k(t)} S_{E k}(T_k;t) v(t_k) dT_k, \qquad \mbox{for any $k \geq 1.$}
\end{aligned}
 \ee
We will observe that 
$$S_E(t;0) v^f_{E k}(t) \simeq v_E(0).$$ 
(This is made precise in Corollary \ref{lem:bw:fw} in Section \ref{sec:bd:bw} below). Thus we write \eqref{rep:bw:0} in the form
\be \label{rep:bw:1} \begin{aligned} 
 v_E(0) & =  \big(v_E(0) - S_E(t;0) v^f_{E k}(t)\big) + S_E(0;t) v_E(t) - S_E(t;0) G^{out}_{E k}(t) \\ & - 2^{-jk \theta'} S_E(t;0) \int_{D_k(t)} S_{E k}(T_k;t) v(t_k) dT_k, \qquad \mbox{for any $k \geq 1.$}
\end{aligned}
 \ee 
The end of the proof is based on identity \eqref{rep:bw:1}: we have a lower bound for the $L^2$ norm of $v_E(0)$ (this is Corollary \ref{cor:datum}), and the contradiction eventually comes from upper bounds for the right-hand side of \eqref{rep:bw:1}. In order to derive these upper bounds, we first turn to bounds for the forward-in-time propagators $S_\star(0;t).$

\subsection{Bounds for the forward-in-time solution operator} \label{sec:bd:S}

Recall that $M_H$ and $M_E$ are introduced in \eqref{def:MHE}, and that $S_H$ and $S_E$ are the associated solution operators, introduced in Section \ref{sec:int:rep}.

In computing bounds for the solution operators $S_\star,$ we may assume that $M_\star$ is evaluated at $t = 0.$ Indeed, this simplification results in an extra remainder term 
 $\op_j(M_\star(t) - M_\star(0)) v,$ which is acting on $v$ and is $O(2^{-j})$ because of the time rescaling. As a consequence, in the iterated representation formula this term belongs to the same category as $G_\star,$ and thus comes in with a $2^{-j k}$ prefactor, with $k$ arbitrarily large.

\begin{lem} \label{lem:S} For some $\g_E > 0$ such that $\g_E \to \Re e \, \l_0$ as $\delta \to 0,$ some $\g_H  \geq 0$ with $\g_E > \g_H,$ for $0 \leq t' \leq t,$ we have the bounds   
 $$ \| S_\star(t';t) \|_{L^2 \to L^2} \lesssim e^{(t - t') \g_\star}, \qquad \star \in \{H, E\},$$
with an implicit constant which depends on $\delta.$ 
\end{lem}

Above, $\l_0$ is the most unstable eigenvalue of $M,$ as defined in Section \ref{sec:spectraldec}, and $\delta$ measures the radius of the space-frequency truncation $\psi$ (or $\psi^\sharp$) which enters in the definition of $M_E$ and $M_H$ (see Section \ref{sec:loc:psi} for the definition of $\psi$ and Section \ref{sec:proj} for the definitions of $M_H$ and $M_E$).

\begin{proof} We could directly apply Proposition \ref{lem:rates}, based on G\r{a}rding's inequality, with $Q = - M_H$ then $Q = - M_E.$ However, our ellipticity assumption (Assumption \ref{ass:ell}) bears on the real part of the spectrum of $i A,$ by extension of $M_E$ and $M_H,$ but condition \eqref{rate:upper} that is used in Proposition \ref{lem:rates}, bears on the spectrum of $\Re e \, M_\star,$ which might happen to be far from the real part of the spectrum of $M_\star.$ This is an issue in view of deriving an upper rate of growth $\g_E$ that will be close to the lower rate of growth $\g_E^-$ (see Lemma \ref{lem:bd:bwS}). 

 We overcome this problem by approximately diagonalizing $M_\star.$ This is reminiscent of the procedure carried out in Section 4 of \cite{LNT}.

 The goal is to derive an upper bound for $u_\star(t) := S_\star(t';t) u_0,$ given $u_0 \in L^2.$ By definition, $u_\star$ solves
  $$ \d_t u_\star + \op_j(M_\star) u_\star = 0, \qquad u(t') = u_0.$$
Let $K_0$ be a matrix of a change of basis to upper triangular form for $M_\star(0,x^0,\xi^0).$ Let $K_\delta$ be the invertible diagonal matrix with diagonal equal to $(1, \delta, \delta^2, \dots, \delta^N).$ Then, $$\tilde M_\star(x^0,\xi^0) := K_\delta^{-1} K_0^{-1} M_\star(0,x^0,\xi^0) K_0 K_\delta$$ is upper triangular, with entries above the diagonal which are $O(\delta).$ The support of the space-frequency $\psi^\sharp$ which enters in $M_\star$ has size $O(\delta).$ Since $M_\star$ has H\"older regularity in $(x,\xi),$ the symbol $$\tilde M_\star(x,\xi) := K_\delta^{-1} K_0^{-1} M_\star(0,x,\xi) K_0 K_\delta$$ is at a distance $O(\delta^{\theta})$ of $\tilde M_\star(x^0,\xi^0).$ We let 
 $\tilde u_\star = K_\delta^{-1} K_0^{-1} u_\star.$ Then $\tilde u_\star$ solves
 \be \label{eq:vstar} \d_t \tilde u_\star + \op_j(\tilde M_\star) \tilde u_\star = 0.\ee 
 By \eqref{puiseux}, the eigenvalues of $\Re e \, \tilde M_\star(x,\xi)$ differ from the eigenvalues of $\Re e \, \tilde M_\star(x^0,\xi^0)$ at most by $O(\delta^{1/N}).$ Now the key is that, since $\tilde M_\star(x^0,\xi^0)$ is upper triangular with entries above the diagonal that are $O(\delta),$ the spectrum of $\Re e \, \tilde M_\star(x^0,\xi^0)$ is differs from the real part of the spectrum of $M_\star(x^0,\xi^0)$ by $O(\delta^{1/N})$ at most.

 Thus the spectrum of $\Re e \, \tilde M_\star(x,\xi)$ differs from the real part of the spectrum of $M_\star(x^0,\xi^0)$ by $O(\delta^{\min(\theta,1/N)})$ at most. The spectra of $M_E(x^0,\xi^0)$ and $M_H(x^0,\xi^0)$ are described in Section \ref{sec:spectraldec}: the unique eigenvalue of $M_E(x^0,\xi^0)$ is $\l_0$ (see \eqref{def:mu}), and the eigenvalues of $M_H(x^0,\xi^0)$ have real part strictly smaller than $\Re e \, \l_0$ (see \eqref{spectral:sep}).  
 
 Thus 
 $$ \Re e \, \tilde M_\star(x,\xi) \leq \g_E, \qquad \mbox{for all $(x,\xi),$ with $\g_E \to \Re e \, \l_0$ as $\delta \to 0.$}$$
 By Proposition \ref{lem:rates}, this implies
 $$ \| \tilde u_\star(t)\|_{L^2} \lesssim e^{(t - t') \g_E},$$
 the small error in the growth rate from Proposition \ref{lem:rates} being incorporated into the constant. A corresponding bound for $u_\star = S_\star(t';t) u_0$ follows.
\end{proof} 

\subsection{Bounds for the iterated forward-in-time solution operators} \label{sec:bd:Sk} 

The iterated solution operators $S_{\star k}$ are defined in \eqref{def:Sk:k} in Section \ref{sec:it}. They involve the remainders $G_\star,$ for which standard pseudo-differential bounds, and regularity of the posited solution, give the following:

\begin{lem} \label{lem:Gstar} On the time interval of existence of the posited solution $\tilde u,$  we have the bounds 
 $$ \| G_\star z \|_{L^2} \lesssim \| z \|_{L^2}, \qquad \mbox{for all $z \in L^2$ and $\star \in \{H, E\},$}$$
 provided that the small parameter $\e$ involved in the regularization of the eigenprojectors $P_\star$ in Section {\rm \ref{sec:proj}} satisfies $\e = 2^{-j /(1 + \theta)}.$ 
 \end{lem}
 
\begin{proof} The remainder $G_\star$ is defined in \eqref{Gstar}, which we reproduce here:
$$ G_\star v := - 2^{j \theta'} \Big( \op_j(\psi^\flat P_\star^\e) \big( \op_j(\psi^\sharp \tilde A) - \op_j(\tilde A) \op_j(\psi^\sharp) \big) + R_\star + \op_j(\psi^\flat \d_t P_\star^\e) \Big) v.$$ 
The projectors $P_\star^\e$ are as regular as $M$ in $\xi$ and $(u,\d_x u)$ in $x.$ In particular, using notation for norms of symbols introduced in \eqref{norm:symb}, $\| \psi^\flat P_\star^\e \|_{0,0,1 + [d/2]} < \infty,$ since $(u,\d_x u) \in H^{s-1} \hookrightarrow L^\infty.$ Thus by \eqref{action}, for all $z \in L^2,$  
 $\| \op(\psi^\flat P_\star^\e) z\|_{L^2 \to L^2} \lesssim \|z \|_{L^2}.$
By regularity of $u,$ Lemma \ref{lem:symbols} and the composition result of Section \ref{sec:composition}, we have, just like in the projection step of Section \ref{sec:proj}, 
 $$ \| \op_j(\psi^\sharp \tilde A) - \op_j(\tilde A) \op_j(\psi^\sharp) \|_{L^2 \to L^2} \lesssim 2^{- j \theta}.$$ 
The term $R_\star$ contains terms of the form $R_{\theta}(\cdot,\cdot)$ which are $L^2 \to L^2$ bounded, uniformly in $j,$ by regularity of $u$ and the composition result of Section \ref{sec:composition}. The other terms in $R_\star$ involve the difference $P_\star^\e - P_\star,$ where $P^\e_\star$ was defined in Section \ref{sec:proj} as a regularization of the eigenprojectors $P_\star.$ We use the results of Section \ref{sec:reg}: 
$$ \| \op_j(\psi^\flat (P_H - P_H^\e)) \|_{L^2 \to L^2} \lesssim \e^\theta.$$ 
Since $\op_j(M)$ and $\op_j(M_\star)$ are $L^2 \to L^2$ bounded, uniformly in $j,$ by \eqref{action}, this implies
$$ \| R_\star \|_{L^2 \to L^2} \lesssim 2^{-j \theta} + \e^\theta. $$ 

The last term in $G_\star$ implies $\d_t P^\e_\star.$ This is where the regularization was needed. We denote, as in the proof of Lemma \ref{lem:pola},
 $$ P_\star(x,\xi) = \underline P_\star\big(t,x, u( 2^{-j} t,x) ,\d_x u(2^{-j} t,x), \xi\big),$$
 so that (see Section \ref{sec:reg}),
 $$ P_\star(t,x,\xi) = \underline  P_\star(2^{-j} t,x, k_\e \star \tilde u(t,x) , k_\e \star \d_x \tilde u(t,x), \xi).$$
 Then, we have
 $$ \d_t P_\star = 2^{-j} \Big( \d_1 \underline P_\star + \d_3 \underline P_\star \cdot k_\e \star \d_t u + \d_4 \underline P_\star \cdot k_\e \star \d_t \d_x u \Big).$$
 The map $\underline P_\star$ is smooth in its arguments, by spectral separation and smoothness of $F.$ In view of the original equation \eqref{equation reference}, we observe that
 $$ \d_t k_\e \star u + k_\e \star F = 0, \quad \mbox{so that} \quad k_\e \star \d_t \d_x u  = - (\d_x k_\e) \star F,$$
 where $F = F(t,x,u,\d_x u).$ 
 In particular, by definition of the regularizing kernel $k_\e$ (or \eqref{prop:h3}),  
 $$ \| k_\e \star \d_t \d_x u \|_{L^\infty} \lesssim \e^{-1},$$
 so that (by \eqref{action}) 
 $$ \| \op(\psi^\flat \d_t P_\star^\e) \|_{L^2 \to L^2} \lesssim 2^{-j} \e^{-1}.$$  
 
 It remains to choose $\e$ so that the terms are the same order of magnitude: $\e^\theta = 2^{-j} \e^{-1},$ leading to $\e = 2^{- j /(1  +\theta)},$ and the $L^2 \to L^2$ norm of $G_\star$ finally appears to be bounded in $j,$ since $\theta' =  \theta/(1 + \theta).$ 
   \end{proof}

\begin{cor} \label{lem:Sk} We have the bounds
$$ \| S_{\star k}(T_k;t) \|_{L^2 \to L^2} \lesssim e^{\g_E (t - t_k)}, \qquad \star \in \{H, E\},$$
with the elliptic rate of growth $\g_E$ given in Lemma {\rm \ref{lem:S},} and an implicit constant which does not depend on $(t';t).$ 
\end{cor}

\begin{proof} Based on the definition of the iterated operators \eqref{def:Sk:k}, we only need to apply repeatedly Lemmas \ref{lem:S} and \ref{lem:Gstar}, and use $\g_H < \g_E.$ 
\end{proof}

\subsection{Bound for the backward-in-time operator} \label{sec:bd:bw}

We prove here a key upper bound for the action of the backward-in-time solution operator $S_E(t;0)$ associated with the most unstable eigenvalues. 

\begin{lem} \label{lem:bd:bwS} For some $\g_E^- \in \R,$ with $\g_E^- < \g_E,$ we have, so long as $t \leq t_\star,$ with $t_\star$ defined in \eqref{bd:t}, the bound  
 $$ \| \op_j(\tilde \psi) \circ S_E(t;0) \|_{L^2 \to L^2} \lesssim e^{- t \g_E^-},$$
 with an implicit constant which does not depend on $t$ nor on $j.$ The cut-off $\tilde \psi$ is as in Corollary {\rm \ref{cor:datum}.} If the support of $\tilde \psi$ is small enough, then the lower rate of growth $\g_E^-$ is positive. By shrinking the support of $\tilde \psi,$ we can make $\g_E - \g_E^-$ arbitrarily small.
 \end{lem}
 
 \begin{proof} We use the proof of Lemma \ref{lem:S}. The spectrum of $\Re e \, \tilde M_E$ differs from $\Re e \, \l_0$ only by $O(\delta^{1/N}).$ Thus
  $$ \Re e \, \tilde M_E(x,\xi) \geq \g_E^-, \qquad \mbox{for all $(x,\xi)$ on the support of $\psi,$}$$
  for some $\g_E^-$ such that $\g_E^- \to \Re e \, \l_0$ as $\delta \to 0.$ Proposition \ref{lem:rates} then provides a lower bound for the solution $\tilde u_E$ to \eqref{eq:vstar}: 
  $$ \| \op_j(\tilde \psi) \tilde u_E(0) \|_{L^2} \lesssim e^{-t \g_E^-} \| \tilde u_E(t) \|_{L^2},$$
  for $\tilde u_\star(0)$ such that $\op_j(\tilde \psi) \tilde u_\star(0) \neq 0,$ and for $t \leq t_\star.$ 
 This translates into the announced bound for $S_E,$ up to $\delta$-dependent implicit constants.
 \end{proof}

\begin{cor} \label{lem:bw:fw}
 We have, so long as $t \leq t_\star,$ with $t_\star$ defined in \eqref{bd:t}, the bound 
 $$ \| \op_j(\psi^\flat \big( v_E(0) - S_E(t;0) v^f_{E k}(t) \big) \|_{L^2} \lesssim 2^{-j (\s + \theta')} e^{t (\g_E - \g_E^-)},$$
 where $\theta' = \theta/(1 + \theta).$ 
\end{cor}

\begin{proof} Based on the definition of $S^f_{Ek}$ in \eqref{def:free:ops}, we have
$$ \begin{aligned} v_E(0) & - S_E(t;0) v^f_{E k}(t) \\ & = \sum_{1 \leq k' \leq k} 2^{-j k' \theta'}  S_E(t;0) \int_{D_{k'}(t)} S_{E k'}(T_{k'};t) \big( S_H(0;t_{k'}) v_H(0) + S_E(0;t_{k'}) v_E(0) \big)\,  dT_{k'}.
\end{aligned}$$ 
For a bound of $\op_j(\psi^\flat) S_E(t;0) S_{E k'}(T_{k'};t) S_E(0;t_{k'}),$ we use Lemma \ref{lem:S}, Corollary \ref{lem:Sk} and Lemma \ref{lem:bd:bwS}:
$$ \| \op_j(\psi^\flat) S_E(t;0) S_{E k'}(T_{k'};t) S_E(0;t_{k'}) \|_{L^2 \to L^2} \lesssim e^{\g_E t_{k'}} e^{\g_E(t - t_{k'})} e^{- \g_E^- t} = e^{t(\g_E - \g_E^-)}.$$
Since $v_\star(0)$ is essentially the $j$th dyadic block of an $H^\s$ datum (up to a projection), we have $\| v_\star(0) \|_{L^2} \lesssim 2^{-j \s}.$ Integrating a constant over $D_k(t),$ with $t \leq t_\star = O(j),$ we find a prefactor that is bounded by a constant times $j^{k+1} 2^{-j k \theta'},$ which can be made less than one by choosing $j$ large enough.  
\end{proof}

\subsection{Bound for the poorly localized remainder terms} \label{sec:bound:out}

We turn to the ``out'' terms $G_\star^{out}$ defined in \eqref{def:Gstarout}, which we reproduce here:
$$ %
\begin{aligned} 
 G_\star^{out} & :=  - \op_j(\psi^\flat P_\star^\e) \Big( \tilde R^{para}_j + \Big( \op_j(\tilde A) \big(\op_j(\psi^\sharp) \op_j(\psi) - \op_j(\psi)\big) +  \big[ \op_j(\psi), \op_j(\tilde A) \big] \Big) \tilde u. 
  \end{aligned}
  $$ %
Unlike the ``in'' remainder terms handled in Lemma \ref{lem:Gstar}, the ``out'' terms above require sharp bounds. The third term in $G_\star^{out}$ is bounded as follows: 

\begin{lem} \label{lem:out1} On a time interval of existence of the posited solution $\tilde u,$ we have the bound 
 $$ \|\op_j(\psi^\flat P_\star^\e) \big[ \op_j(\psi), \op_j(\tilde A) \big] \tilde u \|_{L^2} \lesssim 2^{-j(2s - 1 - d/2)}.$$
\end{lem}

Note that by the composition result of Section \ref{sec:composition}, and the fact that $\psi^\flat \prec \psi,$ the $L^2 \to L^2$ norm of the operator $\op_j(\psi^\flat P_\star^\e) \big[ \op_j(\psi), \op_j(\tilde A) \big] $ is $O(2^{-j (s - 1 - d/2)}).$ The issue here is that $\tilde u$ is not localized in frequency.  
The proof below gives a description of that operator which shows that, in effect, only frequencies $\sim 2^j$ of $\tilde u$ come into play. 

\begin{proof} Recall that $\tilde A$ is defined in \eqref{def:tildeA}. It has the form
 $$ \tilde A =  A_1 \cdot \xi + 2^{-j} A_0, \qquad  A_0,  A_1 \in H^{s-1}.$$  
We may focus on the $A_1$ term, since $A_0$ has the same regularity and is smaller. By definition of para-differential operators in semiclassical quantization \eqref{def:opj},
 $$ \op_j(A_1 \cdot \xi) = H_j^{-1} \op(h_j A_1 \cdot \xi) H_j,$$ and 
 by definition of the admissible cut-off \eqref{def:phiadm},
 $$ H_j^{-1} \op(h_j A_1 \cdot \xi) H_j = \sum_{k \geq 0} H_j^{-1} \Big( (\phi(2^{-k + N_0} D_x) h_j A_1) \cdot \bar \phi_k(D_x) \Big) H_j, \qquad \bar \phi_k(\xi) := \xi \phi_k(\langle \xi \rangle).$$ 
  We observe the identity
\be \label{h:id}
 \chi(D_x) (h_j f) = h_j (\chi(2^{-j} D_x) f),
\ee 
for any Fourier multiplier $\chi(D_x)$ and any Sobolev map $f.$ 
 We have, for any maps $f,\chi,v:$
 $$ H_j^{-1} (f \chi(D_x) v) = (H_j^{-1} f)( h_j^{-1} \chi(D_x) v),$$
 since $f$ acts on $\chi(D_x) v$ by multiplication. Thus 
$$ H_j^{-1} \Big( (\phi(2^{-k + N_0} D_x) h_j A_1)\bar \phi_k(D_x) \Big) H_j v =  h_j^{-1}\Big(  \phi(2^{-k + N_0} D_x) h_j A_1 \Big) h_j^{-1} \Big( \bar \phi_k(D_x) (h_j v) \Big).$$
As a consequence, by \eqref{h:id},
  $$ H_j^{-1} \Big( (\phi(2^{-k + N_0} D_x) h_j A_1)\bar \phi_k(D_x) \Big) H_j v  = (\phi(2^{-k - j + N_0} D_x) A_1) (\bar \phi_k(2^{-j} D_x) v).$$ 
Recall \eqref{def:psi} that $\psi(x,\xi) = \psi_1(x) \psi_2(\xi).$ We use Proposition \ref{prop:lannes}:
 \be \label{lannes:psi} \| \pdo_j(\psi) - \op_j(\psi) \|_{L^2 \to L^2} \lesssim 2^{-j(s'' - d)},\ee 
 for any $s'' > 0,$ since $\psi_j \in C^\infty_c.$ Thus we may replace $\op_j(\psi)$ by $\pdo_j(\psi)$ in the commutator under scrutiny, and 
 $$ \pdo_j(\psi) = H_j^{-1}  h_j \psi_1 \psi_2(D_x)  H_j.$$ 
 We have, just like above, 
 $$ H_j^{-1} h_j \psi_1 \psi_2(D_x)  H_j = \psi_1 \psi_2(2^{-j} D_x).$$
 The focus is on 
$$  [\op_j(\psi), \op_j(A_1 \cdot \xi)]   =  \sum_{k \geq 0} \Gamma_k,$$
with notation  
$$ \Gamma_k := \psi_1 \psi_2(2^{-j} D_x) A_1^{j+k} \bar \phi_k(2^{-j} D_x) - A_1^{j+k} \bar \phi_k(2^{-j} D_x) \psi_1 \psi_2(2^{-j} D_x). $$  
up to a very small term described in \eqref{lannes:psi}, with notation 
\be \label{tildeA1k}
A_1^{j+k} = \phi_0(2^{-k -j + N_0} D_x) A_1.
 \ee 
We observe that 
$$ \Gamma_k = \psi_1 [\psi_2(2^{-j} D_x), A_1^{j+k}]  \bar \phi_k(2^{-j} D_x) -  A_1^{j+k} [\bar \phi_k(2^{-j} D_x),\psi_1] \psi_2(2^{-j} D_x) 
.$$ 
We now take into account the truncation and projection operator $\op_j(\psi^\flat P_\star\e^)$ to the left:
$$ \op_j(\psi^\flat P_\star^\e) = H_j^{-1} \op\big(h_j( \psi^\flat_1 \psi^\flat_2 P_\star^\e)\big) H_j. $$
Since $\psi_2^\flat$ depends only on $\xi,$ 
$$ \op\big(h_j( \psi^\flat_1 \psi^\flat_2 P_\star^\e)\big)  =  \op\big(h_j( \psi^\flat_1 P_\star^\e)\big) \psi_2^\flat(D_x).$$ 
Thus
$$ \begin{aligned} \op_j(\psi^\flat P_\star^\e) & =  h_j^{-1} \op\big(h_j( \psi^\flat_1 P_\star^\e)) h_j \psi_2^\flat(2^{-j} D_x) \\ &  = \sum_{k' \geq 0} h_j^{-1} \pdo\Big( \phi(2^{-k' + N_0} D_x)(h_j (\psi_1^\flat P_\star^\e))\Big) \phi_k(\langle D_x \rangle) h_j \psi_2^\flat(2^{-j} D_x) \\ & = \sum_{k' \geq 0} h_j^{-1} \pdo\Big( \phi(2^{-k' + N_0} D_x)(h_j (\psi_1^\flat P_\star\e^))\Big)  h_j \phi_k(2^{-j} \langle D_x \rangle) \psi_2^\flat(2^{-j} D_x).\end{aligned}$$   
Now recall that in the Littlewood-Paley decomposition $(\phi_k)_{k \geq 0},$ only $\phi_1$ is supported in a neighborhood of the unit sphere. Since $\psi_2^\flat$ is supported near $\xi^0 \in \S^{d-1},$ this means that in the above sum, only $k' = 1$ subsists, and, if the support of $\psi_2^\flat$ is chosen small enough, then $\psi_2^\flat \phi_1 \equiv \psi_2^\flat.$ So 
\be \label{def:Qstar} \begin{aligned} \op_j(\psi^\flat P_\star^\e) & = h_j^{-1} \pdo\Big( \phi(2^{-1 + N_0} D_x)(h_j( \psi_1^\flat P_\star^\e))\Big) h_j \psi_2^\flat(2^{-j} D_x)  \\ & =: Q_\star  \psi_2^\flat(2^{-j} D_x).\end{aligned}\ee 
The full operator under study now looks like
\be \label{full:op}  \sum_{k \geq 0} Q_\star \psi_2^\flat(2^{-j} D_x) \Gamma_k.\ee
Consider %
$$ \begin{aligned} \psi_2^\flat(2^{-j} D_x) \Gamma_k 
 & =: \sum_{1 \leq \ell \leq 3} \Gamma_{k,\ell},
\end{aligned}$$
with notation
$$ \begin{aligned}
\Gamma_{k,1} & := [\psi_2^\flat(2^{-j} D_x),\psi_1] [\psi_2(2^{-j} D_x), A_1^{j+k}]  \bar \phi_k(2^{-j} D_x), \\
\Gamma_{k,2} & :=  \psi_1 \psi_2^\flat(2^{-j} D_x) [\psi_2(2^{-j} D_x), A_1^{j+k}]  \bar \phi_k(2^{-j} D_x) \\
\Gamma_{k,3} & := \psi_2^\flat(2^{-j} D_x) A_1^{j+k} [\bar \phi_k(2^{-j} D_x), \psi_1] \psi_2(2^{-j} D_x).  \end{aligned}$$
We apply the operator \eqref{full:op} to $\tilde u \in H^s,$ where $\tilde u(t,x) = u(2^{-j} t,x).$ %
Once we verify the bound
\be \label{to:do}
\| Q_\star \|_{L^2 \to L^2} \lesssim 1,
\ee
with $Q_\star$ defined in \eqref{def:Qstar}, 
it will remain to show
\be \label{to:bound}
 \sum_{k \geq 0} \| \psi_2^\flat(2^{-j} D_x) \Gamma_{k,\ell} \tilde u \|_{L^2} \lesssim 2^{-j (2 s - 1 - d/2)}, \qquad \mbox{for all $\ell \in \{1,2,3\}.$}
\ee 
We temporarily admit \eqref{to:do} and prove \eqref{to:bound}.

Consider the operator $\psi_2^\flat(2^{-j} D_x) [ \psi_2(2^{-j} D_x), A_1^{j+k}]$ in $\Gamma_{k,2}.$ We observe that
$$ \psi_2^\flat \d_\xi^\a \psi_2 \equiv 0, \qquad \mbox{for all $|\a| > 0,$}$$
since $\psi_2^\flat \prec \psi_2$ (see \eqref{def:prec}). 
Thus, by Lemma \ref{lem:composition}, 
$$ \psi_2^\flat(2^{-j} D_x) [ \psi_2(2^{-j} D_x), A_1^{j+k}] = \psi_2^\flat(2^{-j} D_x) \tilde R_{s-1-d/2}(\psi_2(2^{-j} \cdot), A_1^{j+k}),$$
where we used notation $\tilde R$ from the proof of Lemma \ref{lem:composition}. Since $A_1 \in H^{s-1},$ we know from Lemma \ref{lem:composition} that the $L^2 \to L^2$ operator norm of $\tilde R_{s-1-d/2}(\psi_2(2^{-j} \cdot), A_1^{j+k})$ is $O(2^{-j (s - 1 - d/2)}).$
Thus
$$ \begin{aligned} \| \Gamma_{k,2} \tilde u \|_{L^2} & \leq \| \psi_1\|_{L^\infty} \| \psi_2^\flat \|_{L^\infty} \| \tilde R_{s-1-d/2}(\psi_2(2^{-j} \cdot), A_1^{j+k})\|_{L^2 \to L^2} \| \bar \phi_k(2^{-j} D_x) \tilde u \|_{L^2} \\ & \lesssim 2^{-j(s-1-d/2)} \| \bar \phi_k(2^{-j} D_x) \tilde u \|_{L^2}.\end{aligned}$$
The support of $\bar \phi_k(2^{-j} \cdot)$ is an annulus with frequencies $\sim 2^{j + k}.$ Since $\tilde u \in H^s,$ this implies
$$ \| \bar \phi_k(2^{-j} D_x) \tilde u \|_{L^2} \lesssim 2^{k - (j+k)s},$$
the $2^k$ factor coming from the $\xi$ factor in $\bar \phi_k,$ itself a trace of the fact that the symbol $\tilde A$ is order one. We can sum over $k$ and find
$$ \sum_{k \geq 0} \| \psi_2^\flat(2^{-j} D_x) \Gamma_{k,2} \tilde u \|_{L^2} \lesssim 2^{-j (2 s - 1 - d/2)} \sum_{k \geq 0} 2^{(1 - s) k} \lesssim  2^{-j (2 s - 1 - d/2)},$$
which proves \eqref{to:bound} in the case $\ell = 2.$

 Consider now $\Gamma_{k,1}.$ The commutator $[\psi_2^\flat(2^{-j} D_x, \psi_1]$ involves only smooth commutator functions. Thus we can push the expansion in Lemma \ref{lem:composition} up to an arbitrarily high order. In particular, the $L^2 \to L^2$ norm of the remainder in Lemma \ref{lem:composition} is smaller than $2^{-j(s-1-d/2)}$ if the expansion is pushed to a sufficiently high order. Since $u \in H^s,$ the operator $[\psi_2(2^{-j} D_x), A_1^{j+k}],$ where $A_1^{j+k}$ depends on $(u,\d_x u),$ is $L^2 \to L^2$ bounded. Thus it remains to handle the first terms in the expansion of the commutator $[\psi_2^\flat(2^{-j} D_x, \psi_1].$ Those terms have the form $(\d_x^\a \psi_1) (\d_\xi^\a \psi_2^\flat)(2^{-j} D_x),$ and we can now apply to the corresponding terms in $\Gamma_{k,1}$ the arguments that we used for $\Gamma_{k,2}.$
 
 We turn to $\Gamma_{k,3}.$ First we observe that the supports of $\bar \phi_k$ and $\psi_2$ are disjoint unless $k = 1.$ Thus only $\Gamma_{1,3}$ remains. We have
 $$ \Gamma_{1,3} u = \psi_2^\flat(2^{-j} D_x) (A_1^{j + 1} f) = (\psi_2^\flat(2^{-j} D_x) A_1^{j+1}) f + \tilde \Gamma_{1,3},$$ with notation $f :=  [\bar \phi_1(2^{-j} D_x), \psi_1] \psi_2(2^{-j} D_x) \tilde u,$ and 
 $$ \tilde \Gamma_{1,3} u = \psi_2^\flat(2^{-j} D_x) (A_1^{j+1} f) - (\psi_2^\flat(2^{-j} D_x) A_1^{j+1}) f.$$
We bound 
$$ \| (\psi_2^\flat(2^{-j} D_x) A_1^{j+1}) f \|_{L^2} \leq \| \psi_2^\flat((2^{-j} D_x) A_1^{j+1} \|_{L^2} \| f \|_{L^\infty}.$$
Since $\bar \phi_1(2^{-j} D_x)$ and $\psi_1$ are $L^\infty \to L^\infty$ bounded, uniformly in $j,$ we have
$$ \| f \|_{L^\infty} \lesssim \|\psi_2(2^{-j} D_x) \tilde u\|_{L^\infty},$$
and by Bernstein's inequality (or the Sobolev embedding \eqref{embed})
$$ \| \psi_2(2^{-j} D_x) u\|_{L^\infty} \lesssim 2^{j d/2} \|\psi_2(2^{-j} D_x) \tilde u\|_{L^2} \lesssim 2^{-j(s - d/2)},$$
the last inequality by the posited regularity of $\tilde u.$ Since $A_1 \in H^{s-1},$ we have
$$  \| \psi_2^\flat(2^{-j} D_x) A_1^{j+1} \|_{L^2} \lesssim 2^{-j(s-1)}.$$
Thus we proved  
$$ \| (\psi_2^\flat(2^{-j} D_x) A_1^{j+1}) f \|_{L^2} \lesssim 2^{-j(2s - 1-d/2)},$$
and we turn to $\tilde \Gamma_{1,3} \tilde u:$
$$ \| \tilde \Gamma_{1,3} \tilde u \|_{L^2} = \left\| \int_{\R^d} (\psi_2^\flat(2^{-j}(\xi) - \psi_2^\flat(2^{-j}(\xi - \xi'))) {\mathcal F}(A_1^{j+1})(\xi - \xi') \hat f(\xi') \, d\xi' \right\|_{L^2}.$$
We Taylor expand the frequency cut-off:
  $$ \begin{aligned} \psi_2^\flat(2^{-j}\xi) - \psi_2^\flat(2^{-j}(\xi - \xi')) & = \sum_{1 \leq |\a| \leq m} (\d_\xi^\a \psi_2^\flat)(2^{-j}(\xi - \xi')) \cdot (- 2^{-j} \xi')^\a \\ & + 2^{-j(m + 1)} R(\psi_2^\flat),\end{aligned}$$
  for any $m \in \N,$ where $R(\psi_2^\flat)$ is the remainder,   
  using the same notational convention as in the proof of Lemma \ref{lem:composition}. The term of index $\a$ (with $|\a| > 0$) in the above expansion has a contribution to the $L^2$ norm of $\tilde \Gamma_{1,3} u$ that is controlled by 
  $$ \| (\d_\xi^\a \psi_2^\flat)(2^{-j} D_x) A_1^{1 + j} \|_{L^2} \int_{\R^d} | 2^{-j} \xi|^{|\a|} |\hat f(\xi)| \, d\xi.$$
  Since $|\a| > 0,$ we have, by regularity of $A_1,$
  $$ \| (\d_\xi^\a \psi_2^\flat)(2^{-j} D_x) A_1^{1 + j} \|_{L^2} \lesssim 2^{-j(s-1)}.$$
Besides,
 $$ f = \bar \phi_1(2^{-j} D_x)(\psi_1 \psi_2(2^{-j} D_x) \tilde u) + \psi_1 \bar \phi_1(2^{-j} D_x) \psi_2(D_x) \tilde u) =: f_1  + f_2.$$
 Since $\xi \to |\xi|^{|\a|} \bar \phi_1(\xi)$ is bounded,
 $$ \| (2^{-j} |\xi|)^{|\a|} \hat f_1 \|_{L^1} \lesssim \| \hat \psi_1 \star (\psi_2(2^{-j} \cdot ) {\mathcal F}(\tilde u))\|_{L^1} \lesssim 2^{-j s},$$
 since $\hat \psi_1$ belongs to the Schwartz class and $\tilde u \in H^s.$ Similarly,
 $$  \| (2^{-j} |\xi|)^{|\a|} \hat f_2 \|_{L^1} \lesssim \int_{\R^d} |\xi - \xi'|^{|\a|} |\hat \psi_1(\xi - \xi')| |2^{-j} \xi'|^{|\a|} |\bar \phi_1(2^{-j} \xi') \psi_2(2^{-j} \xi') {\mathcal F}(\tilde u)(\xi')| \, d\xi',$$
 so that
 $$  \| (2^{-j} |\xi|)^{|\a|} \hat f_2 \|_{L^1} \lesssim 2^{-j s}.$$
We finally bound the contribution to $\tilde \Gamma_{1,3}$ of the remainder term in the Taylor expansion of $\psi_2^\flat.$ This is a term of the form
$$ \left\| \int_{\R^d} \int_0^1 \tilde \psi_2(\xi - \xi', t \xi') {\mathcal F}(A_1^{j+1})(\xi - \xi') (2^{-j} t \xi')^\a \hat f(\xi') \, d\xi' \, dt \right\|_{L^2},$$
up to a polynomial in $t,$ with $\tilde \psi_2 \in L^\infty.$ This term has the same form as the previous terms in the expansion of $\psi_2^\flat,$ and we may conclude that the $L^2$ norm of $\tilde \Gamma_{1,3}$ is controlled by $2^{-j(2s-1)}.$ 

In a last step, we prove \eqref{to:do}, where $Q_\star = h_j^{-1} \pdo\Big( \phi_0(2^{-1 + N_0} D_x)(h_j( \psi_1^\flat P_\star^\e))\Big) h_j.$ We use the first bound in \eqref{action}:
$$ \| Q_\star \|_{L^2 \to L^2} \lesssim \sup_{|\a|, |\b| \leq 1 + [d/2]} \| \d_x^\a \d_\xi^\b  (\phi_0(2^{-1 + N_0} D_x)(h_j( \psi_1^\flat P_\star^\e)\|_{L^\infty}.$$
Denote $\check \phi$ the inverse Fourier transform of $\phi_0(2^{-1 + N_0} \cdot),$ so that $\check \phi$ belongs to the Schwartz space. We have
$$  \| \d_x^\a \d_\xi^\b  (\phi_0(2^{-1 + N_0} D_x)(h_j( \psi_1^\flat P_\star^\e)\|_{L^\infty} \lesssim \| \d_x^\a \check \phi_0 \|_{L^1} \| \psi_1^\flat \d_\xi^\b P_\star^\e \|_{L^\infty}.$$ 
By regularity of $P_\star^\e,$ we have $\| \d_\xi^\b P_\star^\e \|_{L^\infty} \lesssim C(\| \tilde u,\d_x \tilde u \|_{L^\infty}) < \infty,$ since $(\tilde u,\d_x \tilde u) \in H^{s-1} \hookrightarrow L^\infty,$ which concludes the proof. 
\end{proof}

\begin{cor} \label{cor:out} On a time interval of existence of the posited solution $\tilde u,$ we have the bound $$\dsp{\| G^{out}_{\star k}(t) \|_{L^2} \lesssim 2^{-j (2 s - 1 - d/2)} e^{t \g_E}.}$$
\end{cor}

\begin{proof} First we bound $G_\star^{out}.$ We saw in the proof of Lemma \ref{lem:Gstar} that $\op_j(\psi^\flat P_\star^\e)$ is $L^2 \to L^2$ bounded, uniformly in $j.$ Thus with  \eqref{R:est} (where we note that $2^j \d_x u_h = \d_x \tilde u$) and \eqref{def:tildeR}, we have
$$ \| \op_j(\psi^\flat P_\star^\e) \tilde R^{para}_j \|_{L^2} \lesssim 2^{-j (2s - 1 - d/2)}.$$
The term $\op_j(\tilde A) \big(\op_j(\psi^\sharp) \op_j(\psi) - \op_j(\psi)\big) \tilde u$ remains. Here we can use the composition result \eqref{composition:para}-\eqref{remainder} up to an arbitrarily high order, since the cut-offs $\psi^\sharp$ and $\psi$ are smooth. In view of $\psi \prec \psi^\sharp,$ this gives
$$ \op_j(\psi^\sharp) \op_j(\psi) - \op_j(\psi) = 2^{-j k} R_k(\psi^\sharp,\psi),$$
where $R_k(\psi^\sharp,\psi)$ maps $L^2$ to $H^{-k}_j,$ with a norm that does not depend on $j.$  
Thus
$$ \| \op_j(\psi^\flat P_\star^\e)  \op_j(\tilde A) \big(\op_j(\psi^\sharp) \op_j(\psi) - \op_j(\psi)\big) \tilde u \|_{L^2} \lesssim 2^{-j k},$$
for any $k > 0.$ With Proposition \ref{lem:out1}, we conclude that
$$ \| G_\star^{out} \|_{L^2} \lesssim 2^{-j (2 s - 1 - d/2)}.$$ 
With Lemma \ref{lem:Sk}, this implies the bound for $G_{\star k}^{out}.$
\end{proof}

  \subsection{Endgame} \label{sec:end1}

We use the backward-in-time representation \eqref{rep:bw:1}, which we reproduce here:
\be \label{rep:bw2} 
\begin{aligned}
 v_E(0) & = \big( v_E(0) - S_E(t;0) v^f_{E k}(t) \big) + S_E(t;0) v_E(t) - S_E(t;0) G^{out}_{E k}(t) \\ & - 2^{-jk \theta} S_E(t;0) \int_{D_k(t)} S_{E k}(T_k;t) v(t_k) dT_k, \qquad \mbox{for any $k \geq 1.$}
\end{aligned}
\ee
On the one hand, we have a lower bound for $\op_j(\tilde \psi)$ applied to the left-hand side of \eqref{rep:bw2}: this is Corollary \ref{cor:datum}. Recall that $\tilde \psi$ is a smooth truncation near $(x^0,\xi^0)$ with an appropriately small support. On the other hand, we have upper bounds for the operator $\op_j(\tilde \psi)$ applied to the right-hand side of \eqref{rep:bw2}.  These show that the left-hand side of \eqref{rep:bw2} (the datum) is much greater than the right-hand side, yielding a contradiction. Indeed:

\smallskip

\noindent $\bullet$ by Corollary \ref{cor:datum},
 $$ \| \op_j(\tilde \psi) v_E(0) \|_{L^2} \geq C 2^{-j \s} (1 + j)^{-1},$$
 for some constant $C > 0;$

\smallskip

\noindent $\bullet$ by Corollary \ref{lem:bw:fw}, for $t = O(j (\g_E - \g_E^-)^{-1}),$  %
 $$ \big\| \op_j(\tilde \psi)( v_E(0) - S_E(t;0) v^f_{E k}(t)) \big\|_{L^2} \lesssim 2^{-j (\s + \theta')} e^{t (\g_E - \g_E^-)}.$$

\smallskip

\noindent $\bullet$ By Lemma \ref{lem:bd:bwS}, 
 $$ \|  \op_j(\tilde \psi)  S_E(t;0) v_E(t) \|_{L^2} \lesssim e^{- t \g_E^-} \| v_E(t)\|_{L^2}.$$
 We saw that $\| \op_j(\psi^\flat P_E^\e) \|_{L^2 \to L^2} \lesssim 1.$ So $\| v_E(t) \|_{L^2} \lesssim \| v (t) \|_{L^2}.$ Now $v = \op_j(\psi) \tilde u,$ and $\psi$ localizes around frequencies $\sim 2^j,$ hence $\| v \|_{L^2} \lesssim 2^{-j s}.$ (Precisely: we could use Proposition \ref{prop:lannes}, and regularity of $\psi,$ to change $\op_j(\psi)$ into $\pdo_j(\psi) = \psi_1(x) \psi_2(2^{-j} D),$ and then use Bernstein's inequality.) Thus 
  $$ \|  \op_j(\tilde \psi)  S_E(t;0) v_E(t) \|_{L^2} \lesssim 2^{-j s} e^{- t \g_E^-}.$$ 

\smallskip

\noindent $\bullet$ By Corollary \ref{cor:out}, and Lemma \ref{lem:bd:bwS}, 
 $$ \| \op_j(\tilde \psi) S_E(t;0) G^{out}_{E k}(t) \|_{L^2} \lesssim 2^{-j (2 s - 1 - d/2)} e^{t(\g_E - \g_E^-)} .$$

\smallskip

\noindent $\bullet$ The fourth and last term in the right-hand side of \eqref{rep:bw2} is smaller: by Lemma \ref{lem:bd:bwS}, Corollary \ref{lem:Sk}, and $\|v (t)\|_{L^2} \lesssim 2^{-j s}$ (see above), we see that  
 $$ \Big\| 2^{-jk} \op_j(\tilde \psi) S_E(t;0) \int_{D_k(t)} S_{E k}(T_k;t) v(t_k) dT_k  \Big\|_{L^2} \lesssim 2^{-j (s  +k)}  e^{t(\g_E - \g_E^-)}.$$  

\smallskip

Putting those bounds together, we find from \eqref{rep:bw2} the inequality
\be \label{final:th1}  \begin{aligned} 2^{-j \s} (1 + j)^{-1}  \leq C \Big( & 2^{-j (\s + \theta' )} e^{t (\g_E - \g_E^-)} + 2^{-j s} e^{-t \g_E^-}  \\ & +  2^{-j (2 s - 1 - d/2)} e^{t(\g_E - \g_E^-)} + 2^{-j (s  +k)}  e^{t(\g_E - \g_E^-)}\Big), \end{aligned}
 \ee 
It is not restrictive to assume that $\s$ is large enough so that $\s + \theta' > 2 s - 1 - d/2.$ (Recall, $\theta' = \theta/(1 + \theta).$) Indeed, the range of $\s$ in the statement of Theorem \ref{th:main} is $\s \in [s, 2 s - 1 -d/2).$ If we manage to prove non-existence of an $H^s$ solution for any value of $\s \in (2 s - 1 - d/2 - \theta', 2 s - 1 - d/2),$ it will prove non-existence of an $H^s$ solution for smaller values of $\s,$ since $H^{\s} \subset H^{\s_1}$ for $\s \geq \s_1.$ Thus among the first, third and fourth terms in the above upper bound have, the third is the largest (since $k$ can be chosen to be arbitrarily large). We obtain, with a different constant $C:$ 
 \be \label{final} 
 2^{-j \s} (1 + j)^{-1}  \leq C \Big( 2^{-j s} e^{-t \g_E^-} + 2^{-j (2 s - 1 - d/2)} e^{t(\g_E - \g_E^-)}\Big).
 \ee 
In \eqref{final}, the constant $C$ does not depend on $j$ nor on $t.$ (It depends on $F,$ the datum, and the radius $\delta$ that describes the size of the support of the space-frequency cut-off around $(x^0,\xi^0).$)

We now choose $\sigma,$ depending on $s$ and $d,$ a final observation time $t_{final},$ then $\delta,$ depending on $\s,$ $s,$ $d,$ and $F,$ and then finally $j$ depending on all other parameters:

\smallskip

\noindent $\bullet$ We choose $\sigma$ (as we may, see the discussion above \eqref{final}) such that
$$ 2 s - 1 - d/2 - \frac{\theta}{\theta + d_\star} < \s < 2 s -1 -d/2,$$
where $\theta = s-1-d/2,$ and $d_\star$ is the $d$-dependent integer that comes up in G\r{a}rding's inequality (see Section \ref{sec:garding}).

\smallskip

\noindent $\bullet$  We let an observation time $t_{final}$ be defined by 
\be \label{def:tfinal}
 t_{final} = \frac{7 j}{8} \cdot \frac{(2 s -  1 - d/2 - \s) \ln 2}{\g_E  - \g_E^-}.\ee
 The ratio $7/8$ is arbitrary; any number in $(0,1)$ would do. By choice of $\s$ above, we observe that this $t_{final}$ is smaller than the time of validity of the growth rates of Section \ref{sec:rates}: $t_{final} < t_\star.$ Also, in the original time frame, this observation time is $2^{-j} t_{final}:$ smaller than $T,$ the posited time of existence of the solution $u,$ if $j$ is chosen large enough (as it will be later).  Thus inequality \eqref{final} holds at $t = t_{final}.$
 
 \smallskip

\noindent $\bullet$ We choose $\delta$ so that
\be \label{delta} e^{-t_{final} \g_E^-} \leq 2^{-j (\s - s + 1/100) }.\ee
For that we only need to choose $\delta$ small enough: indeed, $\g_E^-$ then approaches some value $\g,$ which depends on $F$ and $w,$ and simultaneously $\g_E - \g_E^-$ approaches zero, so that $t_{final}$ becomes large. Thus in \eqref{delta} we are comparing two decaying exponentials in $j,$ one (in the right-hand side) with a fixed decay rate, and the other (in the left-hand side) with a decay rate that we make large in the small $\delta$ limit.  

\smallskip

\noindent $\bullet$  With the above choice for $\s$ and $\delta,$ we consider \eqref{final} at $t = t_{final}.$ This gives
$$ (1 + j)^{-1}  \leq C ( 2^{-j/100} + 2^{-j (2 s - 1 - d/2 -\s)/8}),$$
a contradiction if $j$ is chosen to be large enough, depending on $C$ and on $2s-1-d/2- \s.$

 This proves that, under the ellipticity Assumption \ref{ass:ell}, there are no $C^0([0,T],H^s)$ solutions to \eqref{equation reference} issued from the datum $u_{in} \in H^\s$ (the specific datum defined in Section \ref{sec:datum}), no matter how small $T > 0$ is, and concludes the proof of Theorem \ref{th:main}.

\section{Transition to ellipticity: proof of Theorem \ref{th:2}} \label{sec:weak} 

 As in the proof of Theorem \ref{th:main}, we assume the existence of $T > 0,$ a ball $B_{x_0}$ centered at $x_0$ and $u \in C^0([0,T],H^s(B_{x_0}))$ which solves the initial-value problem \eqref{equation reference}. 
We identify the solution $u$ with its extension to the whole space $\R^d$ via a linear bounded extension operator $H^s(B_{x_0}) \to H^s(\R^d).$ 

 For condition (iv) in Assumption \ref{ass:bif} to make sense, we need the second time derivative of the characteristic polynomial $P$ to be defined pointwise. We saw in \eqref{dtP} that $\d_t P$ involves two derivatives of the datum. The second derivative $\d_t^2 P$ involves three derivatives of the datum, hence the condition $\s > 3 + d/2,$ by Sobolev embedding.

\subsection{The spectral picture}  \label{sec:the:spectral:picture}

We describe here what Assumption \ref{ass:bif} entails for the spectrum of the principal symbol $A$ evaluated along the posited solution $u.$ We denote
\be \label{def:unA}
 \underline A(t,x,\xi) := A(t,x,u(t,x),\d_x u(t,x), \xi).
\ee
The map $\underline A$ is $C^1$ in $t$ and $C^{1,1/2}$ in $x$ (by assumption on $s$ in Theorem \ref{th:2}, Sobolev embedding, and the original system \eqref{equation reference}) and smooth in $\xi \in \S^{d-1}.$

 By condition (i) in Assumption \ref{ass:bif}, the spectrum of $\underline A$ at $t = 0$ is real for all $(x,\xi) \in U.$ %

Given $(x,\xi) \in U,$ consider an eigenvalue $\mu$ of $\underline A$ such that $\d_\l P(0, x,\xi,\mu(0,x,\xi)) \neq 0:$ the eigenvalue $\mu(0,x,\xi)$ of $\underline A(0,x,\xi)$ is simple. As a consequence, it is as smooth as $\underline A,$ pointwise in $(t,x,\xi),$ and it stays real for $(t,x',\xi')$ in a neighborhood of $(0,x,\xi).$ Indeed, the coefficients of $\underline A$ are real, implying that the spectrum is symmetric with respect to $\R:$ eigenvalues away from $\R$ occur in pairs, and by continuity for eigenvalues away from $\R$ to be observed for $t > 0,$ we need eigenvalues to coalesce on the real axis at $t = 0.$ 

Consider now $(x^0,\xi^0) \in U$ given by condition (iii) in Assumption \ref{ass:bif}. For any $\l$ such that $(0,x^0,\xi^0,\l) \in {\mathcal S}_U,$ either $\l$ is a simple eigenvalue, or we have a neighborhood $\Omega_\l$ of $\o = (0,x^0,\xi^0,\l)$ such that the conditions (iv) hold over $\Omega_\l.$ Indeed, the conditions $\d_t P = 0,$ $\d_\l P = 0$ are stated to hold locally, and the inequality in condition (iv) is an open condition. The projection of $\Omega_\l$ over $\R^{2d}_{x,\xi}$ is a neighborhood of $(x^0,\xi^0),$ which we denote $V_\l.$ The finite intersection of the $V_\l$ over all $\l$ in the spectrum at $(0,x^0,\xi^0)$ is a neighborhood of $(x^0,\xi^0)$ which we denote $V.$ 

Given $(x,\xi) \in V$ and an eigenvalue $\mu(0,x,\xi)$ of $\underline A(0,x,\xi),$ let $\o = (0,x,\xi,\mu(0,x,\xi)).$ If we have $\d_\l P(0,\o) = 0,$ then the conditions (iv) in Assumption \ref{ass:bif} hold at $\o.$ The multiplicity of $\mu$ at $(0,x,\xi)$ is exactly two. Thus $\mu$ belongs to a branch (in $(x,\xi)$) of eigenvalues of $\underline A_{|t = 0}$ drawn over $V$ which have constant multiplicity equal to two.

Now consider the time evolution of these branches of eigenvalues, for fixed $(x,\xi) \in V:$

\begin{lem} \label{lem:bif} Condition {\rm (iv)} in Assumption {\rm \ref{ass:bif}} implies that two branches of simple eigenvalues $\mu^\pm$ branch out of the eigenvalue $\mu$ described just above, and out of the real axis at $t = 0,$ such that $\mu^\pm$ are time-differentiable at $t = 0,$ for all $(x,\xi) \in V,$ with 
\be \label{mu:pm}
 \mu^\pm(0,x,\xi) = \mu(0,x,\xi), \quad \d_t \Im m \, \mu^+(0,x,\xi) = - \d_t \Im m \, \mu^-(0,x,\xi) \neq 0.
 \ee
 Besides, under the conditions bearing on $s$ and $\s$ in Theorem {\rm \ref{th:2},} the functions $\mu^\pm$ are as regular as $\underline A,$ and $\d_t \mu^\pm$ is continuous in $(t,x)$ and smooth in $\xi.$ 
\end{lem}

\begin{proof} Given $(x,\xi) \in V$ and an eigenvalue $\mu(0,x,\xi)$ such that $\d_\l P(0,\o) = 0,$ the eigenvalue $\mu(0,x,\xi)$ has multiplicity exactly two, by the inequality in condition (iv) from Assumption \ref{ass:bif}, and is separated from the other eigenvalues. Thus for $t > 0,$ its multiplicity is at most two, and $\mu$ is an eigenvalue for a size-two matrix  $B,$ with real coefficients just like $\underline A,$ the same regularity as $\underline A,$ and with associated characteristic polynomial $P_0.$  The characteristic polynomial $P$ of $\underline A$ factorizes into $P = P_0 P_1,$ with $P_1(0,\o) \neq 0.$ We have 
 $$2  \mu(t,x,\xi) = \tr B(t,x,\xi) \pm  (\det B(t,x,\xi))^{1/2},$$
 with $\det B(0,x,\xi) = 0$ for all $(x,\xi) \in V.$   
 We compute, as in Proposition 1.4 and Appendix A of \cite{LNT}:
 $$ \ba \d_t \det B(0,x,\xi) & = 2 \tr B(0) \d_t \tr B(0) - 4\d_t \det B(0)  
  = -4 \d_t P_0(0,x,\xi) = 0,\ea$$
 and
 $$ \ba \d_t^2 \det B(0,x,\xi) & = 2 (\d_t \tr B(0))^2 + 2 \tr B(0) \d_t^2 \tr B(0) - 4\d_t^2 \det B(0) 
\\ &  =2(\d_{t}\d_{\lambda} P_{0})^{2}-2\d_{\lambda}^{2}P_{0}\d_{t}^{2}P_{0}(0) < 0\ea$$
by the inequality in condition (iv) in Assumption \ref{ass:bif}. We may shrink $V$ if necessary so that $\d_t^2 B(0,x,\xi)$ is bounded away from zero in $V.$ Thus, given $(x,\xi) \in V,$ 
\be \label{eq:mu} \mu(t,x,\xi) = \tr B(t,x,\xi) \pm  t \left( \int_0^1 (1 - s) \d_t^2 \det B(s t, x,\xi) \, ds \right)^{1/2},\ee
where $\d_t^2 B(t,x,\xi) < 0$ is bounded away from zero in $[0,T] \times V,$ if $T > 0$ and $V$ are small enough. %

Thus $\mu$ branches out of the imaginary axis, is differentiable in $t$ at $(0,x,\xi),$ and splits into two eigenvalues $\mu^\pm,$ with $\d_t \Im m \, \mu^\pm(0,x,\xi) \neq 0.$ The functions $(x,\xi) \to \d_t  \mu^\pm(0,x,\xi)$ are continuous in $x$ and smooth in $\xi,$ by condition on $\s:$ since $\s > 3 + d/2,$ the first and second time derivatives of $P,$ hence $P_0,$ at $t = 0,$ are continuous in $x;$ the regularity in $\xi$ comes from the smoothness in $\xi$ of $A$ and \eqref{eq:mu}.  

For $t > 0,$ the eigenvalues $\mu^\pm(t,x,\xi)$ are simple hence as regular as $\underline A.$ By condition $s > 2 + d/2,$ the time-derivative $(\d_t u, \d_t \d_x u)$ is continuous in $x,$ hence for $t > 0,$ $\d_t \mu^\pm$ is continuous in $x;$ besides, it is also smooth in $\xi$ just like $\underline A.$   
\end{proof} 

Thus the spectral picture is as follows: for some $(x^0,\xi^0)$ in $U,$ the eigenvalues of $\underline A(t,x,\xi)$ for $(t,x,\xi)$ near $(0,x^0,\xi^0)$ fall into two categories:
\begin{itemize}
\item real eigenvalues with multiplicity one;
\item and eigenvalues which at $t = 0$ and in a neighborhood $V$ of $(x^0,\xi^0),$ are multiplicity-two and real, and which as $t$ goes from $0$ to $t >0 $ all simultaneously branch out of the real axis into simple eigenvalues with non-zero (and opposite) imaginary parts which are time-differentiable, hence $O(t).$
\end{itemize}

The smooth diagonalizability of $A$ at $t = 0$ is assumed: this is condition (ii) in Assumption \ref{ass:bif}. The discussion above shows that condition (iv) in Assumption \ref{ass:bif} implies furthermore the smooth diagonalizability of $\underline A$ (defined in \eqref{def:unA}) for small positive times. Indeed, the bifurcating eigenvalues are semi-simple at $t = 0$ and split into distinct eigenvalues (see \eqref{mu:pm}) at $t > 0.$ By continuity, these eigenvalues stay distinct for small times. The non-bifurcating eigenvalues are simple at $t = 0,$ hence stay simple for small times, by continuity. The diagonalization matrix is also time-differentiable at $t = 0:$

\begin{lem} \label{lem:evectors} If the space-frequency domain $V$ and $T > 0$ are small enough, then associated with the bifurcation eigenvalues $\mu^\pm$ discussed in Lemma {\rm \ref{lem:bif}} we can find eigenvectors $e^\pm,$ which enjoy the same regularity property in $(t,x,\xi)$ as $\mu^\pm.$ 
\end{lem}

\begin{proof} The symbol $\underline A$ can be smoothly block diagonalized, for all times. Consider $A_0$ a block of size two associated with a pair of bifurcating eigenvalues as in Lemma \ref{lem:bif}. At $t = 0,$ we have $A_0(0,x,\xi) = \mu(0,x,\xi) {\rm Id},$ since $\mu(0,x,\xi)$ is a double and semi-simple eigenvalue of $A_0.$ Thus the matrix $A_0 - \mu {\rm Id}$ vanishes at $t = 0$ and is time-differentiable at $t = 0$ (recall, $\underline A$ is $C^1$ in $t;$ the time regularity of $\mu$ was established in Lemma \ref{lem:bif}), with a time derivative at $t = 0$ which is continuous in $(t,x),$ and smooth in $\xi.$ We may write $A_0 - \mu {\rm Id} = t \tilde A_0(t,x,\xi) =: t (\tilde a_{ij})_{1 \leq i,j \leq 2}.$ If we had $\tilde a_{12}(0,x,\xi) = 0,$ then the eigenvalues of $\tilde A_0$ would stay real. Indeed, those are 
 $$ \mu_0^\pm = \frac{1}{2} \mbox{tr}\tilde A_0 \pm \frac{1}{2} \big( (\tilde a_{11} - \tilde a_{22})^2 - 4 \tilde a_{12} \tilde a_{21} \big)^{1/2}.$$
 Thus $\tilde a_{12}(0,x,\xi) \neq 0,$ and we may assume that $V$ and $T$ are small enough so that $|\tilde a_{12}(t,x,\xi)|$ is bounded away from 0 over $[0,T] \times V.$ Then, the entries $(e_1,e_2)$ of an eigenvector of $\tilde A_0$ take the form
  $$ e_2 = (\mu_0^\pm - \tilde a_{11}) \tilde a_{12}^{-1} e_1,$$
  which implies the announced regularity for the eigenvectors.
\end{proof}

\subsection{Para-linearization and localization}  \label{sec:pf:trans}

 The hyperbolic change of scales from Section \ref{sec:hyp} where we changed $t$ into $2^{-j} t,$ and considered frequencies of size $2^j,$ is not appropriate here. We let instead 
 \be \label{def:tildeu:trans}
 \tilde u(t,x) = u(2^{-j/2} t, x),
 \ee
 and still consider frequencies of size $2^j.$ After para-linearization (see Section \ref{sec:para}), we arrive at 
\be \label{ivp:tildeu:bif} \left\{\begin{aligned}
 \d_t \tilde u + 2^{j/2} \op_j(i A) \tilde u & = - 2^{-j/2} \big( \op_j(\d_3 F) \tilde u + h_j^{-1} R^{para}_j\big),
 \\ 
 \tilde u(0) & = u_{in}, \end{aligned}\right.
 \ee 
 where the principal symbol $A$ is evaluated at $(2^{-j/2} t, x, \tilde u, \d_x \tilde u).$ Note that, in a difference with Section \ref{ivp:tildeu}, the leading operator $\op_j(i A)$ has a large $2^{j/2}$ prefactor. 
 
 As in Section \ref{sec:para}, the para-differential remainder $R^{para}_j$ belongs to $H^{2(s - 1) - d/2},$ uniformly in the small time interval under consideration, and $h_j$ is defined in \eqref{def:hj}. 
 
 Here the $\d_3 F$ term coming from the para-linearization is not seen as a perturbation of the principal symbol, but as a remainder. This is due to the nature of our assumptions: even a small perturbation would here significantly perturb the spectral picture of Section \ref{sec:the:spectral:picture}, since we assume the eigenvalue coalesce. Just like $A,$ in \eqref{ivp:tildeu:bif}, the function $\d_3 F$ is evaluated at $(2^{-j/2} t, x, \tilde u, \d_x \tilde u).$

 Next we use the same localization step as in Section \ref{sec:loc:psi}. We define $v = \op_j(\psi) \tilde u,$ with the same cut-off $\psi.$ This leads to 
\be \label{ivp:v:trans}
\left\{\begin{aligned}
 \d_t v + 2^{j/2} \op_j(M) v & = g, \\ 
 v(0) & = \op_j(\psi) u_{in},
\end{aligned}\right.\ee
where the leading symbol $M$ is defined as 
\be \label{def:M:09} M = i \psi^\sharp A = i \psi^\sharp \underline A(2^{-j/2} t),\ee  (recall that $\underline A$ is defined in \eqref{def:unA}, and note the difference with \eqref{def:M}),
 and the remainder term $g$ is defined as 
\be \label{def:g:bif}
 \ba g & := - [\op_j(\psi), \op_j(i A)] \tilde u - 2^{-j/2}  \op_j(\psi) \big(\op_j(\d_3 F) \tilde u + h_j^{-1} R^{para}_j\big) \\ & + \op_j(i A) \big(\op_j(\psi^\sharp) \op_j(\psi) - \op_j(\psi)\big) \tilde u + (\op_j(i \psi^\sharp A) - \op_j(i A) \op_j(\psi^\sharp)\big) v. \ea \ee 
Compared to the $g$ term from the proof of Theorem \ref{th:main}, we find an extra $\d_3 F$ term here, a different prefactor in front of the para-differential remainder, and $i A$ instead of $\tilde A.$

\subsection{Diagonalization} \label{sec:diag}

We now diagonalize $M:$ by the results of Section \ref{sec:the:spectral:picture}, there is a family of symbols $Q(t)$ of order 0, such that $Q(t,\cdot)$ has the same regularity as $M,$ the symbol $Q$ is differentiable in $t,$ with $\d_t Q(t,\cdot)$ continuous in $(x,\xi),$ and such that $Q M Q^{-1}$ is diagonal. We let 
\be \label{def:w} w = \op_j(\psi^\flat Q) v.\ee
 Note that $Q$ is a priori only defined locally. This is an issue in view of defining $\op_j(Q),$ above, and in the perspective of inverting the change of variable \eqref{def:w}. We can extend the locally defined symbol $Q$ into a globally defined symbol which is order zero and uniformly invertible, so that its inverse is also a symbol of order zero. This is described for instance in Appendix C of \cite{LNT}. In the following we identify $Q$ with its extension. We observe that
 \be \label{Q}
  \op_j(Q^{-1}) \op_j(Q) = {\rm Id} + 2^{-j} R_Q,
 \ee
since by assumption $s - 1 - d/2 > 1$ where $R_Q := R_{1}(Q^{-1},Q),$ with notation from Section \ref{sec:composition}, is $L^2 \to L^2$ bounded, uniformly in $j$ (see \eqref{composition:para}-\eqref{remainder}).
  
 The system in $w$ involves the term
 $$ \ba \op_j(\psi^\flat Q) & \op_j(M) v \\ & = \op_j(\psi^\flat Q) \op_j(M) \Big( \op_j(Q^{-1}) \op_j(Q)-  2^{-j} R_Q \Big) v \\ & = \Big( \op_j(\psi^\flat Q M Q^{-1}) + 2^{-j} R_1(\psi^\flat Q M, Q^{-1}) + 2^{-j} R_1(\psi^\flat Q, M) \op_j(Q^{-1}) \Big) \op_j(Q) v \\ & + 2^{-j} \op_j(\psi^\flat Q) \op_j(M) R_Q v. \ea $$
 We further have
 $$ \ba \op_j(\psi^\flat Q M Q^{-1}) \op_j(Q) & = \Big( \op_j(Q M Q^{-1}) \op_j(\psi^\flat) + 2^{-j} R_{1}(Q M Q^{-1}, \psi^\flat) \Big) \op_j(Q) \\ & = \op_j(Q M Q^{-1}) \op_j(\psi^\flat Q) + 2^{-j} \op_j(Q  MQ^{-1}) R_1(\psi^\flat, Q) \\ & + 2^{-j} R_{1}(Q M Q^{-1}, \psi^\flat) \op_j(Q). \ea $$  
Thus the system in $w$ appears as 
\be \label{sys:w} \d_t w + 2^{j/2} \op_j(Q M Q^{-1}) w = g_Q,\ee 
with
\be \label{def:gQ}
g_Q := \op_j(\psi^\flat Q) g - 2^{-j/2} \op_j(\psi^\flat \d_t Q) v + 2^{-j/2} R(\psi^\flat, M, Q, Q^{-1}) v,
\ee 
where $R(\psi^\flat, M, Q, Q^{-1})$ is the sum of all five remainders terms involving $R_1$ terms that appear above. These terms are exactly the same nature as the $R_H$ terms from \eqref{def:RH}: those are order-one remainders in the composition of an order-one and an order-zero para-differential operator, with $(t,x)$ regularity equal to the one of $(u, \d_x u).$

The new remainder terms (the last two terms in $g_Q$ above) all involve $v.$ Hence they do {\it not} belong to the  ``out'' category, but rather to the $G$ term (see Section \ref{sec:remainder:1}).

Finally, the new $\d_3 F$ term in $g$ is decomposed into an ``in'' and an ``out'' term, as follows:
 $$ \op_j(\psi^\flat Q) \op_j(\psi) \op_j(\d_3 F) \tilde u = \op_j(\psi^\flat Q) \op_j(\d_3 F) v +  \op_j(\psi^\flat Q) [\op_j(\psi^\flat), \op_j(\d_3 F)] \tilde u,$$
 and the first term in the above right-hand side is handled just like the $G v$ terms in the proof of Theorem \ref{th:main} (see Lemma \ref{lem:Gstar}), while the second term is amenable to the analysis of Lemma \ref{lem:out1} (it is actually easier to handle than the term in the proof of Lemma \ref{lem:out1}, for the operator involved is order-zero, and comes with a small $2^{-j/2}$ prefactor).

\subsection{Bounds for the forward-in-time and backward-in-time solution operators} \label{sec:bounds:trans}

Let $\mu$ be one branch of eigenvalues of $\underline A,$ as described in Section \ref{sec:the:spectral:picture}. The operator $\op_j(\mu)$ is $L^2 \to L^2$ bounded, hence has a solution operator $S(t';t),$ for $0 \leq t' \leq t,$ defined as the map which takes a datum $z^0$ at time $t'$ to $z(t),$ where $z(t)$ is the unique solution at time $t$ to the 
initial-value problem
 $$ \d_t z + 2^{j/2} \op_j(i \mu) z = 0, \qquad z(t') = z^0.$$
Note that it is indeed $\op_j(i \mu)$ that we should focus on (and not $\op_j(\mu)$), since $M$ is defined in \eqref{def:M:09} in terms of $i \underline A.$ 

\begin{lem} \label{lem:S:trans} For $0 \leq \t \leq t,$ if the eigenvalue $\mu$ is real, then 
\be \label{bd:imaginary} \| S(\t;t) \|_{L^2 \to L^2} \lesssim 1.
 \ee
If $\mu$ is a bifurcating eigenvalue, then for any $\zeta_-, \zeta_+$ such that 
 $$ \zeta_- <  |\d_t \Im m \, \mu^\pm(0,x^0,\xi^0)| < \zeta_+,$$ 
the following holds true. If the support of $\psi$ is small enough, for $0 \leq \t \leq t:$ 
 \be \label{bd:trans:fwd}  \| S(\t;t) \|_{L^2 \to L^2} \lesssim \exp\big( (t^2 - \t^2)\zeta_+/2 + C (t - \t)2^{-j \theta_\star}\big),\ee
 where $\theta_\star = \theta/(\theta + d_\star),$ with a constant $d_\star$ that depends only on $d.$
 Besides, letting 
 \be \label{tstar:trans}
 t_\star := \left( \frac{2 j (\theta_\star \ln 2 - \e)}{\zeta_+ - \zeta_-}\right)^{1/2}, 
\ee
with $\e > 0$ arbitrarily small,  we have, for $0 \leq \t \leq t \leq t_\star,$ if the support of $\tilde \psi$ is small enough, and if $j$ is large enough depending in $Q,$ $d$ and $\e,$ if if $\mu$ is an unstable bifurcating eigenvalue, meaning $\Im m \, \mu < 0$ for $t > 0,$ then we have the backward-in-time decay bound: 
\be \label{bd:bwd:branch} 
 \| \op(\tilde \psi) \circ S(t;\t) \|_{L^2 \to L^2} \lesssim \exp\big( -(t^2 - \t^2)\zeta_-/2 + C (t - \t)2^{-j \theta_\star}\big).
 \ee 
 Above in \eqref{bd:imaginary} and \eqref{bd:bwd:branch}, the implicit constants do not depend on $t$ nor on $j.$ The positive constant $C$ depends only on $F$ and the space dimension $d,$
\end{lem}

\begin{proof} As described in Section \ref{sec:the:spectral:picture}, under Assumption \ref{ass:bif}, the branch of eigenvalues $\mu$ can either be real for all times, or branch out of the real axis as time becomes positive.

If $\mu$ is real at all time, we perform a simple $L^2$ estimate:
$$ \begin{aligned} 2 \Re e \, \big( \op_j(i \mu) z, z, \big)_{L^2} & = \big( (\op_j(i \mu) + \op_j(i \mu)^\star) z, z \big)_{L^2} \\ & = \big( \op_j(i \mu - i \mu^\star) z, z \big)_{L^2} + O(2^{-j \theta} \| z \|_{L^2}^2), \end{aligned}$$
where $\op_j(\mu)^\star$ denotes the $L^2$ adjoint of $\op_j(\mu),$ and $\mu^\star$ the complex conjugate of $\mu.$ Indeed, we have for adjoint operators a result analogous to the composition result of Section \ref{sec:composition}. With $\mu \in \R$ for all times, this gives
$$ \d_t \| z \|_{L^2}^2 \lesssim 2^{- j (\theta - 1/2)} \| z \|_{L^2}^2,$$
and the bound \eqref{bd:imaginary} then follows from the assumed regularity of the solution.

If $\mu$ is a bifurcating eigenvalue, we use the description of Section \ref{sec:the:spectral:picture}: 
\be \label{mu}
 2^{j/2} \mu(2^{-j/2} t, x, \xi) = i t \mu_1(2^{-j/2} t, x,\xi) + 2^{j/2}\mu_2(2^{-j/2} t, x, \xi), \qquad \mu_1(0,x^0,\xi^0) \neq 0,
\ee
where $\mu_1$ and $\mu_2$ are real. Let $\zeta_\pm$ such that
$$ \zeta_- <  \mu_1(0,x^0,\xi^0) < \zeta_+.$$
 Then, if the support of $\psi$ is small enough, by Lemma \ref{lem:bif},  the spectral bound \eqref{rate:upper} holds with $\g_+(t) = t \zeta_+.$ The upper bound \eqref{bd:imaginary} then follows by Proposition \ref{lem:rates}.  If the support of $\tilde \psi$ is small enough, the spectral bound \eqref{rate:lower} holds with $\gamma_-(t) = t \zeta_-.$ The backward-in-time upper bound \eqref{bd:bwd:branch} then follows from Proposition \ref{lem:rates}.
\end{proof}

\subsection{Conclusion: non-existence of Sobolev solutions in the transition case} \label{sec:end2}

System \eqref{sys:w} is a collection of differential equations that are decoupled to first order. For each of those differential equations, we have solution operators which satisfy the bounds of Lemma \ref{lem:S:trans}. We single out an eigenvalue $\mu = \mu_m$ with $\Re e \, \mu_m < 0$ for $t > 0,$ such that $|\d_t \Im m \, \mu^{\pm}_m(0,x^0,\xi^0)|$ is maximal among the bifurcating eigenvalues. The associated solution operator is $S_m,$ and the associated component of $w$ is $w_m.$ Then, we have a representation formula analogous to \eqref{rep:bw2}:
\be   \label{w:m}
\begin{aligned}
 w_m(0) & = \big( w_m(0) - S_m(t;0) w^f_{m k}(t) \big) + S_m(t;0) w_m(t) - S_m(t;0) G^{out}_{m k}(t) \\ & - 2^{-jk \theta} S_m(t;0) \int_{D_k(t)} S_{m k}(T_k;t) v(t_k) dT_k, \qquad \mbox{for any $k \geq 1.$}
\end{aligned}
\ee
Above, the iterated operator $S_{m k}$ is defined in a way analogous to $S_{E k}$ and $S_{H k}$ in the proof of Theorem \ref{th:main}. That is, we have $S_{m 1} = S_m G_m,$ where $G_m = \op(Q_m) G,$ where $Q_m$ is the projection onto the eigenspace spanned by $\mu_m,$ and parallel to the direct sum of the othe eigenspaces, and  
$S_{m 2} = S_{m 1} \circ \sum_{m'} S_{m' 1},$
the sum running over the finite spectrum of $A.$ Above, $G$ comprises all the remainder terms that are not ``out'', after sorting out the remainder terms just like in Section \ref{sec:remainder:1}. We saw above in the last paragraph of Section \ref{sec:pf:trans} that no ``out'' terms were added in the diagonalization procedure. The ``free'' component of the solution $w^f_{m k}$ and the iterated ``out'' terms $G^{out}_{m k}$ are defined and bounded just like in the proof of Theorem \ref{th:main}, using Lemma \ref{lem:S:trans} instead of the solution operator bounds of the proof of Theorem \ref{th:main}.

  We choose the datum to be exactly like in the proof of Theorem \ref{th:main} (see Section \ref{sec:datum}), with $P_E$ replaced by the projection over the eigenspace spanned by $\mu_m.$ 

Using Lemma \ref{lem:S:trans} and following the proof of Theorem \ref{th:main}, we obtain from \eqref{w:m} the inequality 
\be \label{final:th2}\begin{aligned} 2^{-j \s} (1 + j)^{-1}  \leq C \Big( & 2^{-j (\s + \theta' )} e^{t^2(\zeta_+ - \zeta_-)/2} + 2^{-j s} e^{-t^2 \zeta_-/2}  \\ & +  2^{-j (2 s - 3/2 - d/2)} e^{t^2(\zeta_+ - \zeta_-)/2} + 2^{-j (s  +k)}  e^{t^2(\zeta_+ - \zeta_-)/2}\Big). \end{aligned}
 \ee
 The differences with \eqref{final:th1} are in the rates of growth, now $O(t^2),$ and in the leading term in the remainder, now with a $2 s - 3/2 - d/2$ exponent instead of $2 s - 1 - d/2$ in \eqref{final:th1}. The loss of $2^{j/2}$ comes from the weakness of the defect of hyperbolicity, with destabilizing eigenvalues which are $O(t),$ implying that the relevant time scale is $O(2^{j/2}).$

 From \eqref{final:th2}, we conclude as in Section \ref{sec:end1}. The observation time, at which we record the contradiction, is just a bit smaller than the maximal observation time $t_\star$ \eqref{tstar:trans}. In particular, it is $O(\sqrt j)$ in the rescaled time frame, and $O(\sqrt{ j 2^{-j}})$ in the original time frame, much longer than the instability time $O(j 2^{-j})$ in the elliptic case.

\appendix

\section{Pseudo- and para-differential symbols and operators} \label{sec:symb}  

\subsection{Symbols with limited spatial regularity} \label{sec:symbols:limited}

\label{proof:end}

 Given $k \in \N,$ and $m \in \R,$ the class $C^{k} S^m$ of matrix-valued symbols with limited spatial regularity is defined as follows: a map $a: \R^d \times \R^d \to \R^{N \times N}$ with regularity $C^{k,\theta}$ in its first (spatial) variable and $C^\infty$ in its second (frequency) variable is said to belong to $C^{k,\theta} S^m$ if, for all $\a \in \N^d$ with $|\a| \leq k$ and all $\b \in \N^d,$ there exists $C_{\a\b} > 0$ such that
\be \label{classes} \langle \xi \rangle^{|\b| - m} |\d_x^\a \d_\xi^\b a(x,\xi)| \leq C_{\a\b}, \qquad \mbox{for all $(x,\xi) \in \R^d \times \R^d,$}\ee
where $\langle \xi \rangle := (1 + |\xi|^2)^{1/2}$ and $|\cdot|$ is some norm in the finite-dimensional target space.
Given $k \in \N,$ $\theta \in (0,1)$ and $m \in \R,$ the class $C^{k,\theta} S^m$ of matrix-valued symbols is defined as the space of symbols which are $C^k$ in $x$ and $C^\infty$ in $\xi,$ and such that \eqref{classes} hold, and, for all $|\a| = k,$ 
\be \label{classes:holder}  
\langle \xi \rangle^{|\b| - m} \big|\d_x^\a \d_\xi^\b a(x,\xi) - \d_x^\a \d_\xi^\b a(x',\xi) \big| \leq C_{\a\b} |x - x'|^\theta, \qquad \mbox{for all $(x,x',\xi) \in \R^{2d} \times \R^d,$}\ee

We also encounter the class $C^\infty S^m,$ defined as follows: a map $\s: \R^{N'} \times \R^d \to \R^{N \times N},$ for some $N, N' \in \N,$ $m \in \R,$ is said to belong to $C^\infty S^m$ if it is smooth in both its arguments, and satisfies the bounds
\be \label{classes2}
\langle \xi \rangle^{|\b| - m} |\d_u^\a \d_\xi^\b \s(u,\xi)| \leq C_{\a\b}(|u|), \qquad \mbox{for all $(u,\xi) \in \R^N \times \R^d,$}\ee
for all $(\a,\b) \in \N^{N \times d},$ for some nondecreasing function $C_{\a\b}.$ 

  In our context, symbols with limited spatial regularity $a \in C^{k,\theta} S^m$ appear in the form 
  \be \label{a:s} a(x,\xi) = \sigma(u(x),\xi), \qquad \mbox{where $u \in H^s(\R^d),$ with $s > d/2.$}\ee

\begin{lem} \label{lem:symbols} Given $\s \in C^\infty S^m,$ given $u \in H^s(\R^d)$ with $s > d/2,$ the symbol $a(x,\xi) = \s(u(x),\xi)$ belongs to $C^{k,\theta} S^m,$ with $k = [s - d/2]$ and $\theta = s- d/2 - [s-d/2].$
\end{lem}

\begin{proof} We use the Sobolev embedding into H\"older spaces:  
 \be \label{embed}
  H^s(\R^d) \hookrightarrow C^{[s - d/2], s - d/2 - [s - d/2]}(\R^d).
 \ee
 Given $0 < |\a| \leq [s - d/2],$ by Fa\'a di Bruno's formula, we have the bound
 $$ \langle \xi \rangle^{|\b| - m} |\d_x^\a \d_\xi^\b a(x,\xi)| \lesssim \langle \xi \rangle^{|\b| - m} \sum_{\begin{smallmatrix} 1 \leq |\a'| \leq |\a| \\ \a_1 + \dots + \a_{|\a'|} = \a \end{smallmatrix}} \langle \xi \rangle^{|\b| - m} \d_u^{\a'}  \d_\xi^\b \s(u(x),\xi)  \prod_{1 \leq k \leq |\a'|} |\d_x^{\a_j} u(x)|.$$
 Since every $\a_j$ above has length at most $[s - d/2],$ the maps $\d_x^{\a_j} u$ are bounded, hence the bound \eqref{classes} follows from the assumed bound \eqref{classes2}. 
 
 Given $|\a| = [s-d/2],$ we bound $\langle \xi \rangle^{|\b| - m} |\d_x^\a \d_\xi^\b a(x,\xi) - \d_x^\a \d_\xi^\b a(x',\xi)|$ by the sum of terms of the form 
  \be \label{faa1}
  \langle \xi \rangle^{|\b| - m} |\d_u^{\a'}  \d_\xi^\b \s(u(x),\xi) - \d_u^{\a'} \d_\xi^\b \s(u(x'),\xi)| \prod_{1 \leq k \leq |\a'|} |\d_x^{\a_j} u(x)|,
  \ee
  with $|\a'| \leq [s-d/2]$ and the $\a_j$ summing up to $\a',$ and terms of the form
  \be \label{faa3}
  \langle \xi \rangle^{|\b| - m} |\d_u^{\a'}  \d_\xi^\b \s(u(x'),\xi) | \left( \prod_{1 \leq k \leq |\a'|} |\d_x^{\a_j} u(x)| - \prod_{1 \leq k \leq |\a'|} |\d_x^{\a_j} u(x')|\right).
  \ee
  In \eqref{faa1}, the product involving the $\d_x^{\a_j} u$ is again bounded. We use the Sobolev embedding \eqref{embed}: if $[s - d/2] > 0,$ then $u$ is Lipschitz, otherwise $u \in C^{0,\theta}.$ Since $\s$ is Lipschitz in $u,$ and $u$ is bounded, this gives an upper bound for \eqref{faa1} in the form $C |x - x'|^\theta.$ 
  
  In \eqref{faa3}, the factor involving $\s$ is bounded by \eqref{classes2}. For the product of derivatives of $u,$ we use \eqref{embed} again: if $|\a_j| = [s-d/2]$ for some $j,$ then the product in \eqref{faa3} has only one term, and by the Sobolev embedding $\d_x^{\a_j} u \in C^{0,\theta}.$ Otherwise, all indices $\a_j$ have length strictly less than $[s- d/2],$ implying that all factors in the product are actually Lipschitz and bounded. 
\end{proof}

We verify that the assumption $\s \in C^\infty S^m$ implies a version of the classical nonlinear Sobolev estimate:

\begin{lem} \label{est-Sob-nonlin} Let $\s \in C^\infty S^m,$ such that for all $\xi \in \R^d,$ $\s(0,\xi) = 0.$ Then, for all $s \geq 0,$ for all $u \in H^s \cap L^\infty(\R^d),$ 
 $$ \| \s(u,\xi) \|_{H^s} \leq \langle \xi \rangle^m C(\|u\|_{L^\infty}) \|  u \|_{H^s},$$
 for some nondecreasing function $C$ which depends on $\s$ and $s$ but not on $\xi.$ 
\end{lem}

\begin{proof} We review the classical proof of the Sobolev estimate of smooth nonlinear functions of Sobolev maps, and track the dependence in $\xi.$ We will see that it suffices to assume that $\s$ has $[s] + 2$ continuous derivatives with respect to $u.$

Given $s \geq 0,$ $u \in H^s \cap L^\infty,$ we let $U_k = \phi_0(2^{-k} D) u,$ and $u_k = \phi_k(D) u.$ We have 
 \begin{equation} \label{u-U}
  |\d_x^\a u_k|_{L^2} \leq C_\a 2^{k |\a|} |u_k|_{L^2}, \quad |\d_x^\a u_k|_{L^\infty} \leq C_\a 2^{k |\a|} |u|_{L^\infty}, \quad |\d_x^\a U_k|_{L^\infty} \leq C_\a 2^{k |\a|} |u|_{L^\infty}
 \end{equation}
 as a consequence of the classical Bernstein lemma. We also have, by definition of $U_k,$ for $k \geq 1:$  
 $$ \begin{aligned} \big| U_k - u \big|_{L^2}^2 & = \int_{\R^d} |1 - \phi_0(2^{-k} \xi)|^2 |\hat u(\xi)|^2 \leq \int_{|\xi| \geq \e_1 2^k} |\hat u(\xi)|^2 \, d\xi \\ & \leq \e_1^{-2 s} 2^{-2 k s} \int_{\R^d} |\xi|^{2 s} |\hat u(\xi)|^2 \, d\xi \leq \e_1^{-2s} 2^{-2 ks} |u|^2_{H^s}, 
 \end{aligned}$$
 hence, Taylor expanding: 
 $$ \begin{aligned} \big\| \s(U_k, \xi ) - \s(u, \xi) \big\|_{L^2} & \leq \langle \xi \rangle^m \int_0^1 \big| \d_u \s(U_k + t u, \xi ) \big|_{L^\infty} |U_k - u|_{L^2} \, dt \\ & \leq \langle \xi \rangle^m C(|u|_{L^\infty}) 2^{-ks} \|u\|_{H^s},\end{aligned}$$
by assumption on $\s,$ where we used $|U_k + t u|_{L^\infty} \leq C |u|_{L^\infty},$ which follows from \eqref{u-U}.

From the above, we deduce that for fixed $\xi,$ the sequence $(\s(U_k,\xi))$ converges in $L^2.$ Thus the series $(\s(U_k,\xi) - \s(U_{k-1},\xi))$ converges in $L^2,$ for fixed $\xi,$ and we may write
 $$\s(u ,\xi) = \s(u_0,\xi) + \sum_{k \geq 1} m_k u_k,$$ 
 with notation 
 $$m_k = \int_0^1 \d_u \s(U_{k-1} + tu_k ,\xi ) \, dt.$$ 
 By the first inequality in \eqref{u-U},
 $$ |\d_x^\a (m_k u_k)|_{L^2} \leq C_\a \sum_{\a_1 + \a_2 = \a} |\d_x^{\a_1} m_k|_{L^\infty} 2^{k |\a_2|} |u_k|_{L^2}.$$
 By the Fa\'a di Bruno formula, we observe that $|\d_x^{\a_1} m_k|_{L^\infty}$ is bounded from above, up to a multiplicative constant depending only on $|\a_1|,$ by the sum over $(j, (\b_\ell)_{1 \leq \ell \leq j})$ such that $1 \leq j \leq |\a_1|$ and $\sum_\ell \b_\ell = \a_1,$ of the terms
 $$  \int_0^1 \big| \d_u^{1 + j} \s(U_{k-1} + t u_k, \xi) \big|_{L^\infty} \prod_{1 \leq \ell \leq j} |\d_x^{\b_\ell} (U_{k-1} + t u_k)|_{L^\infty}\, dt.$$
 By the last two bounds in \eqref{u-U}, and assumption \eqref{classes2} on $\sigma,$ we then obtain 
  $$ |\d_x^{\a_1} m_k|_{L^\infty} \leq \langle \xi \rangle^m C(|u|_{L^\infty}) 2^{k |\a_1|}.$$
 Since $u \in H^s,$ it follows from the definition of the sequence $(u_k)$ that $|u_k|_{L^2} \lesssim 2^{-ks} \zeta_k,$ with $|\zeta_k|_{\ell_2} \leq C_s |u|_{H^s}$ for some $C_s > 0.$ Summing up, we have 
  $$ |\d_x^\a (m_k u_k)|_{L^2} \leq  2^{k(|\a| - s)} \zeta_k, \qquad |\zeta_k|_{\ell_2} \leq C(|u|_{L^\infty}) |u|_{H^s}, \qquad |\a| \leq [s] + 1.$$
By a classical result of Littlewood-Paley theory (see for instance Proposition 4.1.13 in \cite{Metivierbis}), this implies the result. 
\end{proof}

\subsection{Pseudo- and para-differential operators}

 The pseudo-differential operator $\pdo(a)$ associated with symbol $a$ is defined by 
$$ \pdo(a) u := \int_{\R^d} e^{i x \cdot \xi} a(x,\xi) \hat u(\xi) \, d\xi, \qquad u \in {\mathcal S}(\R^d),$$
where $\hat u$ is the Fourier transform of $u$ and ${\mathcal S}(\R^d)$ is the Schwartz space of smooth and fast decaying maps over $\R^d.$ Given $a \in C^{k} S^m,$ we denote for $k_1 \leq k:$ 
\be \label{norm:symb}
 \| a \|_{m,k_1,k'} := \sup_{\begin{smallmatrix} (x,\xi) \in \R^d \times \R^d  \end{smallmatrix}} \sup_{\begin{smallmatrix}  |\a| \leq k_1 \\ |\b| \leq k' \end{smallmatrix}} \langle \xi \rangle^{|\b| - m} |\d_x^\a \d_\xi^\b a(x,\xi)|. 
\ee
Given $a \in C^{k,\theta} S^m,$ we denote 
\be \label{norm:symb:H}
\| a \|_{m, k + \theta, k'} := \| a \|_{m,k,k'} + \sup_{\begin{smallmatrix} (x,x',\xi) \in \R^{3d} \\ x \neq x' \end{smallmatrix} } \sup_{\begin{smallmatrix}|\a| = k \\ |\b| \leq k' \end{smallmatrix}} \langle \xi \rangle^{|\b| - m} \frac{|\d_x^\a \d_\xi^\b a(x,\xi) - |\d_x^\a \d_\xi^\b a(x,\xi)|}{|x - x'|^\theta}
\ee
We use a Littlewood-Paley decomposition $(\phi_j)_{j \geq 0}$ (see the first paragraph of Section \ref{sec:loc:psi}).   
 An admissible cut-off is $\phi_{adm}: \R^d \times \R^d \to \R$ defined by\footnote{Note the slight abuse of notation in the definition of $\phi_{adm}:$ the maps $\phi_j$ are defined in terms of the cut-off $\phi_0: \R^d \to \R_+,$ but we evaluate $\phi_j$ at $\langle \xi \rangle$ in $\phi_{adm}.$ We would actually need another Littlewood-Paley decomposition $\tilde \phi_j: \R_+ \to \R_+,$ and let $\phi_{adm}(\eta,\xi) = \sum_{j \geq 0} \phi_0(2^{-j + N} \eta) \tilde \phi_j(\langle \xi \rangle).$ It is easy to check that the abuse of notation in \eqref{def:phiadm} is harmless.}    
 \be \label{def:phiadm} \phi_{adm}(\eta,\xi) = \sum_{j \geq 0} \phi_0(2^{-j + N_0} \eta) \phi_k(\langle \xi \rangle),\ee
 for $N_0 \in \N$ with $N_0 \geq 3.$  
 The  admissible cut-off $\phi_{adm}$ satisfies
\be \label{for:LF} 
\phi_{adm}(\eta,\xi) \equiv \left\{ \begin{aligned} 1, & \quad |\eta| \leq 2^{-N_0} \langle \xi \rangle, \\ 0, & \quad |\eta| \geq 2^{-N_0} \langle \xi\rangle.
 \end{aligned}\right.
\ee 
 Given $a \in C^{k,\theta} S^m,$ we call para-differential symbol associated with $a$ the symbol
$$
 a_{\phi_{adm}}(x,\xi) := {\mathcal F}^{-1} \Big( \phi_{adm}(\cdot,\xi) \hat a(\cdot, \xi)\Big) (x) = \Big(\big( {\mathcal F}^{-1} \phi_{adm}(\cdot,\xi) \big) \star a(\cdot,\xi) \Big)(x),
$$
where convolution takes place in the spatial variable $x,$ and the smooth function ${\mathcal F}^{-1} \phi_{adm}(\cdot,\xi)(x)$ is the inverse Fourier transform of $\phi_{adm}$ in its first variable $\eta.$ The pseudo-differential operator
 \be \label{def:para} \pdo(a_{\phi_{adm}}) =: \op(a)\ee 
 is said to be the para-differential operator associated with $a$ in classical quantization.

Thus $\pdo(a)$ refers to the pseudo-differential operator with symbol $a,$ and $\op(a)$ refers to the para-differential operator (with respect to some admissible cut-off) with symbol $a.$ 

 Consider the maps $h_j$ introduced in \eqref{def:hj}:
 $$ (h_j u)(x) := u(2^{-j} x), \qquad \mbox{and define} \qquad H_j := 2^{-jd/2} h_j.$$
 We denote $\pdo_j(a)$ the operator
 $$ \pdo_j(a) u)(x) := \int_{\R^d} e^{i x \cdot \xi} a(x, 2^{-j}\xi) \hat u(\xi) \, d\xi, \qquad a \in C^{k,\theta} S^m, \quad u \in {\mathcal S}(\R^d).$$
 We observe that
 \be \label{def:pdoj} \pdo_j(a) = H_j^{-1} \pdo(h_j a) H_j,\ee 
 and {\it define}
 \be \label{def:opj}
 \op_j(a) := H_j^{-1} \op(h_j a) H_j.
 \ee

Associated with the change of spatial scale $h_j,$ we have weighted Sobolev norms
\be \label{def:weighted:Sobolev:norm} \| u \|_{j,s} := \| \langle 2^{-j} \xi \rangle^{s/2} \hat u(\xi) \|_{L^2(\R^d_\xi)}.\ee
We sometimes denote $H^s_j$ the Sobolev space $H^s$ equipped with the norm $\| \cdot \|_{j,s}.$ 

 Given $a \in C^{k,\theta} S^m:$ if $k \geq 1 + [d/2],$ the pseudo-differential operator $\pdo_j(a)$ maps $L^2$ to $H^{-m};$ if $k \geq 0,$ the para-differential operator $\op_j(a)$ maps $L^2$ to $H^{-m},$ and
 \be \label{action} \| \pdo(a) \|_{L^2 \to H_j^{-m}} \lesssim \| a \|_{m,1 + [d/2], 1 + [d/2]}, \qquad \| \op_j(a) \|_{L^2 \to H_j^{-m}} \lesssim \| a \|_{m,0, 1 + [d/2]},\ee 
 the implicit constant depending only on dimensions. For a proof of the statement on pseudo-differential operators in \eqref{action}, see for instance Theorem 1.1.4 and its proof in \cite{Lerner}; for a proof of the statement on para-differential operators in \eqref{action}, see for instance Theorem 4.3.5 in \cite{Metivierbis}. The extension to semiclassical operators follow from the definitions \eqref{def:pdoj} and \eqref{def:opj}, and the property 
 \be \label{prop:Hj}
 \| H_j u \|_{H^s} = \| u \|_{j,s}.
\ee

\subsection{Regularized symbols}  \label{sec:reg} 

Let $a \in C^{0,\theta} S^m,$ with $m \in \R$ and $0 < \theta < 1.$ We regularize $a$ into $a^\e \in C^\infty S^{m},$ by spatial convolution with a regularizing kernel $k_\e(x) = \e^{-d} k(x/\e),$ with $k \in C^\infty_c(\R^d),$ such that $k \geq 0$ and the integral of $k$ over $\R^d$ is equal to one. For all $f \in C^{0,\theta},$ we have
 $$ \| f - k_\e \star f \|_{L^\infty} \leq C_\theta \e^\theta \| f \|_{C^{0,\theta}}, \qquad C_\theta := \int_{\R^d} |y|^\theta k(y) \, dy.$$
 We apply the above bound pointwise in $\xi$ to $\d_\xi^\a a,$ for any $\a,$ and obtain
 \be \label{prop:h1} \langle \xi \rangle^{m - |\a|} \|\d_\xi^\a (a - a^\e)(\cdot,\xi)\|_{L^\infty} \leq C_{\theta} \e^{\theta} \| a \|_{m,\theta,|\a|},\ee 
where we used notation \eqref{norm:symb:H} for the symbolic norm. Applying the above to all $|\a| \leq 1 + [d/2],$ we deduce by \eqref{action} the bound 
 \be \label{prop:h2} \| \op_j(a - a^\e) \|_{H_j^m \to L^2} \lesssim \e^\theta.
 \ee 
 Spatial derivatives of $k_\e$ bring out powers of $\e^{-1}:$ 
 \be \label{prop:h3} \| \d_x^\b (k_\e \star f )\|_{L^\infty} \leq C_\b \e^{-|\b|} \| f \|_{L^\infty}, \qquad \mbox{for some $C_\b > 0.$}\ee 
 By \eqref{prop:h3}, norms of $a^\e$ are large in $\e:$ 
\be \label{prop:h4}
 \| a^\e \|_{m, k,k'} \lesssim \e^{-k}, \qquad k > 0.
 \ee
For $a$ in the form \eqref{a:s}, a variant on the above regularization procedure is given by  
\be \label{reg:2} a^\e(x,\xi) := \sigma(k_\e \star u, \xi).\ee  
By regularity of $\s$ in $u,$ the above bounds \eqref{prop:h2} and \eqref{prop:h4} hold also in this case.

\subsection{Para-differential approximation to a semiclassical pseudo-differential operator} \label{sec:paradiff:remainder}

For $a$ in the form \eqref{a:s}, we can estimate how small the difference $\op_j(a) - \pdo_j(a)$ is: 

\begin{prop} \label{prop:lannes} Given $a$ in the form \eqref{a:s}, the difference $\pdo_j(a) - \op_j(a)$ maps the space $H^{m + d/2}(\R^d)$ into $L^2(\R^d)$  and we have
\be \label{lannes} 
 \big\| \big( \pdo_j(a) - \op_j(a) \big) v  \big\|_{L^2} \lesssim 2^{-j (s - d/2)} C(|u|_{L^\infty})  \| u \|_{H^s}  \| v \|_{j, m + d/2}
 \ee 
 for all $v \in H^{m + d/2},$ where the weighted Sobolev norm $\| \cdot \|_{j,m + d/2}$ is defined in \eqref{def:weighted:Sobolev:norm}, and $C$ is nondecreasing. If in \eqref{a:s} the Sobolev index $s$ is strictly greater than the spatial dimension $d,$ then we have  for all $v\in H^m:$ 
 \be \label{lannesbis} 
 \big\| \big( \pdo_j(a) - \op_j(a) \big) v  \big\|_{L^2} \lesssim 2^{-j (s - d)} C(|u|_{L^\infty})  \| u \|_{H^s}  \| v \|_{j,m}.
 \ee 
\end{prop}
The semiclassical estimate \eqref{lannes} is found in \cite{em3}, based on the classical version found in \cite{Lannes}. However, both proofs in \cite{Lannes} and \cite{em3} suffer from (relatively minor) algebraic mistakes, which we correct here.

\begin{rem} \label{rem:lannes} The proof of Proposition {\rm \ref{prop:lannes}} tells us a bit more. For any symbol $a$ of order $m,$ with an adequate degree of spatial regularity, the difference $\pdo(a) - \op(a)$ (classical quantization) maps $H^{m + d/2}$ to $H^{s'}$ with norm given by the upper bounds in \eqref{a2} and \eqref{aLF}:
$$ \ba \| \op(a) - \pdo(a) \|_{H^{m + d/2} \to H^{s'}} & \lesssim \sup_{ \begin{smallmatrix} |\b| \leq 2 [d/2] + 2 \\ \xi \in \R^d \end{smallmatrix} } \langle \xi \rangle^{|\b| - m} \| \d_\xi^\b ((1 - \tilde \phi_0(D_x)) a) \|_{H^{s'}} \\ & \quad + \| (1 - \tilde \phi_0(D_x)) \phi_0(2^{- N_0} D_x) a(\cdot,\xi) \|_{m, 1 + [d/2], 1 + [d/2]}. \ea $$
 It also maps $H^{m - m'}$ to $L^2$ with norm given by the upper bounds in \eqref{a2bis} and \eqref{aLF}:
 $$ \ba \| \op(a) - \pdo(a) \|_{H^{m - m'} \to L^2} & \lesssim \sup_{ \begin{smallmatrix} |\b| \leq 2 [d/2] + 2 \\ \xi \in \R^d \end{smallmatrix} } \langle \xi \rangle^{|\b| - m} \| \d_\xi^\b ((1 - \tilde \phi_0(D_x)) a) \|_{H^{d/2 + m'}} \\ & +   \| (1 - \tilde \phi_0(D_x)) \phi_0(2^{- N_0} D_x) a(\cdot,\xi) \|_{m, 1 + [d/2], 1 + [d/2]}.\ea$$ 
 Above $\| \cdot \|_{m,k,k'}$ is our notation for symbolic norms, introduced in \eqref{norm:symb}, and  $\phi_0$ and $\tilde \phi_0$ are low-frequency cut-offs that are identically equal to one in a neighborhood of $0.$ By ``adequate'' degree of regularity, we mean any measure of regularity such that the upper bounds are finite.  
\end{rem}

\subsubsection{Proof of Proposition \ref{prop:lannes}} \label{sec:proof:lannes} We prove here Proposition \ref{prop:lannes}, based on the results of \cite{Lannes}.

 First we compute 
\begin{equation} \label{basic:computation:remainder} \sum_{0 \leq j' \leq j_0} \phi_{j'}  = \phi_0 + \phi_1 + \dots + \phi_{j_0}  = \phi_0(2^{- j_0} \xi), \qquad \mbox{for all $j_0 \in \N.$} 
\end{equation} 
where $(\phi_j)_{ j \geq 0}$ is the Littlewood-Paley decomposition introduced in Section \ref{sec:loc:psi}.
Next we decompose any symbol $a$ into 
$$ a(x,\xi) = \sum_{\begin{smallmatrix} p \geq 0 \\ q \geq 0 \end{smallmatrix} } \phi_p(D_x) a(\cdot, \xi) \phi_q(\xi),$$
and
$$ a(x,\xi) = \sum_{\begin{smallmatrix} |p - q | < N_0 \end{smallmatrix} } \phi_p(D_x) a(\cdot, \xi) \phi_q(\xi) + \sum_{\begin{smallmatrix} |p - q| \geq N_0 \end{smallmatrix} } \phi_p(D_x) a(\cdot, \xi) \phi_q(\xi).$$
The sum over $\{ (p,q) \in \N \times \N, \quad |p-q| \geq N_0\}$ decomposes into a sum over $p \geq N_0$ with $q \leq p - N_0$ and the symmetrical sum:
$$ \begin{aligned} \sum_{\begin{smallmatrix} |p - q| \geq N_0 \end{smallmatrix} } \phi_p(D_x) a(\cdot, \xi) \phi_q(\xi) = \sum_{\begin{smallmatrix} p \geq N_0 \\ q \leq p - N_0 \end{smallmatrix} } \phi_p(D_x) a(\cdot, \xi) \phi_q(\xi)+ \sum_{\begin{smallmatrix} q \geq N_0 \\ p \leq q - N_0 \end{smallmatrix} } \phi_p(D_x) a(\cdot, \xi) \phi_q(\xi).\end{aligned}$$
We use \eqref{basic:computation:remainder} in order to express the second term in the above right-hand side:
$$ \begin{aligned} \sum_{\begin{smallmatrix} q \geq N_0 \\ p \leq q - N_0 \end{smallmatrix} } \phi_p(D_x) a(\cdot, \xi) \phi_q(\xi) & = \sum_{q \geq N_0} \Big(\sum_{0 \leq p \leq q- N_0} \phi_p(D_x) \Big) a (\cdot,\xi) \phi_q(\xi) \\ & = \sum_{q \geq N_0} \phi_0(2^{-q + N_0} D_x) a(\cdot,\xi) \phi_q(\xi).\end{aligned}$$
Thus by definition of the admissible cut-off $\phi_{adm}:$
$$ \sum_{\begin{smallmatrix} q \geq N_0 \\ p \leq q - N_0 \end{smallmatrix} } \phi_p(D_x) a(\cdot, \xi) \phi_q(\xi) = \phi_{adm}(D,\cdot) a(x,\xi) - \sum_{0 \leq q \leq N_0-1} \phi_0(2^{-q + N_0} D_x) a(\cdot,\xi) \phi_q(\xi).$$
  We have symmetrically
 $$ \sum_{\begin{smallmatrix} p \geq N_0 \\ q \leq p - N_0 \end{smallmatrix} } \phi_p(D_x) a(\cdot, \xi) \phi_q(\xi)  = \sum_{p \geq 0} \phi_p(D_x) a(\cdot,\xi) \phi_0(2^{-p + N_0} \xi) - \sum_{0 \leq p \leq N_0-1} \phi_p(D_x) a(\cdot,\xi) \phi_0(2^{-p + N_0} \xi).$$ 
 In the end we obtain the decomposition
\begin{equation} \label{our:decomposition}
a(x,\xi) = \phi_{adm}(D_x,\xi) a(\cdot,\xi)(x) + a_R(x,\xi) + a_2(x,\xi) + a_{LF}(x,\xi),
\end{equation}
with notation
\begin{equation} \label{def:R2LF} 
 \left\{\begin{aligned}
 a_R(x,\xi) & := \sum_{\begin{smallmatrix} |p - q| < N_0 \end{smallmatrix} } \phi_p(D_x) a(\cdot, \xi) \phi_q(\xi), \\
 a_2(x,\xi) & := \sum_{p \geq 0} \phi_p(D_x) a(\cdot,\xi) \phi_0(2^{-p + N_0} \xi), \\
 a_{LF}(x,\xi) & := - \sum_{0 \leq q \leq N_0-1} \phi_0(2^{-q + N_0} D_x) a(\cdot,\xi) \phi_q(\xi) - \sum_{0 \leq p \leq N_0-1} \phi_p(D_x) a(\cdot,\xi) \phi_0(2^{-p + N_0} \xi).
 \end{aligned}\right.\end{equation}
Decomposition \eqref{our:decomposition} corrects decomposition (2.8) in \cite{Lannes}, where the low-frequency term $a_{LF}$ is missing.

By property \eqref{for:LF} of $\phi_{adm},$ and definition of the low-frequency cut-off $\phi_0$ in Section \ref{sec:loc:psi}, we observe that 
$$ \phi_0(2^{N_0} \e_2 \eta)) (1 - \phi_{adm}(\eta,\xi)) = 0, \qquad \mbox{for all $\eta, \xi,$}$$
or, equivalently,
$$ 1 - \phi_{adm}(\eta,\xi) \equiv (1 - \phi_0(2^{N_0} \e_2 \eta))(1 - \phi_{adm}(\eta,\xi)), \qquad \mbox{for all $\eta, \xi.$}$$
In particular, we may write \eqref{our:decomposition} is the form
\be \label{our:dec:2}
a(x,\xi) = \phi_{adm}(D_x,\xi) a(\cdot,\xi)(x) + (1 - \tilde \phi_0(D_x)) \big( a_R(\cdot,\xi) + a_2(\cdot,\xi) + a_{LF}(\cdot,\xi)\big)(x),
\ee
with the notation $\tilde \phi_0(\eta) := \phi_0(2^{N_0} \e_2 \eta).$ 
By definition of para-differential operators \eqref{def:para} and \eqref{our:dec:2}: 
$$ \pdo(a) - \op(a) = \pdo( (1 - \tilde \phi_0(D_x))(a_2 + a_R + a_{LF})).$$  
By Proposition 20 of \cite{Lannes}, for all $s' \leq s:$ 
\be \label{a2}
 \| \pdo((1 - \tilde \phi_0(D_x)) a_2) v \|_{H^{s'}} \lesssim \sup_{ \begin{smallmatrix} |\b| \leq 2 [d/2] + 2 \\ \xi \in \R^d \end{smallmatrix} } \langle \xi \rangle^{|\b| - m} \| \d_\xi^\b ((1 - \tilde \phi_0(D_x)) a) \|_{H^{s'}} \| v \|_{H^{m + d/2}},
\ee
and also, for $d/2 + m' \leq s:$ 
\be \label{a2bis}
 \| \pdo((1 - \tilde \phi_0(D_x)) a_2) v \|_{L^2} \lesssim \sup_{ \begin{smallmatrix} |\b| \leq 2 [d/2] + 2 \\ \xi \in \R^d \end{smallmatrix} } \langle \xi \rangle^{|\b| - m} \| \d_\xi^\b ((1 - \tilde \phi_0(D_x)) a) \|_{H^{d/2 + m'}} \| v \|_{H^{m - m'}}.
\ee

By Proposition 23 of \cite{Lannes}, the $a_R$ term in \eqref{our:dec:2} satisfies the same bounds, \eqref{a2} and \eqref{a2bis}.

The low-frequency term $a_{LF},$ or $(1 - \tilde \phi_0(D_x)) a_{LF}$ in decomposition \eqref{our:dec:2}, is a finite sum of terms of the form $\pdo((1 - \tilde \phi_0(D_x)) \phi_0(2^{-q} D_x) a(\cdot,\xi) \phi_q(\xi)).$ Given $v \in L^2,$ the map $\pdo((1 - \tilde \phi_0(D_x)) \phi_0(2^{-q} D_x) a(\cdot,\xi) \phi_q(\xi)) v$ has compact support in Fourier. In particular, its $H^{s'}$ norm, for $0 \leq s' \leq s,$ is controlled by its $L^2$ norm. 
We may use the classical bound \eqref{action}: 
 $$ \begin{aligned} \| \pdo((1 - \tilde \phi_0(D_x)) & \phi_0(2^{-q} D_x) a(\cdot,\xi) \phi_q(\xi)) v \|_{L^2} \\ & \lesssim \| (1 - \tilde \phi_0(D_x)) \phi_0(2^{-q} D_x) a(\cdot,\xi) \|_{m, 1 + [d/2], 1 + [d/2]} \| \phi_q(D_x) v \|_{H^m}. \end{aligned}$$
For fixed $q$ and any $m,$ we have $\| \phi_q(D_x) v \|_{H^m} \lesssim \| v \|_{L^2},$ for all $v \in L^2.$ Thus
\be \label{aLF} \| \pdo( (1 - \tilde \phi_0(D_x)) a_{LF} ) v \|_{H^{s'}} \lesssim \| (1 - \tilde \phi_0(D_x)) \phi_0(2^{- N_0} D_x) a(\cdot,\xi) \|_{m, 1 + [d/2], 1 + [d/2]} \| v \|_{L^2}.\ee

With the decomposition \eqref{our:dec:2}, estimates \eqref{a2} and \eqref{aLF} bound the action of the difference $\pdo(a) - \op(a)$ on $H^{m + d/2},$ and estimates \eqref{a2bis} and \eqref{aLF} bound the action of $\pdo(a) - \op(a)$ on $H^{m - m'},$ with $m' \leq s - d/2,$ where $s$ measures the spatial regularity of $a.$  We note that these estimates do not depend on the form of $a$ assumed in \eqref{a:s}, and are valid for any symbol $a$ of order $m$ with an adequate degree of spatial regularity. 

 For $a$ in the form \eqref{a:s}, we now put in the semiclassical quantization, and consider  
 \be \label{the:diff} \| (\pdo_j(a) - \op_j(a)) v \|_{L^2} = \| H_j^{-1} \big( \pdo(h_j a) - \op(h_j a)) H_j v \|_{L^2},\ee
 for $v \in H^{m + d/2}(\R^d)$ or $v \in H^m(\R^d).$   
 In view of bounds \eqref{a2} and \eqref{a2bis}, and the expression of $a$ in terms of $\s$ in \eqref{a:s}, we now want a control of $$ \sup_{ \begin{smallmatrix} |\b| \leq 2 [d/2] + 2 \\ \xi \in \R^d \end{smallmatrix} } \langle \xi \rangle^{|\b| - m} \| ((1 - \tilde \phi_0(D_x)) \d_\xi^\b h_j \s(u,\xi)) \|_{H^{s'}}.$$ 
 We have, by definition of $h_j$ \eqref{def:hj}: 
$$ \begin{aligned} \| ((1 - \tilde \phi_0(D_x)) \d_\xi^\b h_j \s(u,\xi)) \|^2_{H^{s'}} \leq  \int_{|\eta| \geq c} 2^{-2 j d} \langle \eta \rangle^{2 s'} |\d_\xi^\b {\mathcal F}\big( \s(u(\cdot),\xi)\big)(2^j \eta)|^2 \, d\eta. \end{aligned}$$   
 where $c = \e_1/(2^{N_0} \e_2).$ The above upper bound is equal to  
 $$ 2^{-j d} \int_{|\eta| \geq 2^j c} \langle \eta \rangle^{2 s'} |\d_\xi^\b {\mathcal F}\big( \s(u(\cdot),\xi)\big)(\eta)|^2 \, d\eta,$$
which we can bound by 
 $$ 2^{2 j s' -j d} \int_{|\eta| \geq 2^j c} |\eta|^{2s} (2^j c)^{-2s} |\d_\xi^\b {\mathcal F}\big( \s(u(\cdot),\xi)\big)(\eta)|^2 \, d\eta \lesssim 2^{-j (2 s - 2 s' - d)} \| \d_\xi^\b \s(u, \xi) \|_{H^s}.$$ 
 We now use Lemma \ref{est-Sob-nonlin} (applied to $\d_\xi^\b \s,$ which satisfies assumption \eqref{classes2} with $m$ replaced by $m - |\b|$), and find that the contribution of $(h_j a)_2$ and $(h_j a)_R$ in the difference \eqref{the:diff} is controlled by 
 $$ 2^{-j (s - d/2)} C(\| u \|_{L^\infty}) \| u \|_{H^s} \| H_j v \|_{H^{m + d/2}}, \quad \mbox{if we use \eqref{a2},}$$
 and by 
 $$ 2^{-j (s - d)}  C(\| u \|_{L^\infty}) \| u \|_{H^s} \| H_j v \|_{H^m}, \quad \mbox{if we use \eqref{a2bis}.}$$
 Finally, we deduce from \eqref{aLF} that the low-frequency term $(h_j a)_{LF}$ contributes 
$$ \langle \xi \rangle^{m - |\b|} \|\d_\xi^\b \d_x^\a (1 - \tilde \phi_0(D_x)) \phi_0(2^{-N} D) a \|_{L^\infty} \lesssim   \langle \xi \rangle^{m - |\b|} \big\| {\mathcal F} \Big(\d_x^\a (1 - \tilde \phi_0(D_x)) \phi_0(2^{-N_0} D) \d_\xi^\b h_j a\Big) \big\|_{L^1}.$$ 
Since $(1 - \tilde \phi_0(D_x)) \phi_0(2^{-N_0} D) \d_\xi^\b h_j a$ is compactly supported in Fourier (with respect to $x$), the $\d_x^\a$ contributes only a constant, and we focus on 
$$  \begin{aligned} \big\| {\mathcal F} \Big((1 - \tilde \phi_0(D_x)) \phi_0(2^{-N_0} D) \d_\xi^\b h_j a\Big) \big\|_{L^1} & \leq \int_{c_1 \leq |\eta| \leq c_2} |{\mathcal F}(\d_\xi^\b h_j a)(\eta)| \, d\eta \\ & = \int_{2^j c_1 \leq |\eta| \leq 2^j c_2} |{\mathcal F}(\d_\xi^\b a)(\eta)| \, d\eta,\end{aligned}$$
for some $0 < c_1 < c_2.$ In the integrand we can write $1 \lesssim 2^{-j s} |\eta|^s,$ and then by H\"older's inequality we bound the above by a constant times  
$$ 2^{j d/2} 2^{- j s} \| \d_\xi^\b a \|_{H^s},$$
and we may now apply Lemma \ref{est-Sob-nonlin} to $\d_\xi^\b a.$  This completes the proof of Proposition \ref{prop:lannes}.

\subsection{Operator composition} \label{sec:composition}

 Given $m_1, m_2 \in \R,$ given $r  > 0$, given $a_1 \in C^r S^{m_1},$ $a_2 \in C^r S^{m_2}$ if $r \in \N,$ or given $a_1 \in C^{[r], r - [r]} S^{m_1},$ $a_2 \in C^{[r], r - [r]} S^{m_2}$ if $r \notin \N,$ we have 
\be \label{composition:para}
 \op_j(a_1) \op_j(a_2) = \sum_{0 \leq k  < r } 2^{-j k} \frac{(-i)^{|\a|}}{\a!} \op_j\left(  \d_\xi^\a a_1 \d_x^\a a_2\right) + 2^{-j r} R_r(a_1,a_2),
\ee 
with a remainder $R_r(a_1,a_2)$
which maps $H^{m_1 + m_2 - r}$ to $L^2,$ with norm 
 \be \label{remainder} \big\| R_r(a_1,a_2) v \big\|_{L^2} \lesssim \Big( \| a_1\|_{m_1,0,m(d,r)} \| a_2 \|_{m_2,r,d} + \| a_1 \|_{m_1,r,m(d,r)} \| a_2 \|_{m_2,0,d}\Big) \| v \|_{j,m_1 + m_2 - r},\ee
 for some $m(d,r)$ depending on $d$ and $r.$ 
For a proof of \eqref{composition:para}-\eqref{remainder}, see for instance Theorem 6.1.4 of \cite{Metivierbis}.

The above fractional-order composition result has a simple proof in the particular case of a Fourier multiplier and a Sobolev map: 

\begin{lem} \label{lem:composition} Given $s > 1 + d/2,$ given a Fourier multiplier $\chi(D_x)$ with $|\cdot|^{s-1-d/2} {\mathcal F}^{-1} \chi \in L^1(\R^d),$ given $f \in H^{s-1}(\R^d),$ we have
$$ \begin{aligned} \Big\| [\chi(D_x), f] v  - \sum_{1 \leq \a \leq [s-1 - d/2]} \frac{(-i)^{|\a|}}{|\a|!} & (\d_x^\a f)(x) (\d_\xi^\a \chi)(D_x) v \Big\|_{L^2} \\ & \lesssim \| f \|_{H^{s-1}} \big \| |\cdot|^{s-1-d/2} {\mathcal F}^{-1}\chi \big\|_{L^1} \| v \|_{L^2},\end{aligned}$$
In particular, under the above assumptions, we have
$$ \begin{aligned} \Big\| [\chi(2^{-j} D_x, f] v  & - \sum_{1 \leq \a \leq [\theta] -1} 2^{-j |\a|} \frac{(-i)^{|\a|}}{|\a|!}  (\d_x^\a f)(x) (\d_\xi^\a \chi)(2^{-j} D_x) v \Big\|_{L^2} \\ & \leq 2^{-j \theta}  \sup_{|\a| = [\theta]} \| \d_x^{\a} f \|_{0,\theta - [\theta]} \big \| |\cdot|^{\theta} {\mathcal F}^{-1}\chi \big\|_{L^1}  \| v \|_{L^2}.\end{aligned}$$
The H\"older norm in the upper bound is finite by the Sobolev embedding \eqref{embed:in:proof}.
\end{lem}
In the statement of Lemma \ref{lem:composition}, we used the convention that the sum over an empty set is equal to zero. This simple composition result is used in the proof of Lemma \ref{lem:out1}.

\begin{proof} For any $v \in L^2:$ 
$$ [\chi(D_x), f] v = \int_{\R^3} \big({\mathcal F}^{-1} \chi\big)(x - x') (f(x') - f(x)) v(x') \ dx',$$
and with a Taylor expansion of $f \in H^{s-1},$ with $s-1 > d/2:$ 
\be \label{compo:in:proof} [\chi(D_x), f] = \sum_{1 \leq \a \leq [s-1 - d/2]-1} \frac{(-i)^{|\a|}}{|\a|!} (\d_x^\a f)(x) (\d_\xi^\a \chi)(D_x) + R_{[s-1-d/2]}(\chi,f).\ee
The remainder is explicitly%
\be \label{remainder:in:proof} \begin{aligned} R_{[\theta]}(\chi,f) v  = \int_{\R^3} & \int_0^1 \frac{(-i(1 - t))^{[\theta]-1}}{([\theta]-1)!} \\ & \quad \cdot \big({\mathcal F}^{-1} \chi\big)(x - x') \sum_{|\a| = [\theta]} \d_x^{\a} f(x + t (x' - x)) \cdot (x' - x)^{(\a)} v(x') \, dx' \, dt'.\end{aligned}\ee
In \eqref{remainder:in:proof}, given $\a = (\a_1,\dots,\a_d)\in \N^d,$ given $y = (y_1, \dots, y_d) \in \R^d,$ we used notation
$$ \d_x^{\a} f(x + t y) \cdot y^{(\a)} = \sum_{1 \leq k \leq d} y_1^{\a_1} \cdots y_d^{\a_d} \d_{x_1}^{\a_1} \cdots \d_{x_d}^{\a_d} f(x + t y).$$ 
By the Sobolev embedding 
 \be \label{embed:in:proof} H^{s-1} \hookrightarrow C^{[\theta], \theta - [\theta]}, \qquad \theta = s - 1 - d/2,\ee
 given $f \in H^{s-1}$ and $|\a| = [\theta],$ the map $\d_x^{\a} f$ belong to $C^{0,\theta - [\theta]}.$ Thus we may write 
 $$ \d_x^{\a} f(x + t (x' - x)) = \d_x^\a f(x) +  f_\a(x,t(x'-x)) |x'-x|^{\theta - [\theta]},$$
 where $f_\a$ is bounded in both its arguments. This implies
\be \label{remainder:in:proof:2}
\begin{aligned} R_{[\theta]}(\chi,f) v = \frac{(-i)^{[\theta]}}{[\theta]!} \sum_{|\a| = m} (\d_x^\a f)(x) (\d_\xi^\a \chi)(D_x) v + \tilde R_{\theta}(\chi,f) v,
\end{aligned}
\ee
with notation 
$$ \begin{aligned} \tilde R_{\theta}(\chi,f) v = \int_{\R^3} \int_0^1 \frac{(-i(1 - t))^{m-1}}{(m-1)!} \big({\mathcal F}^{-1} \chi\big)(x - x') \tilde f(x, t(x' - x)) |x' - x|^{\theta} v(x') \, dx' \, dt',\end{aligned}$$
where $\tilde f$ is bounded in both its arguments, with
$$ \| \tilde f \|_{L^\infty} \leq \sup_{\begin{smallmatrix} |\a| = [\theta] \\ x \neq y \end{smallmatrix} } \frac{ |\d_x^\a f(x)  - \d_x^\a f(y)|}{|x - y|^{\theta - [\theta]}} =: \sup_{|\a| = [\theta]} \| \d_x^{\a} f \|_{0,\theta - [\theta]}.$$
  In particular, given $\chi$ compactly supported and $f \in H^{s-1},$ we find
$$ \| \tilde R_{\theta}(\chi,f) v \|_{L^2} \lesssim \| |\cdot|^{\theta} {\mathcal F}^{-1}\chi\|_{L^1} \sup_{|\a| = [\theta]} \| \d_x^{\a} f \|_{0,\theta - [\theta]} \| v \|_{L^2},$$
where the upper bound is finite by \eqref{embed:in:proof}.

In semiclassical quantization, that is with $\chi(2^{-j} D_x)$ in place of $\chi(D_x),$ it suffices to observe that
$$ \big\| | \cdot |^{\theta} {\mathcal F}^{-1} (\chi(2^{-j} \cdot)) \big\|_{L^1} = 2^{-j\theta} \big\| | \cdot |^{\theta} {\mathcal F}^{-1} (\cdot)) \big\|_{L^1}.$$
 \end{proof}

\subsection{G\r{a}rding's inequality}  \label{sec:garding} 

For $d_\star := 3 + d,$ for all $Q \in C^{d_\star} S^0,$ 
if $\Re e \, Q \geq 0,$ then
\be \label{garding0}
 \Re e \, (\pdo_j(Q) u, u)_{L^2} + 2^{-j} \| Q \|_{0, d_\star,d_\star} \| u \|_{L^2}^2 \geq 0.
\ee
Inequality \eqref{garding0} is G\r{a}rding's classical inequality in a semiclassical setting, as it can be found for instance in Theorem 4.32 in \cite{Zworski}\footnote{Theorem 4.32 in \cite{Zworski} deals with operators in Weyl quantization. We can for instance use Lemma 4.1.5 in \cite{Lerner} in order to bound the difference between the operator in Weyl quantization and the operator in classical quantization. The difference involves spatial derivatives of the symbol, which in semiclassical quantization bring out a factor $2^{-j}.$}. The specific number of derivatives $d_\star = 3 + d$ needed in \eqref{garding0} is found in the proof of Theorem 1.1.26 in \cite{Lerner}. 

The G\r{a}rding inequality \eqref{garding0} extends to para-differential operators, via the analysis of Section \ref{sec:proof:lannes}. Indeed, we observe that $(H_j^{-1})^\star = H_j$ (the adjoint is in the $L^2$ sense), so that 
$$ \Re e \, (\op_j(Q) u, u )_{L^2} = \Re e \, ((\op(h_j Q) - \pdo(h_j Q)) H_j u, H_j u)_{L^2} + \Re e \, (\pdo_j(Q) u, u)_{L^2},$$
and, as stated in Remark \ref{rem:lannes}, we may use bounds \eqref{a2bis} and \eqref{aLF} in order to bound the action of $\op(h_j Q) - \pdo_j(Q)$ as an operator from $L^2$ to $L^2.$ 

Consider the case that $Q$ is compactly supported, jointly in $(x,\xi),$ and the estimates \eqref{a2bis} and \eqref{aLF} bounding the action of $\op_j(h_j Q) - \pdo_j(h_j Q).$ We can use the arguments from the proof of Proposition \ref{prop:lannes} to verify the bound  
$$ \langle \xi \rangle^{|\b|} \| (1 - \tilde \phi_0(D_x)) (h_j \d_\xi^\b Q) \|_{H^{d/2}} \lesssim 2^{-j(s - d)} \| \d_\xi^\b Q \|_{H^s},$$
for $s$ as large as allowed by the regularity of $Q.$ Since $Q$ is compactly supported in $x,$ the above is controlled by $2^{-j(s - d)} \| \d_\xi^\b \d_x^\a Q \|_{L^\infty},$ for $|\a| \leq [s] + 1.$ We see in the last lines of the proof of Proposition \ref{prop:lannes} that the low-frequency term $Q_{LF}$ contributes a similar term. Thus, for $Q$ compactly supported in $(x,\xi),$ we find 
$$ \| (\op_j(h_j Q) - \pdo_j(h_j Q)) v \|_{L^2} \lesssim 2^{-j (s - d)} \| Q \|_{0, [s] + 1, 2 ([d/2] + 1)} \| v \|_{L^2},$$
hence a paradifferential version of G\r{a}rding's inequality as follows:
\be \label{garding:para}
 \begin{aligned} \Re e \, (& \op_j(Q) u,  u )_{L^2} \\ & + \Big( 2^{-j} \| Q\|_{0,d_\star, d_\star} + 2^{-j(s - d)} \| Q\|_{0, [s] + 1, 2([d/2] + 1)} \Big) \| u \|_{L^2}^2 \geq 0, \\ & \mbox{for all $Q \in C^{\max(d_\star,[s] + 1)} S^0$ compactly supported in $(x,\xi),$ and all $u \in L^2,$}
\end{aligned}
 \ee     
with $s > d.$ 

\section{Rates of growth via G\r{a}rding for the flows of operators with symbols having H\"older spatial regularity} \label{sec:rates}

Given $\theta > 0$ and $T > 0,$ consider a continuous map $Q: t \in [0,T] \to Q(t) \in C^{[\theta],\theta - [\theta]} S^0$ (see the definition of classical classes of symbols with limited spatial regularity in Appendix \ref{sec:symb}) with support near $(x^0, \xi^0),$ for some $(x^0, \xi^0) \in \R^d\times \S^{d-1},$ in the following sense:
\be \label{Q:support}
 \mbox{supp}\, Q(t) \subset \big\{ (x,\xi) \in \R^d \times \R^d, \quad |x - x^0| + |\xi - \xi^0| \leq R \}, \quad \mbox{for some $0 < R < 1,$}
\ee
uniformly in $t.$ The family of symmetric matrices $\Re e \, Q(t,\cdot)$ is continuous and compactly supported. We denote $\g_+$ an upper spectral bound:%
\be \label{rate:upper}
\Re e \, \,  Q(t) \leq \g_+(t), \quad \mbox{for some $t \to \g_+(t) \in \R,$ for all $(x,\xi),$}
\ee
and $\g_-(t)$ a spectral lower bound in restriction to a smaller ball near $(x^0,\xi^0):$ %
\be \label{rate:lower}
 \begin{aligned} \g_-(t) \leq \Re e \, Q(t), \quad & \mbox{for some $\g_- \in \R,$} \\ & \mbox{for all $(x,\xi)$ such that $|x - x^0| + |\xi - \xi^0| \leq r,$ for some $0 < r < R.$} \end{aligned}
\ee
In \eqref{rate:upper} and \eqref{rate:lower}, the matrix $\Re e \, Q$ is the symmetric matrix $(Q + Q^*)/2.$ We use notation $\psi,$ like in the main proof, to describe a smooth space-frequency cut-off such that \eqref{rate:lower} holds on the support of $\psi.$ %

Consider the initial-value problem
\be \label{ivp:Q} \left\{\begin{aligned}
 \d_t u = \op_j(Q) u, \\ u(0) = u_0 \in L^2.
\end{aligned}\right.\ee
 Since $Q(t)$ is order 0, the linear operator $\op_j(Q(t))$ maps $L^2$ to $L^2,$ hence \eqref{ivp:Q} has a unique global solution $u$ by the Cauchy-Lipschitz theorem. 

We can prove the following lower and upper rates of growth for the solution $u$ to \eqref{ivp:Q}:

\begin{prop} \label{lem:rates} For the solution $u$ to \eqref{ivp:Q}, for some $C > 0 $ which depends only on $Q$ and the dimension $d:$  
\begin{itemize}
\item We have the upper bound 
\be \label{est:rate:upper} \| u(t)\|_{L^2} \leq  \exp\Big( \int_0^t \big( \g_+(t') + C 2^{-j \theta_\star} \big) \, dt' \, \Big) \| u_0\|_{L^2},
\ee
where $\theta_\star := \theta/(\theta + d_\star),$ with $d_\star$ is a number of derivatives of the symbols in G\r{a}rding's inequality \eqref{garding0}, $\theta > 0$ is the H\"older regularity of $Q,$ and $\g_+$ is introduced in \eqref{rate:upper}.
\item Let $\tilde \psi$ be a space-frequency cut-off such that the lower bound \eqref{rate:lower} holds over the support of $\tilde \psi.$ Then, given $u_0$ such that $\op_j(\tilde \psi) u_0 \neq 0,$ for $t$ such that
\be \label{bd:t}
0 \leq t \leq t_\star, \quad \mbox{with $t_\star$ defined by $\dsp{\int_0^{t_\star} (\g^+ - \g^-)(t') \, dt' = j (\theta_\star \ln 2 - \e)}$} 
\ee
for any $\e > 0,$ where $\theta_\star$ is introduced below \eqref{est:rate:upper} and the lower rate $\g_-$ is introduced in \eqref{rate:lower}, we have the lower bound, 
\be \label{est:rate:lower}
\frac{9}{10} \exp\Big(\int_0^t \big( \g_-(t') - C 2^{-j \theta_\star}) \, dt' \Big) \, \Big) \| \op_j(\tilde \psi) u_0\|_{L^2} \leq \|u(t)\|_{L^2},
\ee
for large enough $j,$ depending only on $Q,$ $d,$ and $\e.$  
\end{itemize}
\end{prop}

\begin{proof} We use notations and results from Section \ref{sec:reg}, as we regularize $Q$ into $Q^\e \in C^\infty S^0,$ by spatial convolution with a smoothing kernel. By \eqref{prop:h1}-\eqref{prop:h4}, we have 
 \be \label{QtildeQ} \| \op_j(Q - Q^\e) \|_{L^2 \to L^2} \lesssim \e^\theta, \qquad \mbox{and} \qquad  \| Q^\e \|_{0, k,k'} \lesssim \e^{-k}, 
 \ee 
 using notation \eqref{norm:symb}. 
 Via \eqref{QtildeQ}, the lower and upper bounds \eqref{rate:lower} and \eqref{rate:upper} for $Q$ translate into lower and upper bounds for $Q^\e.$ We let 
 $$ \g_+^\e = \g_+ + C \e^\theta, \qquad \g_-^\e := \g_- - C \e^\theta,$$
 for some $C > 0,$
 and
 $$v_\pm = \exp\Big( - \int_0^1 \g^\e_\pm(t') \, dt'\Big) u.$$
Then, $v_\pm$ solve
$$ \begin{aligned} \d_t v_- = \op_j( Q^\e -  \g_-^\e) v_- + \op_j(Q - Q^\e) v_- - C \e^\theta v_-, \\ \d_t v_+ + \op_j(\g_+^\e - Q^\e) v_+ = \op_j(Q - Q^\e) v_+ + C \e^\theta v_+.\end{aligned}$$
Since $\Re e \, (\g_+^\e - Q^\e ) \geq 0$ in the whole domain $\R^d_x \times \R^d_\xi,$ we may apply G\r{a}rding's inequality \eqref{garding0}, and obtain 
$$ \frac{1}{2} \d_t \| v_+ \|_{L^2}^2 \leq C( 2^{-j} \e^{- d_\star} + 2^{-j (s - d)} \e^{- ([s] + 1)} + \e^\theta) \|v_+\|_{L^2}^2,$$
for any $s$ and for some constant $C > 0,$ which depends neither on $j$ nor on $\e.$ We choose $\e$ and $s$ so that the errors above are the same order of magnitude:
\be \label{choice:epsilon}
 \e := 2^{- j/(\theta + d_\star)},
\ee 
and $s \geq (1 - 1/(\theta + d_\star))^{-1} (d + (\theta + 1)/(\theta + d_\star)).$ Then all three errors above have size $2^{-j \theta/(\theta + d_\star)},$ up to a multiplicative constant which does not depend on $j.$  
The upper bound \eqref{est:rate:upper} follows.

 The proof of the lower bound \eqref{est:rate:lower} is just a bit more delicate, since the posited lower bound \eqref{rate:lower} on $\Re e \, Q$ is valid only locally, on the support of $\tilde \psi.$ Let $\tilde \psi^\flat$ such that $\tilde \psi^\flat \prec \tilde \psi,$ and  
$$ w := \op_j(\tilde \psi^\flat) v_-.$$ Then, $w$ solves
$$ \d_t w = \op_j(\tilde \psi^\flat) \op_j(Q^\e - \g_-^\e) v_- + \op_j(\tilde \psi^\flat) \op_j(Q - Q^\e) v_-.$$
Using $\tilde \psi^\flat \prec \tilde \psi$ and \eqref{composition:para}, we find  
$$ \op_j(\tilde \psi^\flat) \op_j(Q^\e) = \op_j(\tilde \psi^\flat \tilde \psi) \op_j( Q^\e) = \op_j(\tilde \psi Q^\e) \op_j(\tilde \psi^\flat) + 2^{-j} \op_j(Q^\e_1),$$
where $Q^\e_1 \in S^0$ is bounded by \eqref{remainder}, in particular involves one spatial derivative of $Q^\e,$ so that, with the second bound in \eqref{QtildeQ}, 
$$ \| \op_j(Q^\e_1) \|_{L^2 \to L^2} \lesssim \e^{- 1}.$$
 uniformly in $j.$ Thus the equation in $w$ is  
$$ \d_t w = \op_j(\tilde \psi (Q^\e - \g_-^\e)) w + 2^{-j} \op_j(Q^\e_1) v_- + \op_j(\tilde \psi^\flat) \op_j(Q - Q^\e) v_- + C \e^\theta w.$$
We now apply \eqref{garding:para}: %
$$ \frac{1}{2} \d_t \| w \|_{L^2}^2 + 2^{- j \theta/(\theta + d_\star)} C \| w \|_{L^2}^2 \geq \Re e \, \Big( 2^{-j} \op_j(Q^\e_1) v_- + \op_j(\tilde \psi^\flat) \op_j(Q - Q^\e) v_-, w \Big)_{L^2}.$$  
The $2^{-j \theta/(\theta + d_\star)}$ prefactor in the G\r{a}rding error term is as in the upper bound, above. The upper bound \eqref{est:rate:upper} gives an upper bound for $v_-:$
$$ \| v_-(t) \|_{L^2} \leq \| u_0 \|_{L^2} \exp\Big( \int_0^t \big( \, (\g^\e_+ - \g^\e_-)(t') + C 2^{-j \theta_\star} \big) \, dt' \Big),$$
 hence %
$$ \begin{aligned} \d_t \| w \|_{L^2} & + C 2^{-j \theta_\star} \| w \|_{L^2} \\ & \geq  - C \Big( 2^{-j} \e^{- 1} + \e^\theta \Big) \exp\Big( \int_0^t \big( \, \g^\e_+ - \g^\e_- + C 2^{-j \theta_\star} \big) \, dt' \Big) \| u_0 \|_{L^2}.\end{aligned}$$ 
 The above lower bound implies 
$$ \begin{aligned} \| w(t)\|_{L^2} & \geq e^{- t C 2^{-j \theta_\star}} \| w_0\|_{L^2} \\ & - C \Big( 2^{-j} \e^{- 1} + \e^\theta \Big) \exp\Big( \int_0^t \big(  (\g^\e_+ - \g^\e_-)(t') + C 2^{-j \theta_\star} \big) \,  dt' \Big) \| u_0 \|_{L^2}.\end{aligned}$$
Recall the definition of $\e$ in terms of $j$ in \eqref{choice:epsilon}. We have $2^{-j} \e^{- 1} \leq 2^{- j \theta/(\theta + d_\star)}.$ Thus the above lower bound takes the form
   $$ \begin{aligned} \| w(t)\|_{L^2} & \geq e^{- t C 2^{-j \theta_\star}} \| w_0\|_{L^2} \\ & - C 2^{-j \theta_\star} \exp\Big( \int_0^t \big(  (\g^\e_+ - \g^\e_-)(t') + C 2^{-j \theta_\star} \big) \, dt' \Big) \| u_0 \|_{L^2}.\end{aligned}$$
By assumption on $u_0,$ the ratio $c_0 := \| w(0) \|_{L^2}/\| u_0 \|_{L^2}$ is positive, and we have
$$ \begin{aligned} \| w(t)\|_{L^2} & \geq e^{- t C 2^{-j \theta_\star}} \| w_0\|_{L^2}\Big(1  - \frac{C}{c_0} 2^{-j \theta_\star} \exp\Big(\int_0^t (\g^\e_+ - \g^\e_- + C 2^{-j \theta_\star} \,)\Big) .\end{aligned}$$
Thus the condition is 
$$  \frac{C}{c_0} 2^{-j \theta_\star} \exp\Big(\int_0^t ( \, \g^\e_+ - \g^\e_- + C 2^{-j \theta_\star} ) \, \Big) \leq \frac{1}{10}.$$ 
That is, for some $C > 0$ which depends only on $Q$ and the dimension $d,$ 
$$ C 2^{-j \theta_\star} \exp\Big( \int_0^t ( \g_+ - \g_- + C 2^{-j \theta_\star} )\, \Big) \| u_0\|_{ L^2} \leq \frac{1}{10} \|  \op_j(\psi^\flat) u_0\|_{L^2}.$$
If $t$ is not too large, in the sense of \eqref{bd:t}, and $j$ is large enough, then the above inequality holds, and this implies the lower bound \eqref{est:rate:lower}. 
\label{page:endofapp}
 \end{proof}  

\end{document}